\documentclass[12pt]{article}
\usepackage{authblk} % Adds affiliation

\usepackage[T1]{fontenc}
\usepackage[utf8]{inputenc}
\usepackage[english]{babel}
\usepackage{enumerate,amsmath,amsthm}
\usepackage{amssymb}
\usepackage{bbding}         % various symbols (squares, asterisks, scissors, ...)
\usepackage{bm}             % boldface symbols (\bm)
\usepackage{bbm} % \mathbb{1}
\usepackage{bookman}
\usepackage{comment}%for commentiong out larger portions of text
\usepackage{graphicx}       % embedding of pictures
\usepackage{wrapfig} %% Obrázky v textu
\usepackage{fancyvrb}       % improved verbatim environment
\usepackage{dcolumn}        % improved alignment of table columns
\usepackage{enumitem,etoolbox,hyperref}
\usepackage{tabu}
\usepackage{multirow}
\usepackage{graphicx} %% Obrázky
\usepackage{wrapfig} %% Obrázky v textu
\usepackage[font=small,labelfont=bf]{caption} % smaller font in figure labels
\graphicspath{ {Pictures/} } %% Path to pictures directory
\usepackage[section]{placeins} % FloatBarrier around sections

\usepackage{caption} % For captions in minipages
\usepackage[normalem]{ulem} % striked letters
\usepackage{tikz} % For creating diagrams
\usetikzlibrary{calc} %for coordinate calculations in tikz
\usetikzlibrary{arrows}% in tikz graphs
\usetikzlibrary{arrows.meta}

\usepackage{float}

\usepackage[final]{listings}

\makeatletter
\AtBeginDocument{%
  \expandafter\renewcommand\expandafter\subsection\expandafter{%
    \expandafter\@fb@secFB\subsection
  }%
}
\makeatother

\newcommand{\ru}[1]{\rule{0pt}{#1 em}}%changing 
\newcommand{\tupsingle}{\ru{3}} % change number in ru 

\input{mycommands.sty}

\newlength\titlebox 

\theoremstyle{definition}
\newtheorem{definition}{Definition}%[section]

\theoremstyle{plain}
\newtheorem{theorem}{Theorem}%[section]
\newtheorem*{theorem*}{Theorem}%theorem without numbering
\newtheorem{lemma}[theorem]{Lemma}

\newtheorem{proposition}[theorem]{Proposition}
\newtheorem{corollary}{Corollary}[theorem]

\theoremstyle{remark}
\newtheorem{remark}{Remark}
\newtheorem{example}{Example}

\newlength{\RoundedBoxWidth}
\newsavebox{\GrayRoundedBox}
\newenvironment{GrayBox}[1][\dimexpr\textwidth-4.5ex]%
   {\setlength{\RoundedBoxWidth}{\dimexpr#1}
    \begin{lrbox}{\GrayRoundedBox}
       \begin{minipage}{\RoundedBoxWidth}}%
   {   \end{minipage}
    \end{lrbox}
    \begin{center}
    \begin{tikzpicture}%
       \draw node[draw=black,fill=black!10,rounded corners,%
             inner sep=2ex,text width=\RoundedBoxWidth]%
             {\usebox{\GrayRoundedBox}};
    \end{tikzpicture}
    \end{center}}

\newlength{\WhiteRoundedBoxWidth}
\newsavebox{\WhiteRoundedBox}
   {\setlength{\WhiteRoundedBoxWidth}{\dimexpr#1}
    \begin{lrbox}{\WhiteRoundedBox}
       \begin{minipage}{\WhiteRoundedBoxWidth}}%
   {   \end{minipage}
    \end{lrbox}
    \begin{center}
    \begin{tikzpicture}%
       \draw node[draw=black,rounded corners,%
             inner sep=2ex,text width=\WhiteRoundedBoxWidth]%
             {\usebox{\WhiteRoundedBox}};
    \end{tikzpicture}
    \end{center}}

\counterwithin{table}{section}  % Keep table numbering within sections
\counterwithin{figure}{section}  % Keep figures numbering within sections

\input{amsbibstyle.sty}
\title{On Random Simplex Picking\\Beyond the Blashke Problem}
\author[ ]{Dominik Beck}
\affil[ ]{\small Faculty of Mathematics and Physics, Charles University, Prague}
\affil[ ]{\textit {\href{mailto:beckd@karlin.mff.cuni.cz}{beckd@karlin.mff.cuni.cz}}}
\date{\vspace{-2.3em}}
\date{December 10, 2024}

\begin{document}
\maketitle
\begin{abstract}
New selected values of odd random simplex volumetric moments (moments of the volume of a random simplex picked from a given body) are derived in an exact form in various bodies in dimensions three, four, five, and six. In three dimensions, the well-known Efron’s formula was used by Buchta \& Reitzner and Zinani to deduce the mean volume of a random tetrahedron in a tetrahedron and a cube. However, for higher moments and/or in higher dimensions, the method fails. As it turned out, the same problem is also solvable using the Blashke-Petkantschin formula in Cartesian parametrisation in the form of the Canonical Section Integral (Base-height splitting). In our presentation, we show how to derive the older results mentioned above using our base-height splitting method and also touch on the essential steps of how the method translates to higher dimensions and for higher moments.
\end{abstract}

{\hfill
\begin{minipage}{0.875\textwidth}
\centerline{\scriptsize \textbf{Acknowledgements}}
{\scriptsize The study was supported by the Charles University, project GA UK No. 71224 and by Charles University Research Centre program No. UNCE/24/SCI/022. We would also like to acknowledge the Dual Trimester Program: ``Synergies between modern probability, geometric analysis and stochastic geometry'' organized by the Hausdorff Research in Bonn and the impact it had on this research}
\end{minipage}
\hfill}

\newpage
\tableofcontents

\newpage
\section{Introduction}
\subsection{Definitions}
Let $K_d$ be a compact and convex body in $\mathbb{R}^d$ with $\dim K_d = d$. One family of such bodies are the $d$-simplex $T_d$, $d$-cube $C_d$ or $d$-orthoplex $O_d = \hull(\pm\vect{e}_1,\ldots,\pm\vect{e}_d)$ (the dual of $C_d$). More generally, we write $\solP_d$ for a polytope of dimension $d$ ($d$-polytope). Specifically, $\solP_2$ stands for a \emph{polygon}, $\solP_3$ a \emph{polyhedron}\index{polyhedron} and $\solP_4$ a \emph{polychoron}\index{polychoron}. Let $\mathbb{X} = (\mathbf{X}_0,\mathbf{X}_1,\ldots\mathbf{X}_n)$ be a sample of $(n+1)$ random points $\mathbf{X}_j$, $j=0,\ldots,n$ with $n\geq d$ selected uniformly and independently from the interior of $K_d$ and let $\solH_n(K_d) = \operatorname{conv}(\mathbb{X}) = \operatorname{conv} (\mathbf{X}_0,\ldots,\mathbf{X}_n)$ (or shortly $\solH_n$) be the convex hull of those points. We define $\Delta_n = \operatorname{vol}_d \solH_n(K_d)$ and its corresponding \emph{metric moments}
\begin{equation}
v_n^{(k)}(K_d) = \frac{\mathbb{E}\,\Delta_n^k}{(\operatorname{vol}_d K_d)^k}
\end{equation}
the normalisation factor in the denominator ensures they stay affinely invariant, that is with respect to affine transformations of $K_d$. When $n=d$, we refer to $v_d^{(k)}(K_d)$ as the \emph{volumetric moments} in $K_d$.

\subsection{Known results}
\subsubsection{Even moments}
\begin{wrapfigure}{r}{0.4\textwidth}
    \centering
    \includegraphics[width=0.35\textwidth]{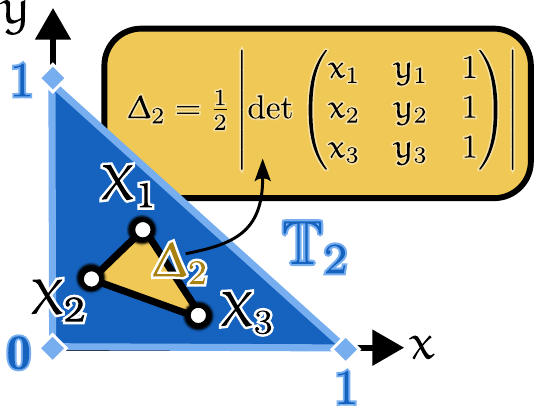}
    \caption{Random triangle area $\Delta_2$ written as a determinant}
    \label{fig:secT3ConI}
\end{wrapfigure}
For even $k$ and any $n \leq d$, volumetric moments $v_n^{(k)}(K_d)$ are trivial to obtain, especially for polytopes. First, note that $\Delta_n$ can be expressed as an absolute value of a determinant of the coordinates of the $n+1$ points forming the vertices of the convex hull $\mathbb{H}_n$ (or as a square root of Gram determinant when $n<d$). Rising this determinant to some (even) power $k$, we obtain some polynomial in coordinates. This is then integrated over the original polytope $P_d$. This is illustrated in Figure \ref{fig:secT3ConI} for the even volumetric moments in the unit triangle $\mathbb{T_2} = \hull([0,0],[1,0],[0,1])$. In general, writing the expectation as an integral, we have for even $k$ and points $\vect{x}_i = (x_i,y_i)^\top, i=1,2,3$,
\begin{equation}
v_2^{(k)}(\mathbb{T}_2) = 2^{k+3} \int_{\mathbb{T}_2^3} \Delta_2^{k} \ddd \vect{x}_0\dd \vect{x}_1 \dd \vect{x}_2.
\end{equation}

For completeness, we enlist in Table \ref{tab:Evenkres} the first three even moments $v_d^{(k)}(P_d)$ for the families of polytopes $T_d$, $C_d$ and $O_d$ upto $d=5$.

\begin{table}[H]
    \centering
\begin{tabular}{|c|c|c|c|}
\hline
 \ru{1.2}$v_d^{(k)}(T_d)$ & $k=2$ & $k=4$ & $k=6$\\
 \hline
 \ru{1.2}$d=1$ & $\frac{1}{6}$ & $\frac{1}{15}$ & $\frac{1}{28}$\\[0.8ex]
 $d=2$ & $\frac{1}{72}$ & $\frac{1}{900}$ & $\frac{403}{2116800}$\\[0.8ex]
 $d=3$ & $\frac{3}{4000}$ & $\frac{871}{123480000}$ & $\frac{2797}{11202105600}$\\[0.8ex]
 $d=4$ & $\frac{1}{33750}$ & $\frac{2083}{96808320000}$ & $\frac{28517}{264649744800000}$ \\[0.8ex]
 $d=5$ & $\frac{5}{5445468}$ & $\frac{24995}{682923373461504}$ & $\frac{11490716929}{618668393733836328960000}$ \\[0.8ex]
 \hline 
\hline
 \ru{1.2}$v_d^{(k)}(C_d)$ & $k=2$ & $k=4$ & $k=6$\\
 \hline
 \ru{1.2}$d=2$ & $\frac{1}{96}$ & $\frac{1}{2400}$ & $\frac{761}{27095040}$\\[0.8ex]
 $d=3$ & $\frac{1}{2592}$ & $\frac{701}{839808000}$ & $\frac{29563}{7466363412480}$\\[0.8ex]
 $d=4$ & $\frac{5}{497664}$ & $\frac{887}{1146617856000}$ & $\frac{6207797}{38533602917272780800}$ \\[0.8ex]
 $d=5$ & $\frac{1}{4976640}$ & $\frac{2899}{7166361600000000}$ & $\frac{3591192719}{1348676
   102104547328000000000}$ \\[0.8ex]
 \hline
 \hline
 \ru{1.2}$v_d^{(k)}(O_d)$ & $k=2$ & $k=4$ & $k=6$\\
 \hline
 \ru{1.2}$d=3$ & $\frac{3}{8000}$ & $\frac{4259}{5268480000}$ & $\frac{7200523}{1835352981504000}$\\[0.8ex]
 $d=4$ & $\frac{1}{108000}$ & $\frac{3959}{5664669696000}$ & $\frac{74002087}{462508951339008000000}$ \\[0.8ex]
 $d=5$ & $\frac{5}{29042496}$ & $\frac{228685}{699313534424580096}$ & $\frac{7261177207}{405955079162673083006928814080000}$ \\[0.8ex]
 \hline
 \end{tabular}
    \caption{Selected values of $v_d^{(k)}(P_d)$ with $P_d=T_d,C_d,O_d$, even $k$ and $d\leq 5$.}
    \label{tab:Evenkres}
\end{table}

\subsubsection{Odd moments}
Extending the work of Crofton, Hostinsk\'{y} \cite[p.~22--26]{hostinsky1925probabilites} considered and solved many problems concerning geometric probability. One of them is the ball tetrahedron picking, which was the first metric moment obtained in $d=3$, it reads
\begin{equation}
    v_3^{(1)}(\mathbb{B}_3) = \frac{9}{715}.
\end{equation}
The result was generalised to higher dimensions by Kingman \cite{kingman1969random}. For the mean volume of a $d$-simplex picked from a $d$-ball, Kingman got
\begin{equation*}
v_d^{(1)}(\solB_d) = \frac{2^{d} \Gamma^2 \left(\frac{(d+1)^2}{2}\right) \Gamma^{d+1} (d+1)}{(d+1)^{d-1}\, \Gamma \left((d+1)^2\right) \Gamma^{2(d+1)}\left(\frac{d+1}{2}\right)}.
\end{equation*}
The result above can be obtained as a special case of even more general formula by Miles \cite[p.~363, Eq. (29)]{miles1971isotropic}, the consequence of which is that the metric moments $v_n^{(k)}(\solB_d)$ are known for any $n$, $k$ and $d$. Particularly,
\begin{equation}
v_d^{(k)}(\solB_d) =\left(\frac{\Gamma \left(\frac{d}{2}+1\right)}{\pi ^{d/2} d!}\right)^{ k} \left(\frac{d}{d+k}\right)^{ d+1} \frac{ \Gamma \left(\frac{(d+1) (d+k)}{2} +1\right)}{\Gamma \left(\frac{d (d+k+1)}{2}+1\right)} \left(\frac{\Gamma\left(\frac{d}{2}\right)}{\Gamma\left(\frac{d+k}{2}\right)}\right)^{\!\!d}\,\, \prod _{l=1}^{d-1} \frac{\Gamma\left(\frac{k+l}{2}\right)}{\Gamma\left(\frac{l}{2}\right)}. 
\end{equation}
Less is known about volumetric moments in polytopes. In two dimensions, one of the classical problems of random geometry is to find the mean convex hull area and its moments, that is to express $v_n^{(k)}(K_2)$ for various $K_2$ and with $n \geq 2$. A lot of results were obtained in this direction. For example, Buchta and Reitzner \cite{buchta1997equiaffine} found a formula expressing $v_n^{(1)}(P_2)$ for any convex polygon $P_2$.

\vspace{1em}
Apart from a ball, not many exact results were known in three dimensions. Here, we are interested in expressing $v_n^{(k)}(K_3)$ with $n \geq 3$, which represents the $k$-th moment of a random volume of a convex hull of $(n+1)$ points. When $n = 3$, the convex hull is almost surely a tetrahedron, so $v_3^{(1)}(K_3)$ represents the \emph{mean tetrahedron volume} and similarly, we have $v_3^{(k)}(K_3)$ for higher moments. The famous difficult problem proposed by Klee \cite{klee1969expected} and popularised by Blaschke concerns finding $v_3^{(1)}(T_3)$, the mean volume of a tetrahedron formed by four uniformly selected random points from the interior of a fixed unit volume tetrahedron. The first attempt was made by Reed. In \cite{reed1974random}, he uses the Crofton reduction technique \cite{ruben1973more} which enables him to express the exact value of $v_3^{(1)}(T_3)=V_{3333}$ as a linear combination of mean volumes of four irreducible configurations $(3320)$, $(2222)$, $(3311)$, $(3221)$, in which the points forming the random tetrahedron are chosen from sets of lower dimensions.
\begin{itemize}
    \item $(3320):$ two points inside, one on a face and the fourth being a vertex,
    \item $(2222):$ points on faces only, one on each face,
    \item $(3311):$ two points inside and two on the opposite edges,
    \item $(3221):$ one point inside, two points on adjacent faces and the fourth being a vertex.
\end{itemize}
Reed was, however, only able to express $V_{3320}=3/64$ and $V_{2222}=1/27$ in a closed form. The last two configurations in the list were only solved by Mannion \cite{mannion1994volume} using a clever handling of improper integrals, giving
\begin{equation*}
    v_3^{(1)}(T_3) = \frac{13}{720} - \frac{\pi^2}{15015} \approx 0.017398.
\end{equation*}
However, Buchta and Reitzner \cite{buchta1992expected} obtained this value earlier using Efron's section formula \cite{efron1965convex}, c.f. \cite[p.~372]{mathai1999introduction}, which relates $v_n^{(1)}(K_3)$ for $n \geq 3$ and $K_3$ with a triple integral over cutting planes. By using Euler's polyhedral formula, Efron obtained
\begin{equation}
 v_n^{(1)}(K_3) = \frac{n}{n+2} - \frac{n(n+1)}{12}\expe{ \Gamma_{\! 3}(\mathbb{X}')^{n-1} + (1-\Gamma_{\! 3}(\mathbb{X}'))^{n-1}},
\end{equation}
where $\Gamma_{\! 3}(\mathbb{X}') = \vol_3 K_3^+/\vol_3 K_3$ is the volume fraction of one of the two parts $K_3^+\sqcup K_3^-$ into which $K_3$ is divided by a plane $\bm{\sigma} = \mathcal{A}(\mathbb{X}')$ going through the collection $\mathbb{X}' = (\vect{X}_1',\vect{X}_2',\vect{X}_3')$ of random points $\vect{X}_j'$, $j\in\{1,2,3\}$ drawn from $K_3$ uniformly and independently. We can transform the expected value above after some nontrivial algebraic manipulations into a set of calculable double integrals. The same technique enabled Zinani \cite{zinani2003expected} to deduce
\begin{equation*}
    v_3^{(1)}(C_3) = \frac{3977}{21600} - \frac{\pi^2}{2160} \approx 0.01384277.
\end{equation*}

The derivation of $v_3^{(1)}(C_3)$ itself is straightforward, but at the same time unworldly difficult, containing millions of intermediate integrals necessary to solve (to do so, Zinani used the package Mathematica 4.0). No other values of odd volumetric moments in three dimensions were known.

\vspace{1em}
In higher dimensions, there were no results for polytopes. Efron's formula completely breaks down because of the existence of cyclic polytopes. At least, for any $K_d$, we have the following relation shown by Buchta in \cite[p.~96]{buchta1986conjecture} by a simple projection argument
\begin{equation}\label{Eq:Buchta}
    v_{d+1}^{(1)}(K_d) = \frac{d+2}{2}\,v_d^{(1)}(K_d)
\end{equation}

\vspace{1em}
However, Efron's formula is not the only approach to volumetric moments. The original method by Reed and Mannion to obtain $v_3^{(1)}(T_3)$ was the Crofton's reduction technique. Another derivation of $v_3^{(1)}(T_3)$ and $v_3^{(1)}(C_3)$ which appeared recently and was not using Efron's formula (but equally difficult) was due to Philip \cite{philip2006tetrahedron,philip2007cube}. As we shall see later in this paper, there is yet another way. Had it not been for Philip's work, the author of this paper would not have been convinced that there might still be another method for obtaining volumetric moments.

\subsection{New results}
\subsubsection{Main theorem (Canonical section integral)}
The objective of this work is to extend the number of polytopes for which the volumetric moments are expressed exactly and to present the method to find it effectively. For this purpose, we developed a new technique which can handle these problems. Our method relied crucially on the following theorem:
\begin{theorem}\label{Thm:Canon}
Let $K_d$ be a $d$-dimensional convex body, $\mathbbm{x}' = (\vect{x}_1,\ldots,\vect{x}_d)$ a collection of $d$ points in $K_d$ and $\bm{\sigma} = \mathcal{A}(\mathbbm{x}') \in \mathbb{A}(d,d-1)$ be a hyperplane parametrised by $\bm{\eta}=(\eta_1\ldots,\eta_d)^\top \in \mathbb{R}^d$ as $\vect{x}\in \bm{\sigma} \Leftrightarrow \bm{\eta}^\top \vect{x} = 1$, then
\begin{equation}
 v_d^{(k)}(K_d) = \frac{(d-1)!}{ d^k}\int_{\mathbb{R}^d\setminus K_d^\circ} v_{d-1}^{(k+1)}(\bm{\sigma}\cap K_d) \,\zeta_d^{d+k+1}(\bm{\sigma}) \iota^{(k)}_d(\bm{\sigma}) \lambda_d(\dd \bm{\eta})
\end{equation}
for any real $k>-1$, where
\begin{equation}
    \zeta_d(\bm{\sigma}) = \frac{\vol_{d-1}(\bm{\sigma}\cap K_d)}{\|\bm{\eta}\|\vol_d K_d}, \qquad \iota^{(k)}_d(\bm{\sigma}) = \int_{K_d} |\bm{\eta}^\top \vect{x}-1|^k \lambda_d(\dd\vect{x}).
\end{equation}
\end{theorem}

Unlike the approach which relies on Efron's formula, our own formula enables us to link odd moments with even moments on sections. However, as mentioned earlier, even moments are simple to compute generally. The theorem itself is proven in Section \ref{sec:Canon}. Let us look what results can be deduced using this formula...

\subsubsection{Three dimensions}
First, we found higher volumetric moments in the tetrahedron, cube, and octahedron. That is $v_3^{(k)}(T_3)$, $v_3^{(k)}(C_3)$ and $v_3^{(k)}(O_3)$. The results are summarized in Table \ref{tab:OddkresCO} below.

The theorem is proven 

\begin{table}[H]
    \centering
\begin{tabular}{|c|c|c|c|}
\hline
 \ru{1.2} & $k=1$ & $k=3$ & $k=5$\\
 \hline
 \ru{1.2}$v_3^{(k)}(T_3)$ & $\frac{13}{720} - \frac{\pi^2}{15015}$ & $\frac{733}{12600000}+\frac{79 \pi ^2}{2424922500}$ & $\frac{5125739}{4356374400000}-\frac{547 \pi ^2}{8943995970000}$\\[0.5ex]
 $v_3^{(k)}(C_3)$ & $\frac{3977}{216000}-\frac{\pi ^2}{2160}$ & $\frac{8411819}{450084600000}-\frac{\pi ^2}{3402000}$ & $\frac{306749173351 \pi ^2}{124439390208000}-\frac{2225580641145943786613}{91479676456923955200000}$\\[0.5ex]
 $v_3^{(k)}(O_3)$ & $\frac{19297 \pi ^2}{3843840}-\frac{6619}{184320}$ & $\frac{1628355709 \pi ^2}{19864965120000}-\frac{81932629}{103219200000}$ & $\frac{6356364544399 \pi ^2}{1611922729697280000}-\frac{205491225433}{5287025049600000}$ \\[0.5ex]
 \hline 
\end{tabular}
    \caption{Selected values of $v_3^{(k)}(T_3)$, $v_3^{(k)}(C_3)$ and $v_3^{(k)}(O_3)$ for odd $k$.}
    \label{tab:OddkresCO}
\end{table}

Next, we considered finding the mean tetrahedron volume $v_3^{(1)}(P_3)$ for various other polyhedra $P_3$ shown in Table \ref{tab:allsolids} (including the case of a tetrahedron and a cube).
% Additional commands control the vertical padding of the following tables:
\renewcommand{\ru}[1]{\rule{0pt}{#1 em}}%changing height in a cell

\begin{table}[H]
    \centering
\begin{GrayBox}
    \centering
    \begin{tabular}{cccc}
    \ru{4.0}
                \includegraphics[width=40pt]{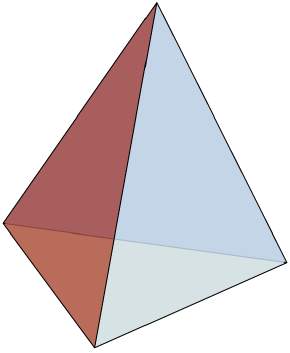} &
                \includegraphics[width=45pt]{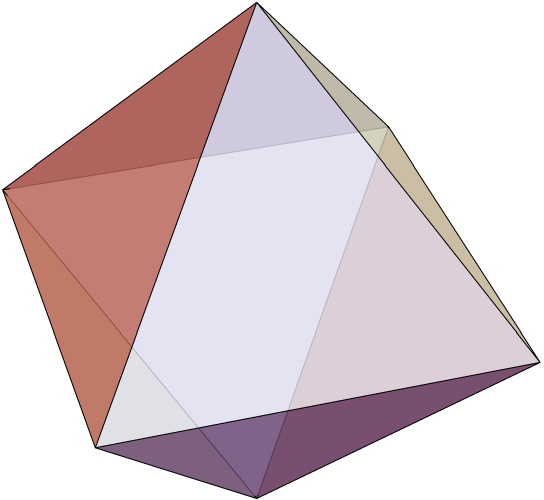} &
                \includegraphics[width=40pt]{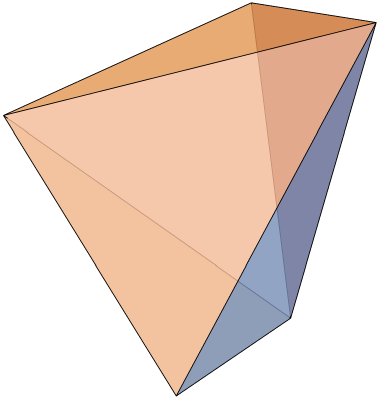}&
                \includegraphics[width=50pt]{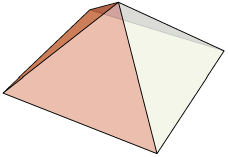} \\
       $T_3$, tetrahedron & $O_3$, octahedron & \begin{tabular}{c} tetrahedron \\ bipyramid\end{tabular} & \begin{tabular}{c} square \\ pyramid\end{tabular} \\[2ex]
    \ru{4.0}
                \includegraphics[width=40pt]{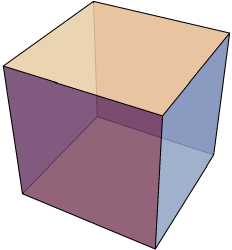} &
                \includegraphics[width=38pt]{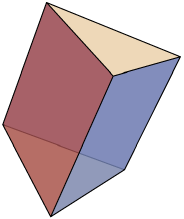}&
               \includegraphics[width=38pt]{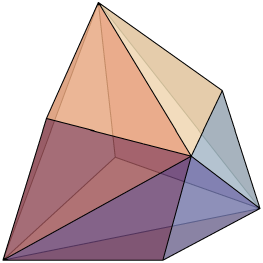}&
              \includegraphics[width=40pt]{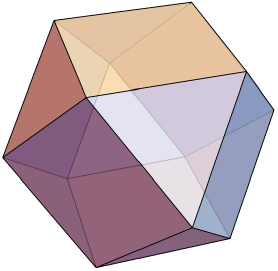}\\
       $C_3$, cube & \begin{tabular}{c} triangular \\ prism \end{tabular} & \begin{tabular}{c} triakis \\ tetrahedron* \end{tabular} & cuboctahedron \\[2ex]
       \ru{4.0}
       \includegraphics[width=40pt]{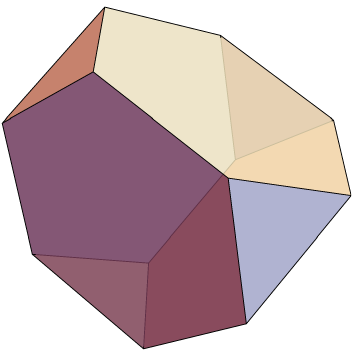}&
       \includegraphics[width=40pt]{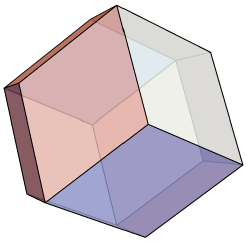}&
       \includegraphics[width=38pt]{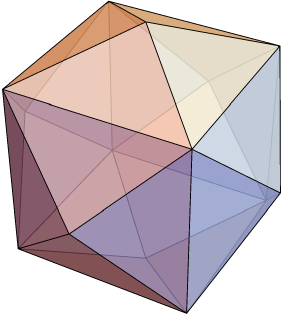}&
       \includegraphics[width=40pt]{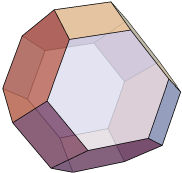}
       \\
    \begin{tabular}{c} truncated \\ tetrahedron \end{tabular} & \begin{tabular}{c} rhombic \\ dodecahedron\end{tabular} & \begin{tabular}{c} tetrakis \\ hexahedron* \end{tabular} & \begin{tabular}{c} truncated \\ octahedron* \end{tabular}\\[2ex]
    \end{tabular}
    \caption{Polyhedra for which we considered $v_3^{(1)}(K_3)$}
    \label{tab:allsolids}
\end{GrayBox}
\end{table}

To be honest with the reader, the polyhedra indicated by ${}^*$ have not been computed yet (section integrals are availible only in some particular genealogies), but they will surely appear in an updated version of this paper. Interestingly, in contrast to the well known tetrahedron and cube case, $v_3^{(1)}(P_3)$ often involves logarithms and special values of the so called \emph{dilogarithm}\index{dilogarithm function} function $\operatorname{Li}_2(x) = \sum_{n=1}^\infty x^2/n^2$, especially
\begin{equation}
    \operatorname{Li}_2\!\left(\tfrac{1}{4}\right) \approx 0.2676526390827326069191838284878115758198570669\ldots
\end{equation}

Table \ref{Tab:AllTetr} below summarises all new results of exact mean tetrahedron volume in various $3$-bodies $K_3$. For completeness, the previously known cases of a ball, tetrahedron and a cube have been added as well. Each $K_3$ is having volume one or alternatively, the right column displays $v_3^{(1)}(K_3)$.

\bgroup
% Additional commands control the vertical padding of the following tables:
\renewcommand{\ru}[1]{\rule{0pt}{#1 em}}%changing height in a cell
\newcommand{\tup}{\ru{2.2}} % change number to increase upper space in table 
\renewcommand{\tupsingle}{\ru{-1}} % the same but with a different value
\newcommand{\tdown}{1.5} % change number to increase down space in table
\newcommand{\tdownsingle}{-0.2} % change number to increase down space in table
\newcommand{\tupmodif}{-0.3} % change number to increase up space in table single rows in the right columns

\renewcommand{\arraystretch}{0.8}
%\subsection*{Mean solid/tetrahedron picking}
%% \textstyle -> incl. smaller fractions
   \begin{table}[H]
\setlength{\tabcolsep}{27pt} % reduces collumn width
\setlength\jot{-5pt} % math expressions smaller space between rows
\begin{minipage}{\textwidth} % must be there otherwise footnotes not showing
    \begin{tabular}{|c|c|}
\hline
$K_3$ &  $v_3^{(1)}(K_3)$\\
\hline
\tupsingle  \!\!\!\!\!\!\!\begin{tabular}{c} \textit{ball}, \cite{hostinsky1925probabilites} \\ $\mathit{0.012587413}$ \end{tabular}\!\!\!\!\!\!\! & $\begin{aligned}
    & \\[\tupmodif em] & \textstyle\frac{9}{715} \\ &   
\end{aligned}$ \\[\tdownsingle em]
    \hline
\tup \!\!\!\!\!\!\!\begin{tabular}{c} rhombic \\ dodecahedron \\ $0.012938482$ \end{tabular}\!\!\!\!\!\!\! & $\begin{aligned}
  & \textstyle\frac{2421179003623}{17933819904000}+\frac{37061863 \pi^2}{29889699840}-\frac{9406373047 \ln 2}{9340531200}\\[1ex]
    & \textstyle-\frac{1757220593 \ln^2 2}{2490808320}+\frac{282589831 \ln 3}{283852800}-\frac{6078271 \operatorname{Li}_2\left(\frac{1}{4}\right)}{8515584}
    \end{aligned}$ \\[\tdown em]
    \hline
\tup \!\!\!\!\!\!\!\begin{tabular}{c} cuboctahedron \\ $0.013002516$ \end{tabular}\!\!\!\!\!\!\! & $\begin{aligned}
    & \textstyle\frac{117410162173}{525525000000}+\frac{8752199 \pi ^2}{2402400000}-\frac{192940695481 \ln 2}{105105000000}\\[1ex]
    & \textstyle-\frac{318759601 \ln^2 2}{250250000}+\frac{506316394917 \ln 3}{280280000000}-\frac{648098487 \operatorname{Li}_2\left(\frac{1}{4}\right)}{500500000}  
\end{aligned}$ \\[\tdown em]
   \hline
\tupsingle  \!\!\!\!\!\!\!\begin{tabular}{c} octahedron \\ $0.013637411$ \end{tabular}\!\!\!\!\!\!\! & $\begin{aligned}
    & \\[\tupmodif em] & \textstyle\frac{19297 \pi^2}{3843840}-\frac{6619}{184320} \\ &   
\end{aligned}$ \\[\tdownsingle em]
\hline
\tupsingle  \!\!\!\!\!\!\!\begin{tabular}{c} \textit{cube}, \cite{zinani2003expected} \\ $\mathit{0.013842776}$ \end{tabular}\!\!\!\!\!\!\! & $\begin{aligned}
    & \\[\tupmodif em] & \textstyle\frac{3977}{216000}-\frac{\pi^2}{2160} \\ &   
\end{aligned}$ \\[\tdownsingle em]
  \hline
\tup \!\!\!\!\!\!\!\begin{tabular}{c} truncated \\ tetrahedron \\ $0.014845102$ \end{tabular}\!\!\!\!\!\!\! & $\!\!\!\!\!\begin{aligned}
& \textstyle \!\!\!\frac{35604506258521}{162358039443600}-\frac{13447020779 \pi ^2}{96641690145}\!+\!\frac{9972537226592 \ln 2}{3382459155075}\!+\!\frac{3485442712 \ln ^2 2}{1400604205}\! \\[1ex]
        & \textstyle\!\! -\!\frac{8953623027 \ln 3}{7884520175}\!-\!\frac{53493528168 \ln 2 \ln 3}{32213896715}\!+\!\frac{53162662164 \operatorname{Li}_2\left(\frac{1}{4}\right)}{32213896715}\!\! \end{aligned}\!\!\!\!\!$ \\[\tdown em]
    \hline
\tup \!\!\!\!\!\!\!\begin{tabular}{c} triangular \\ bipyramid \\ $0.015082427$ \end{tabular}\!\!\!\!\!\!\! & $\!\!\!\!\!\begin{aligned}
  & \textstyle \!\!\frac{1712190037}{16812956160}\!+\!\frac{81471636487 \pi ^2}{907899632640}\!-\!\frac{185777703053 \ln 2}{50438868480}\!-\!\frac{909434448983 \ln ^2 2}{121053284352}\\[1ex]
  & \textstyle \!\!\!+\!\frac{3498264683 \ln
   3}{2401850880}\!+\!\frac{20912895 \ln 2 \ln 3}{2050048}\!-\!\frac{1887867 \ln^2 3}{585728}\!-\!\frac{62045573287 \text{Li}_2\left(\frac{1}{4}\right)}{57644421120}\!\!\!
    \end{aligned}\!\!\!\!\!$ \\[\tdown em]
 \hline
\tupsingle \!\!\!\!\!\!\!\begin{tabular}{c} triangular prism \\ $0.015357705$ \end{tabular}\!\!\!\!\!\!\! & $\begin{aligned}
    & \\[\tupmodif em] & \textstyle\frac{1859}{116640}-\frac{\pi ^2}{17010} \\ &   
\end{aligned}$ \\[\tdownsingle em]
 \hline
\tupsingle \!\!\!\!\!\!\!\begin{tabular}{c} square pyramid \\ $0.015782681$ \end{tabular}\!\!\!\!\!\!\! & $\begin{aligned}
    & \\[\tupmodif em] & \textstyle\frac{941 \pi^2}{72072}-\frac{977}{8640} \\ &   
\end{aligned}$ \\[\tdownsingle em]
 \hline
\tupsingle  \!\!\!\!\!\!\!\begin{tabular}{c} \textit{tetrahedron}, \cite{buchta1992expected} \\ $\mathit{0.017398239}$ \end{tabular}\!\!\!\!\!\!\! & $\begin{aligned}
    & \\[\tupmodif em] & \textstyle\frac{13}{720}-\frac{\pi^2}{15015} \\ &   
\end{aligned}$ \\%[\tdownsingle em]
    \hline
    \end{tabular} 
\end{minipage}
\caption{Mean tetrahedron volume $v_3^{(1)}(K_3)$ in various bodies $K_3$ \label{Tab:AllTetr}}
\end{table}
\egroup

Note that none of the results above can be derived by hand. In principle, it could, but it would require superhuman skills we do not posses. Like Zinani, we used CAS package Mathematica (11.0 -- 14.0) to deduce the intermediate steps (see codes in appendices).

\subsubsection{Higher dimensions}
Also, our another goal is to present a new technique and deduce the values of $v_d^{(k)}(P_d)$ for various odd $k$ and $d=3,4,5$ in the most elementary way (for even $k$, they are trivial). The results for $T_d$ are shown in Table \ref{tab:OddkresT}.
\begin{table}[H]
    \centering
\begin{tabular}{|c|c|}
\hline
 \ru{1.2} & $v_d^{(1)}(T_d)$\\
 \hline
 \ru{1.2}$d=3$ & $\frac{13}{720} - \frac{\pi^2}{15015}$\\[0.5ex]
 $d=4$ & $\frac{97}{27000}-\frac{2173 \pi ^2}{52026975}$ \\[0.5ex]
 $d=5$ & $\frac{2207}{3265920}-\frac{244129 \pi ^2}{14522729760}+\frac{73522 \pi ^4}{541513323351}$ \\[0.5ex]
 $d=6$ & $\frac{26609}{217818720}-\frac{3396146609 \pi ^2}{621871356506400}+\frac{1318349152898 \pi ^4}{12180206401298390455}$ \\[0.5ex]
 \hline
 \multicolumn{2}{c}{}\\[-2ex]
 \hline
 \ru{1.2} & $v_d^{(3)}(T_d)$\\
 \hline
 \ru{1.2}$d=3$ & $\frac{733}{12600000}+\frac{79 \pi ^2}{2424922500}$ \\[0.5ex]
 $d=4$ & $\frac{1955399}{3403417500000}+\frac{63065881 \pi ^2}{39669996140775000}$ \\[0.5ex]
 $d=5$ & $ \frac{362173019}{98363448852480000}+\frac{10217818563857 \pi ^2}{557436796045056999751680}+\frac{602363516243 \pi^4}{569934065465972279392320}
$\\[0.5ex]
\hline 
 \multicolumn{2}{c}{}\\[-2ex]
 \hline
 \ru{1.2} & $v_d^{(5)}(T_d)$\\
 \hline
 \ru{1.2}$d=3$ & $\frac{5125739}{4356374400000}-\frac{547 \pi ^2}{8943995970000}$\\[0.5ex]
 $d=4$ & $\frac{12443146181}{9803685146371200000}-\frac{1262701803371 \pi ^2}{3557043272871373325040000}$\\[0.5ex]
\hline 
\end{tabular}
    \caption{Selected values of $v_d^{(k)}(T_d)$ for odd $k$ and $d=3,4,5,6$.}
    \label{tab:OddkresT}
\end{table}

In higher dimensions in general, other higher order \emph{polylogarithm} functions will appear, that is $\operatorname{Li}_s(x) = \sum_{n=1}^\infty x^n/n^s$. As a consequence, in four dimensions for example, many exact formulae involve \emph{Ap\'{e}ry's constant} \index{Ap\'{e}ry's constant} (which coincides with $\operatorname{Li}_3(1)$):
\begin{equation}
    \zeta(3) = \sum_{n=1}^\infty \frac{1}{n^3} \approx 1.20205690315959428539973816151,\ldots
\end{equation}

An example is the volumetric moments of the tesseract, which are shown in Table \ref{tab:OddkresC4}.
\begin{table}[H]
    \centering
\begin{tabular}{|c|c|}
\hline
 \ru{1.2} & $v_4^{(k)}(C_4)$\\
 \hline
 \ru{1.2}$k=1$ & $\frac{31874628962521753237}{1058357013719040000000}-\frac{26003 \pi^2}{1399680000}+\frac{610208 \ln 2}{1913625}-\frac{536557 \zeta (3)}{2592000}$\\[0.5ex]
 $k=3$ & $\frac{19330626155629115959}{1682723192209145856000000}-\frac{52276897 \pi^2}{216801070940160000}+\frac{10004540239 \ln 2}{77977156950000}-\frac{6155594561 \zeta (3)}{73741860864000}$\\[0.5ex]
 \hline 
\end{tabular}
    \caption{Values of $v_4^{(k)}(C_4)$ for $k=1,3$.}
    \label{tab:OddkresC4}
\end{table}

\newpage
\section{Preliminaries}
\subsection{Symmetries and genealogic decomposition}
\subsubsection{Configurations}
Let $\mathcal{G}(P_d)$ be the group of all isometries of $P_d$ (the \emph{symmetric group}\index{symmetry group} of $P_d$). That is, $\mathcal{G}(P_d)$ is isomorphic to the group of permutations of vertices of $P_d$ such that it leaves $P_d$ unchanged upto rigid transformations (including reflections). Note that in $d=3$, $\mathcal{G}(P_d)$ only consists of \emph{rotations, reflections and improper rotations}\index{improper rotation}. In \textbf{Schoenflies notation}\index{Schoenflies notation}, they are denoted $C_n, \sigma, S_n$, respectively (together with inversion $I$ and identity $E$). Let us select some subset $S$ of vertices $V$ of $P_d$. We can imagine selecting vertices by colouring them (black/white) to create solid $P_d(S)$. We denote $\mathcal{P}_d$ as the set of all those polytopes with pre-selected (coloured) vertices. We say two $P_d(S_1),P_d(S_2) \in \mathcal{P}_d$ are equivalent if there is $g \in \mathcal{G}(P_d)$ such that $g P_d(S_2) = P_d(S_1)$. Moreover we say they are section equivalent if $g P_d(S_2) = P_d(S_1)$ or $g P_d(S_2) = P_d(V\setminus S_1)$. We see that the first condition is more strict since in the latter case, we also identify two coloured polytopes with switched colours. We call the representants of all equivalent classes of coloured polytopes as \textbf{configurations}\index{configuration}. For example, all isometries of a regular octahedron are
\begin{equation}
\mathcal{G}(O_3) = \{E,6C_2,8C_3,6C_4,3C_4^2,I,3\sigma_h,6\sigma_d,6S_4,8S_6\}
\end{equation}

\subsubsection{Weights and orders}\index{configuration!weight}
The size of an orbit of some configuration $C=P_d(S)$ with selected representant vertices $S$ is by definition $o_C = |\mathcal{G}(P_d) C|$, where $\mathcal{G}(P_d)C = \{gC\,|\,g \in \mathcal{G}(P_d)\}$ is the \textbf{orbit}\index{orbit} of $C$. By \emph{orbit-stabilizer lemma}\index{orbit-stabilizer lemma}, $o_C = |\mathcal{G}(P_d)|/|\mathcal{G}_C(P_d)|$, where $\mathcal{G}_C(P_d) = \{g\in \mathcal{G}(P_d) \,| \, gC = C\}$ is the \textbf{stabilizer} subgroup\index{stabilizer subgroup}. The total number of equivalent configurations is given by Burnside's lemma\index{Burnside lemma}. We can find those configurations via the help of a computer, see GECRA (Code \ref{code:GECRA}) in the appendix. The procedure is as follows: First, we represent $\mathcal{G}(P_d)$ as a subgroup of the symmetry group $\mathcal{S}_{|V|}$ with $|V|$ whose elements (permutations) which act of vertices of $P_d$ we represent as permutation matrices. This representation is of course an isomorphism. Then, we can represent a selection (colouring) of vertices $S$ as a vector $\vect{s}$ of length $|V|$ of ones and zeros. Let us denote the set of all such vectors as $\mathbb{S}$. There are $2^{|V|}$ such vectors. The set of all configurations is then simply the classes
\begin{equation}
    \bigcup_{\vect{s} \in \mathbb{S}} |\mathcal{S}_{|V|} \vect{s}| = \bigcup_{\vect{s} \in \mathbb{S}} \{g \vect{s} \, | \, g \in \mathcal{S}_{|V|} \}.
\end{equation}
So far, we have not employed the section equivalence $g C \sim C'$, where we write $C' = P_d(V\setminus S)$. Therefore, for a given configuration $C$ with $S$ selected (coloured) vertices out of total $n$ vertices of $P_d$, we define the section weight $w_C$ as the size of the orbit of $C$ with respect to the section equivalence, that is, by symmetry
\begin{equation}
    w_C = \begin{cases}
    o_C, & |S| < n/2\\
    o_C/2, & |S| = n/2.
    \end{cases}
\end{equation}
Since $\bm{\sigma} \cap P_d$ is also a polytope (more precisely, a $(d-1)$-polytope), we define the \emph{order}\index{configuration!order} $n_C$ of a configuration $C$ as the number of vertices of $\bm{\sigma} \cap P_d$. We claim this number is well defined for a given configuration.

\subsubsection{Realisable configurations}
\vspace{-1em}
\begin{table}[H]
\centering
\setlength{\tabcolsep}{5.2pt}
\begin{GrayBox}
\centering
\begin{tabular}{ccccccc}
    \!\!$\mathcal{G}_C(O_3):$\!\! & \!$\mathcal{G}(O_3)$\! & $\begin{Bmatrix} E,2C_4,C_4^2\\
    2\sigma_h,2\sigma_d \end{Bmatrix}$ & $\begin{Bmatrix} E, C_2,\\ \sigma_h, \sigma_d \end{Bmatrix}$ & $\begin{Bmatrix} E,2C_2,2C_4,3C_4^2,\\
    I,3\sigma_h,2\sigma_d,2S_4\\
     \end{Bmatrix}$ & $\begin{Bmatrix} E,2C_3,\\ 3\sigma_d\end{Bmatrix}$ & $\begin{Bmatrix} E,C_2,\\
     \sigma_h,\sigma_d \end{Bmatrix}$ \\[1em]
    \begin{tabular}{c}
         \!$C:$\!  \\[3em] 
    \end{tabular}
    &    
    \!\includegraphics[width=45pt]{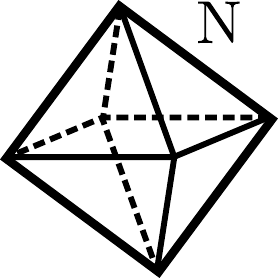}\!
    &
    \includegraphics[width=45pt]{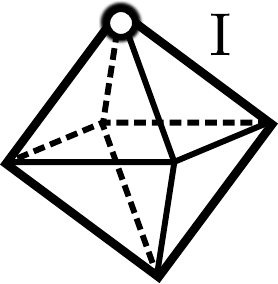}
    &
    \includegraphics[width=45pt]{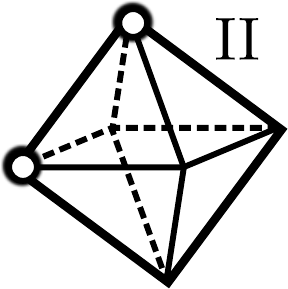}
    &
    \includegraphics[width=45pt]{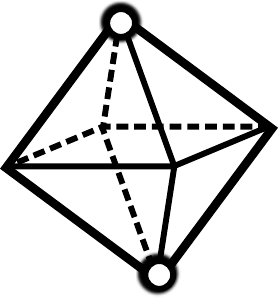}
    &
    \includegraphics[width=45pt]{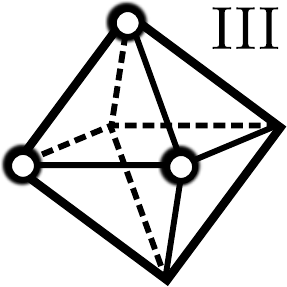}
    &
    \includegraphics[width=45pt]{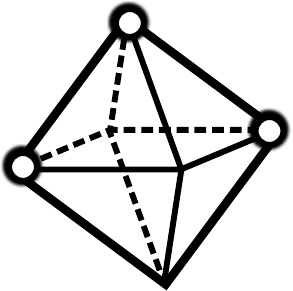}
    \\[-1.3em]
    \!\!$o_C:$\!\! & $1$ & $6$ & $12$ & $3$ & $8$ & $12$
    \\
    \!\!$w_C:$\!\! & $1$ & $6$ & $12$ & $3$ & $4$ & $6$
\end{tabular}
\end{GrayBox}
\vspace{-1em}
\caption{Section equivalent configurations of $O_3$}
\label{tab:realconocta}
\end{table}    

In the example above shown in Table \ref{tab:realconocta}, configurations $\mathrm{N},\mathrm{I},\mathrm{II},\mathrm{III}$ are those whose points can be separated by a plane. In general that is, there exists a $(d-1)$-plane $\bm{\sigma}$ such that all vertices in $S$ lie on side of $\bm{\sigma}$ and all remaining vertices $V\setminus S$ lie on the other side. Those configurations are said to be \emph{realisable}\index{configuration!realisable}. We write $\bm{\sigma}/ C$, where $C$ is a configuration and we write $\mathcal{C}(P_d)$ for the set of all realisable configurations. We can check whether a configuration is realisable by checking whether there is a nonempty subset of $\mathbb{R}^d$ satisfying Equations \eqref{Eq:Separ}.

\subsubsection{Genealogy}
Realisable configurations have a unique property -- assuming $P_d$ is convex, we can obtain them from realisable configurations with fewer coloured vertices by successively adding (colouring) another neighbouring vertex. This corresponds to a continuous shift of $\bm{\sigma}$. The graph (in fact, a \emph{Hasse diagam}\index{Hasse diagram}) of such successions is called the \textbf{genealogy}\index{Genealogy} of $P_d$ configurations with section weights $w_C$. Generalogies for selected polyhedra are shown in Appendix \ref{Apx:Gen}. For example, the genealogy of the octahedron section equivalent configurations from Table \ref{tab:realconocta} are shown in Figure \ref{fig:OCTAHE_GENEALOGY}.

\subsubsection{Decomposition of functionals}
Let $P \subset \mathbb{R}^d$ be a convex $d$-polytope. Consider an affinely invariant functional
\begin{equation}
    F(P) = \frac{1}{(\vol_d P)^d} \int_{P^d} f(\mathbbm{x}) \lambda_d^d (\dd \mathbbm{x}).
\end{equation}
By symmetry, we can decompose this functional as follows
\begin{equation}\label{Eq:DecoCon}
F(P) = \sum_{C \in \mathcal{C}(P)} w_C F(P)_C,
\end{equation}
where
\begin{equation}
F(K)_C = \frac{1}{(\vol_d P)^d} \int_{P^d} \mathbbm{1}\{\mathcal{A}(\mathbbm{x})/ C\} f(\mathbbm{x})\lambda_d^d (\dd \mathbbm{x}).
\end{equation}
Note that the property $\mathcal{A}(\mathbbm{x})/ C$ is also affinely invariant, since any affine transformation does not change the set of vertices $S$ separated by $\bm{\sigma}$. As a consequence, also $F_C(K)$ stays invariant under affine trasformations of $K$. By defining $P_C = \{\vect{x} \in P \, |\, \mathcal{A}(\mathbbm{x})/C\}$, we may also write
\begin{equation}
F(K)_C = \frac{1}{(\vol_d P)^d} \int_{P_C^d} f(\mathbbm{x})\lambda_d^d (\dd \mathbbm{x}).
\end{equation}

\subsection{Integral calculus on real affine subspaces}
%\subsection{Invariant measures}
First, we shall discuss the common techniques of multidimensional integration. The notation used in this section is borrowed from the textbook \emph{Lectures on convex geometry} by Hug and Weil \cite{hug2020lectures}. Once again, let us recall some basic facts and definitions.
\begin{definition}[$\mathbb{S}^{d-1},\omega_d$] Let $\mathbb{S}^{d-1}$ be a unit sphere in $\mathbb{R}^d$ with the usual surface area measure $\sigma_{d}(\cdot)$. That is, for the surface area of $\mathbb{S}^{d-1}$, we get
\begin{equation}
    \omega_d = \int_{\mathbb{S}^{d-1}} \sigma_{d}(\dd \vect{u}) = \sigma_{d}(\mathbb{S}^{d-1}) = \frac{2\pi^{\tfrac{d}{2}}}{\Gamma(\tfrac{d}{2})}.
\end{equation}
Also note that we can decompose $\vect{x} = r \vect{u}$, where $\vect{u} \in \mathbb{S}^{d-1}$ and $r \in (0,\infty)$, the usual Lebesgue measure $\lambda_d$ splits into radial and angular part as $\lambda_d(\dd \vect x) = r^{d-1} \dd r \sigma_d(\dd \vect{u})$.
\end{definition}

\begin{definition}[$\solB_d,\kappa_d$]
We write $\solB_d\subset \mathbb{R}^d$ for the unit ball (with unit radius) and $\kappa_d$ for its volume. Splitting the Lebesgue measure into radial and angular part,
\begin{equation}
\kappa_d = \vol_d\solB_d = \int_{\solB_d}\lambda_d(\dd \vect{x}) = \omega_d \int_0^1 r^{d-1}\ddd r = \omega_d/d.
\end{equation}
\end{definition}

\begin{definition}
We denote $\mathbb{G}(d,p)$\index{Grassmannian!linear} as the set of all linear $p$-dimensional subspaces of $\mathbb{R}^d$, this set is often called the (linear) Grassmannian. More generally, we denote $\mathbb{A}(d,p)$\index{Grassmannian!affine} as the set of all $p$-dimensional affine subspaces of $\mathbb{R}^d$ ($p$-planes), this set is called the affine Grasmannian.
\end{definition}

\begin{remark}
Both spaces $\mathbb{G}(d,p)$ and $\mathbb{A}(d,p)$ have a finite-dimensional basis. More concretely, we have $\dim \mathbb{G}(d,p) = (d-p)p$ and $\dim\mathbb{A}(d,p) = (d-p)(p+1)$.
\end{remark}

\begin{definition}
Let $K_d \subset \mathbb{R}^d$. We define $\mathbb{G}_{K_d}(d,p) = \{\bm{\gamma} \in \mathbb{G}(d,p)\, |\, \bm{\gamma} \cap K_d \neq \emptyset \}$ and analogously, $\mathbb{A}_{K_d}(d,p) = \{\bm{\sigma} \in \mathbb{A}(d,p)\, |\, \bm{\sigma} \cap K_d \neq \emptyset \}$.
\end{definition}

\begin{definition}
Let $\nu_p$ be the probability Haar measure\index{Haar measure} on $\mathbb{G}(d,p)$. That is, $\nu_p$ is invariant under action of the group of proper rigid sphere transformations $\mathcal{SO}(n)$\index{group!special orthogonal} and $\nu_p(\mathbb{G}(d,p)) = 1$.
\end{definition}

\begin{definition}
We define the standard Haar measure $\mu_p$ on $\mathbb{A}(d,p)$ by
\begin{equation}
\mu_p(\cdot) = \int_{\mathbb{G}(d,p)} \int_{\bm{\gamma}_{\!\perp}} \mathbbm{1}\{\bm{\gamma}+\vect{y} \in \cdot\} \lambda_{d-p}(\dd \vect{y}) \nu_p(\dd\bm{\gamma}),
\end{equation}
where $\bm{\gamma}_{\perp} \in \mathbb{G}(d,d-p)$ is the linear space orthogonal to $\bm{\gamma}$. That is, $\bm{\gamma}_{\perp}  \oplus \bm{\gamma} = \mathbb{R}^d$.
\end{definition}
\begin{lemma}
$\mu_p(\mathbb{A}_{\solB_d}(d,p)) = \kappa_{d-p}=\omega_{d-p}/(d-p)$.
\end{lemma}
\begin{proof}
By symmetry, we have for any $\gamma_0 \in \mathbb{G}(d,p)$,
\begin{equation}
\mu_p(\mathbb{A}_{\solB_d}(d,p)) = \int_{\bm{\gamma}_{\!\perp}} \mathbbm{1}\{\bm{\gamma}+\vect{y} \in \mathbb{A}_{\solB_d}(d,p)\} \lambda_{d-p}(\dd \vect{y}) = \vol_{d-p}(\solB_d\cap \bm{\gamma}_\perp) = \kappa_{d-p}.
\end{equation}
\end{proof}

%The following Lemma is taken from Hug and Weil \cite[Theorem 5.4]{hug2020lectures}
%\begin{lemma}
%The Borel measure $\mu_p$ on $\mathcal{B}(\mathbb{A}_{\mathbb{B}_d}(d,p))$ is the unique $\mathcal{G}(d)$-invariant measure on $\mathbb{A}_{\mathbb{B}_d}(d,p)$ with $\mu_p(\mathbb{A}_{\mathbb{B}_d}(d,p)) = \kappa_{d-p}$.
%\end{lemma}

\subsubsection{Cartesian parametrisation}\index{Grassmannian!Cartesian parametrisation}
In the case of $p=d-1$, the affine Grassmannian $\mathbb{A}(d,d-1)$ consists of hyperplanes of dimension $d-1$. Note that $\dim \mathbb{A}(d,d-1) = d$ so in order to parametrize the space of all affine planes, we need exactly $d$ parameters. One choice of those parameters are the coordinates of the closest point to a given hyperplane, we write $\bm{\xi}=(\xi_1,\ldots,\xi_d)^\top$ to be the vector from the origin to the closest point on the hyperplane $\bm{\sigma}$. Another choice of parametrisation is by using spherical inversion of $\bm{\xi}$. Namely,
\begin{equation}
    \bm{\eta} = \frac{\bm{\xi}}{\bm{\xi}^\top\bm{\xi}} = \frac{\bm{\xi}}{\| \bm{\xi}\|^2}.
\end{equation}
so $\|\bm{\xi}\|=1/\|\bm{\eta}\|$. There is a nice interpretation of $\bm{\eta}$. Namely, a plane $\bm{\sigma}$ defined uniquely by the vector $\bm{\eta}$ has a nonempty intersection with convex body $K_d \subset \mathbb{R}^d$ if and only if $\bm{\eta}$ does not lie in the \emph{polar body}\index{polar body} $K_d^\circ$ defined as
\begin{equation}
    K_d^\circ = \{  \vect{x} \in \mathbb{R}^d \, |\, \vect{x}^\top \vect{y} \leq 1, \vect{y} \in K_d \}.
\end{equation}
This follows from the fact that the points $\vect{x}$ on the hyperplane $\bm{\sigma} \in \mathbb{A}(d,d-1)$ satisfy $\bm{\eta}^\top \vect{x} = 1$. The following lemma gives us then the Jacobian of transformation between the standard Haar measure on a Grasmannian of hyperplanes and the Lebesgue measure of the closest point intercepts:

\begin{lemma}
Let $\bm{\sigma} \in \mathbb{A}(d,d-1)$ and $\bm{\eta} =(\eta_1,\ldots,\eta_d)^\top$ be the plane vector associated to $\bm{\sigma}$ such that $\vect{x} \in \bm{\sigma} \Leftrightarrow \bm{\eta}^\top \vect{x} = 1$, then
\begin{equation}
\mu_{d-1}(\dd \bm{\sigma}) = \frac{2}{\omega_d  }\frac{1}{\|\bm{\eta}\|^{1+d}} \lambda_d(\dd \bm{\eta}).
\end{equation}
\end{lemma}

\begin{proof} First, we show that our measure on the right hand side is invariant with respect to action of the group $\mathcal{G}(d)$\index{group!proper rigid motions} of all proper rigid motions in $\mathbb{R}^d$. We may view any $g(M,\vect{b}) \in \mathcal{G}(d)$ by its corresponding action on points $\vect{x} \in \mathbb{R}^d$. That is,
\begin{equation}\label{Eq:x'x}
    \vect{x}' = g(M,\vect{b})\circ\vect{x} = M \vect{x} + \vect{b}
\end{equation}
where $\vect{b}$ is a translation vector and the matrix $M$ corresponds to (proper) rotations, hence $M$ satisfies
\begin{equation}
\det M =1 \text{  and  } M^\top M = MM^\top = I_d,
\end{equation}
where $I_d$ is the $d\times d$ identity matrix\index{identity matrix}. Let us find $\bm{\sigma}' = g^{-1}(M,\vect{b})\circ (\bm{\sigma})$ onto which $\bm{\sigma}$ is mapped by applying $g^{-1}(M,\vect{b})$. Its associated plane vector $\bm{\eta}'$ must satisfy $\bm{\eta}'^\top \vect{x}' = 1
$. By multiplying Equation \eqref{Eq:x'x} by $M^\top$ from the left, we obtain $M^\top \vect{x}' = \vect{x} + M^\top \vect{b}$. Further multiplying by $\bm{\eta}^\top$ from the left, we get $\bm{\eta}^\top M^\top \vect{x}' = 1 + \bm{\eta}^\top M^\top \vect{b}$, from which we identify
\begin{equation}
\bm{\eta}' = g^{-1}(M,\vect{b})\circ \bm{\eta} = \frac{M \bm{\eta}}{1+\vect{b}^\top M \bm{\eta}}.
\end{equation}
For the norm, we have by using $M^\top M = I_d$,
\begin{equation}
\|\bm{\eta}'\| = \frac{\|\bm{\eta}\|}{|1+\vect{b}^\top M \bm{\eta}|}.   
\end{equation}
Let us calculate the Jacobian of transformation from $\bm{\eta}'$ to $\bm{\eta}$. We have
\begin{equation}
\frac{\partial \bm{\eta}'}{\partial \bm{\eta}} = \frac{M (1+\vect{b}^\top M \bm{\eta})-M\bm{\eta}\vect{b}^\top M}{(1+\vect{b}^\top M \bm{\eta})^2}  = M\frac{I_d -\frac{\bm{\eta}\vect{b}^\top M}{1+\vect{b}^\top M \bm{\eta}}}{1+\vect{b}^\top M \bm{\eta}},
\end{equation}
By Matrix Determinant Lemma,
\begin{equation}
    \det\left(\frac{\partial \bm{\eta}'}{\partial \bm{\eta}}\right) = \frac{\det M}{(1+\vect{b}^\top M \bm{\eta})^d} \left(1-\frac{\vect{b}^\top M\bm{\eta}}{1+\vect{b}^\top M \bm{\eta}}\right) = \frac{\det M}{(1+\vect{b}^\top M \bm{\eta})^{1+d}}.
\end{equation}
In total,
\begin{equation}
    \frac{1}{\|\bm{\eta}'\|^{1+d}} \lambda_d(\dd \bm{\eta}')
    = \frac{|1+\vect{b}^\top M \bm{\eta}|^{1+d}}{\|\bm{\eta}\|^{1+d}}\frac{\det M}{|1+\vect{b}^\top M \bm{\eta}|^{1+d}}  \lambda_d(\dd \bm{\eta}) 
    = \frac{1}{\|\bm{\eta}\|^{1+d}}\lambda_d(\dd \bm{\eta})
\end{equation}
for any $M$ and $\vect{b}$. Therefore, $\|\bm{\eta}\|^{-1-d} \lambda_d(\dd \bm{\eta})$ is a Haar measure on $\mathbb{A}(d,d-1)$ and as such, it must differ from $\mu_{d-1}(\dd \bm{\sigma})$ by a constant multiple \cite[Theorem 5.4]{hug2020lectures}, say
\begin{equation}
\mu_{d-1}(\dd \bm{\sigma}) = \frac{c}{\|\bm{\eta}\|^{1+d}} \lambda_d(\dd \bm{\eta})
\end{equation}
for some $c$. To check this constant is indeed $c=2/\omega_d$, let us calculate the $\mu_{d-1}$ measure over planes which pass trough $\solB_d$ (the unit ball with radius one). On one hand, by definition, we already know that $\mu_{d-1}(\mathbb{A}_{\solB_d}(d,d-1)) = \omega_1 = 2$. On the other, let us characterise the condition under which a $(d-1)$ hyperplane $\bm{\sigma}$ intercepts $\solB_d$. This happens exactly when the closest point on $\bm{\sigma}$ lies inside of $\solB_d$. That is, $\|\bm{\xi}\| < 1$, or equivalently $\|\bm{\eta}\|>1$. Hence, by using spherical coordinates and symmetry, $\lambda_d(\dd \bm{\eta}) = \omega_d r^{d-1} \ddd r$, where $r = \|\bm{\eta}\|$, and therefore%and $\uvect{u} = \bm{\eta}/\|\bm{\eta}\|$,
\begin{equation}
    \mu_{d-1}(\mathbb{A}_{\solB_d}(d,d-1)) = \int_{\mathbb{R}^d\setminus \solB_d} \frac{c}{\|\bm{\eta}\|^{1+d}}  \lambda_d(\dd\bm{\eta}) = \omega_d \int_1^\infty  c \frac{r^{d-1}}{ r^{1+d}} \dd r = \omega_d c,
\end{equation}
so $c=2/\omega_d$ indeed.
\end{proof}

\begin{remark}
Simple calculation of Jacobian of transformation between $\bm{\eta}$ and $\bm{\xi}$ (only the radial part is affected) reveals that
\begin{equation}
\mu_{d-1}(\dd \bm{\sigma}) = \frac{2}{\omega_d} \|\bm{\xi}\|^{1-d} \lambda_d(\dd \bm{\xi}).
\end{equation}
\end{remark}

\subsubsection{Blaschke-Petkantschin formula}\index{Blaschke-Petkantschin formula}
The following formula by Blaschke and Petkantschin enables us to reparametrize an integral over some set of points $\mathbbm{x}=(\vect{x}_0,\ldots,\vect{x}_p)$ as an integral over planes in $\mathbb{A}(d,q)$, $q\geq p$ on which these points lie.
\begin{theorem}\label{Thm:BP} Let $f:(\mathbb{R}^d)^{p+1} \to \mathbb{R}$ be a Lebesgue integrable function of a collection $\mathbbm{x} = (\vect{x}_0,\ldots,\vect{x}_p)$ of points $\vect{x}_j \in \mathbb{R}^d$, $j=0,\ldots,p$. Denote $\solH_p = \hull (\mathbbm{x})$ and $\Delta_p = \vol_p \solH_p$, then for any integer $q$ such that $0 \leq p \leq q \leq d$,
\begin{equation}
    \int_{(\mathbb{R}^d)^{p+1}} f(\mathbbm{x}) \lambda_d^{p+1}(\dd\mathbbm{x}) = \beta_{dqp}\int_{\mathbb{A}(d,q)}\int_{\bm{\sigma}^{p+1}} f(\mathbbm{x}) \Delta_p^{d-q} \lambda_q^{p+1} (\dd\mathbbm{x}) \mu_q(\dd \bm{\sigma}),
\end{equation}
where
\begin{equation}
    \beta_{dqp} = (p!)^{d-q} \pi ^{\frac{1}{2} p (d-q)} \prod_{j=0}^{p-1} \frac{\Gamma \left(\frac{q-j}{2}\right)}{\Gamma \left(\frac{d-j}{2}\right)},
\end{equation}
$\lambda_d^{p+1}(\dd\mathbbm{x}) = \prod_{j=0}^p \lambda_d(\dd \vect{x}_j)$ and $\lambda_q^{p+1}(\dd\mathbbm{x}) = \prod_{j=0}^p \lambda_q(\dd \vect{x}_j)$ are the Lebesgue measures on $(\mathbb{R}^d)^{p+1}$ and $\bm{\sigma}^{p+1}$, respectively.
\end{theorem}
\begin{proof}
    See Rubin \cite{rubin2018blaschke} for an elementary proof.
\end{proof}
\begin{remark}\label{Rem:BetaOmega}
Denote $\gamma_d = \int_0^\infty r^{d-1} e^{-r^2/2} \ddd r = 2^{\frac{d}{2}-1}\Gamma(\tfrac{d}{2})$ as before. We have $\omega_d\gamma_d = \sqrt{2\pi}^d$. We can express $\beta_{dqp}$ in terms of $\gamma$'s and $\omega$'s as follows:
\begin{equation}
    \beta_{dqp} =(p!)^{d-q} \sqrt{2\pi}^{p (d-q)} \prod_{j=0}^{p-1} \frac{\gamma_{q-j}}{\gamma_{d-j}} = (p!)^{d-q} \prod_{j=0}^{p-1} \frac{\omega_{d-j}}{\omega_{q-j}}.
\end{equation}
\end{remark}

The statement of the Blashke-Petkantschin formula is way too general for our purposes. We will only need its special cases. First, often we assume that the affine plane on which the points lie is exactly their affine hull almost surely. This corresponds to the case $p=q$, for which we define $\beta_{dp} = \beta_{dpp}$. Another special case is got by restricting the domain of integration using the following choice of $f$: Let $K_d \subset \mathbb{R}^d$ be a compact convex body with $\dim K_d = d$ and let $f(\mathbbm{x}) = \tildee{f}(\mathbbm{x}) \prod_{0\leq i \leq k} \mathbbm{1}_{K_d}(\vect{x}_i)$ for some $\tildee{f}: K_d^{p+1} \to \mathbb{R}$ suitably integrable, then
\begin{equation}\label{Eq:BPres}
    \int_{K_d^{p+1}} \tildee{f}(\mathbbm{x}) \lambda_d^{p+1}(\dd\mathbbm{x}) = \beta_{dqp}\int_{\mathbb{A}_{K_d}(d,q)}\int_{(K_d \cap \, \bm{\sigma})^{p+1}} \tildee{f}(\mathbbm{x}) \Delta_p^{d-q} \lambda_q^{p+1}(\dd\mathbbm{x}) \mu_q(\dd \bm{\sigma}).
\end{equation}
In this paper, mostly we use the special case with $\tildee{f}(\mathbbm{x}) = g(\mathcal{A}(\mathbbm{x})) \Delta_p^k$, where $g(\cdot)$ is a function of the cutting plane $\bm{\sigma} = \mathcal{A}(\mathbbm{x}) \in \mathbb{A}(d,q)$ only. In this case, the Blaschke-Petkantschin formula restricted on $K_d$ as in Equation \eqref{Eq:BPres} becomes, using definition of $v_p^{(n)}(\cdot)$ and denoting $\bm{\sigma}_{\! K_d} = K_d \cap \bm{\sigma}$ ($\dim \bm{\sigma}_{\!K_d} = q$ almost surely),
\begin{equation}\label{Eq:BPresAlter}
    \int_{K_d^{p+1}} g(\bm{\sigma}) \Delta_p^k\lambda_d^{p+1}(\dd\mathbbm{x}) = \beta_{dqp}\int_{\mathbb{A}_{K_d}(d,q)} v_p^{(d-q+k)}(\bm{\sigma}_{\! K_d}) (\vol_q \bm{\sigma}_{\!K_d})^{1+(d+k)\frac{p}{q}} g(\bm{\sigma}) \mu_q(\dd \bm{\sigma}).
\end{equation}
This still very general relation can be further reformulated in terms of expected values. Let us select the collection $\mathbb{X}=(\vect{X}_0,\vect{X}_1,\ldots,\vect{X}_p)$ of $(p+1)$ random points $\vect{X}_i$ independently from the same distribution $\mathsf{Unif}(K_d)$. Then
\begin{corollary}
\label{Cor:BP} With respect to the uniform probability measure $\mathsf{Unif}(K_d)$,
\begin{equation}
    \expe{g(\bm{\sigma}) \Delta_p^k} = \frac{\beta_{dqp}}{(\vol_d K_d)^{p+1}}\int_{\mathbb{A}_{K_d}(d,q)} v_p^{(d-q+k)}(\bm{\sigma}_{\! K_d}) (\vol_q \bm{\sigma}_{\!K_d})^{1+(d+k)\frac{p}{q}} g(\bm{\sigma}) \mu_q(\dd \bm{\sigma}).
\end{equation}
\end{corollary}

If moreover $q=d-1$, for which, by using Remark \ref{Rem:BetaOmega},
\begin{equation}
 \beta_{d(d-1)p} = p! \prod_{j=0}^{p-1} \frac{\omega_{d-j}}{\omega_{d-1-j}} = \frac{p! \omega_d}{\omega_{d-p}},
\end{equation}
we may use the Cartesian parametrisation $\vect{x}\in \bm{\sigma} \Leftrightarrow \bm{\eta}^\top \vect{x} = 1$, to get
\begin{equation}\label{BP:genspec}
    \expe{g(\bm{\sigma}) \Delta_p^k} = \frac{2\, p!}{\omega_{d-p}(\vol_d K_d)^{p+1}}\int_{\mathbb{R}^d\setminus K_d^\circ} v_{p}^{(k+1)}(\bm{\sigma}_{\! K_d}) (\vol_{d-1}\bm{\sigma}_{\! K_d})^{1+\frac{(d+k)p}{d-1}} g(\bm{\sigma}) \frac{\lambda_d(\dd \bm{\eta})}{\|\bm{\eta}\|^{1+d}},
\end{equation}
where $\bm{\sigma}$ is now a function of $\bm{\eta}$. If moreover $p=q=d-1$, we get
\begin{equation}\label{BP:genpqdm1}
    \expe{g(\bm{\sigma}) \Delta_{d-1}^k} = \frac{(d-1)!}{(\vol_d K_d)^d}\int_{\mathbb{R}^d\setminus K_d^\circ} v_{d-1}^{(k+1)}(\bm{\sigma}_{\! K_d}) (\vol_{d-1}\bm{\sigma}_{\! K_d})^{d+k+1} g(\bm{\sigma}) \frac{\lambda_d(\dd \bm{\eta})}{\|\bm{\eta}\|^{1+d}}.
\end{equation}
We may write this relation in the form of the following corollary
\begin{corollary}\label{BP:zeta}
With respect to the uniform probability measure $\mathsf{Unif}(K_d)$,
\begin{equation*}
\expe{g(\bm{\sigma}) \Delta_{d-1}^k} = (d-1)!(\vol_d K_d)^{k+1} \int_{\mathbb{R}^d\setminus K_d^\circ} v_{d-1}^{(k+1)}(\bm{\sigma}_{\! K_d}) \zeta_d^{d+k+1}(\bm{\sigma}) g(\bm{\sigma}) \|\bm{\eta}\|^k \lambda_d(\dd \bm{\eta})
\end{equation*}
where we defined
\begin{equation}
    \zeta_d(\bm{\sigma}) = \frac{\vol_{d-1}(\bm{\sigma}\cap K_d)}{\|\bm{\eta}\|\vol_d K_d}
\end{equation}
\end{corollary}
This form of the Blashke-Petkantschin formula is illustrated in Figure \ref{fig:CorBP} below.
\begin{figure}[H]
    \centering     \includegraphics[width=0.35\textwidth]{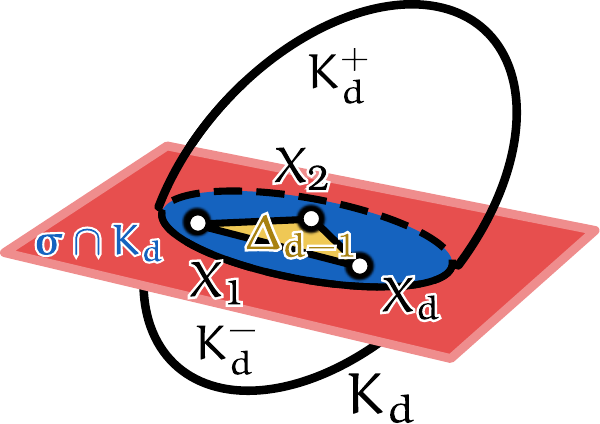}
    \caption{Blashke-Petkantschin formula replaces integration over space by integration over sections planes diving $K_d$ into $K_d^- \sqcup K_d^+$}
    \label{fig:CorBP}
\end{figure}

\begin{remark}\label{Rem:HomZeta}
We show that $\| \bm{\eta} \|$ always cancels out in $\zeta_d(\bm{\sigma})$. First, note that $\bm{\sigma}$ always separates $K_d$ into disjoint union $K_d^+ \sqcup K_d^-$, where
\begin{equation}\label{Eq:Kplusminus}
    K_d^+ = \{\vect x \in K_d \, | \, \bm{\eta}^\top \vect{x} < 1\}, \qquad K_d^- = \{\vect x \in K_d \, | \, \bm{\eta}^\top \vect{x} > 1\}.
\end{equation}
From homogeneity of $d$-volume, $\frac{\vol_d(\bm{\sigma}\cap K_d)}{\| \bm{\eta} \|} = -\sum_{j=1}^d \eta_j\frac{\partial\vol_d K_d^+}{\partial \eta_j}$ and thus
\begin{equation}\label{Eq:ZetaEtaGen}
\zeta_d(\bm{\sigma}) = - \frac{1}{\vol_d K_d} \sum_{j=1}^d \eta_j\frac{\partial\vol_d K_d^+}{\partial \eta_j} = \frac{1}{\vol_d K_d} \sum_{j=1}^d \eta_j\frac{\partial\vol_d K_d^-}{\partial \eta_j}.
\end{equation}
\end{remark}

\subsection{General volumetric moments}\label{sec:GVM}
In the following subsections, we will investigate the method of section integral. That is, instead of integrating over points, we integrate over section in the spirit of Blashke-Petkantschin formula. In fact, there are two approaches. The first is based on Efron's section formula. This was the original approach used by Buchta and Reitzner \cite{buchta1992expected} to deduce $v_3^{(1)}(T_3)$ and eventually \cite{buchta2001convex} also $v_n^{(1)}(T_3)$ for any integer $n\geq 3$, as well by Zinani \cite{zinani2003expected} who obtained $v_3^{(1)}(C_3)$. The second approach, which is the main approach used in this paper, is discussed in the next section and the rest of this chapter.

\subsubsection{Efron section integral}

\vspace{1em}
Let $K_d \subset \mathbb{R}^d$ be a convex body and let $\mathbb{X}' = (\vect{X}_1',\ldots,\vect{X}_d')$ be a collection of points $\vect{X}_j'$ drawn uniformly and independently from $K_d$. It is convenient to denote
\begin{equation}
\gamma_n(K_d) = \expe{\Gamma_{\! d}(\mathbb{X}')^{n-d+2} + (1-\Gamma_{\! d}(\mathbb{X}'))^{n-d+2}},
\end{equation}
where $\Gamma_{\! d}(\mathbb{X}') = \vol_d K_d^+/\vol_d K_d$ is the $d$-volume fraction of one of the two parts $K_d^+\sqcup K_d^-$ into which $K_d$ is divided by a hyperplane $\bm{\sigma} = \mathcal{A}(\mathbb{X}')$. Written as an integral, this is equivalent to

\begin{equation}
\gamma_n(K_d) =\frac{1}{(\vol_d K_d)^d}\int_{K_d^d}\Gamma_{\! d}(\mathbbm{x}')^{n-d+2} + (1-\Gamma_{\! d}(\mathbbm{x}'))^{n-d+2} \lambda_d^d(\dd\mathbbm{x}'),
\end{equation}
where $\mathbbm{x}' = (\vect{x}_1',\vect{x}_2',\ldots,\vect{x}_d')$ is the collection of points $\vect{x}_j' = (x_{1j}',\ldots,x_{dj}')^\top$, $j\in\{1,2,\ldots,d\}$ and $\lambda_d^d(\dd \mathbbm{x}') = \lambda_d(\dd\vect{x}_1') \lambda_d(\dd \vect{x}_2')\cdots \lambda_d(\dd\vect{x}_d')= \prod_{i,j=0}^d \dd x_{ij}'$ is the usual Lebesgue measure on $(\mathbb{R}^d)^d$.

\vspace{1em}
Note that $\gamma_n(K_d)$ is an affine functional. If $K_d$ is some sufficiently symmetric polytope $P_d$, we can further use genealogic decomposition
\begin{equation}
    \gamma_n(P_d) = \sum_{C \in \mathcal{C}(P_d)} w_C \,\gamma_n(P_d)_C.
\end{equation}
Efron's section formulae \cite{efron1965convex} then can be written in the following compact form
\begin{equation}
v_n^{(1)}(K_2) = 1 -  \frac{n+1}{2} \gamma_n(K_2), \qquad v_n^{(1)}(K_3) = \frac{n}{n+2} - \frac{n(n+1)}{12}\gamma_n(K_3).
\end{equation}
By Blaschke-Petkantschin formula (in the form of Corollary \ref{BP:zeta}) with $k=0$ and $g(\bm{\sigma}) = \Gamma_{\! d}(\bm{\sigma})^{n-d+2} + (1-\Gamma_{\! d}(\bm{\sigma}))^{n-d+2}$, we get
\begin{equation*}
\gamma_n(K_d) = (d-1)!\vol_d K_d \int_{\mathbb{R}^d\setminus K_d^\circ} v_{d-1}^{(1)}(\bm{\sigma}\cap K_d) \zeta_d^{d+1}(\bm{\sigma}) g(\bm{\sigma}) \lambda_d(\dd \bm{\eta}),
\end{equation*}
where $\bm{\eta}$ is the Cartesian representation of $\bm{\sigma}$ defined by the relation $\bm{\eta}^\top \vect{x} = 1$. In this representation, we have $K_d^+ = \{\vect x \in K_d \, | \, \bm{\eta}^\top \vect{x} < 1\}$
and (by Remark \ref{Rem:HomZeta})
\begin{equation}
\zeta_d(\bm{\sigma}) = \frac{\vol_{d-1}(\bm{\sigma}\cap K_d)}{\|\bm{\eta}\|\vol_d K_d} = - \frac{1}{\vol_d K_d} \sum_{j=1}^d \eta_j\frac{\partial\vol_d K_d^+}{\partial \eta_j} = -\sum_{j=1}^d \eta_j\frac{\partial\Gamma_{\! d}(K_d)}{\partial \eta_j}.
\end{equation}
The integral above can be always solved when the integrand is a rational function. This happens when $K_3 = P_3$ a convex polygon. Then, $\bm{\sigma}\cap P_3$ is some convex polytope $P_2$. Since $v_2^{(1)}(P_2)$ is known for any convex polytope (due to Buchta and Reitzner \cite{buchta1997equiaffine}), in fact it is a rational function, we can plug this value into the integral and then integrate everything out. We can use this formula to deduce the first volume moment relatively easily regardless of the number of points in the convex hull. This is the method that we originally used to derive $v_3^{(1)}(P_3)$ for polyhedra in Table \ref{tab:allsolids}.

\subsubsection{Canonical section integral}\label{sec:Canon}
As there is no analog of Efron's section formula for higher moments and dimensions, we might use the second section integral approach applicable to volumetric moments $v_d^{(k)}(K_d)$ for any $k$ (picking a $d$-simplex from a $d$-dimensional body $K_d$). Let us restate the main Theorem \ref{Thm:Canon} used in this work
\begin{proposition}
Let $K_d$ be a $d$-dimensional convex body, $\mathbbm{x}' = (\vect{x}_1,\ldots,\vect{x}_d)$ a collection of $d$ points in $K_d$ and $\bm{\sigma} = \mathcal{A}(\mathbbm{x}') \in \mathbb{A}(d,d-1)$ be a hyperplane parametrised by $\bm{\eta}=(\eta_1\ldots,\eta_d)^\top \in \mathbb{R}^d$ as $\vect{x}\in \bm{\sigma} \Leftrightarrow \bm{\eta}^\top \vect{x} = 1$, then
\begin{equation}
 v_d^{(k)}(K_d) = \frac{(d-1)!}{ d^k}\int_{\mathbb{R}^d\setminus K_d^\circ} v_{d-1}^{(k+1)}(\bm{\sigma}\cap K_d) \,\zeta_d^{d+k+1}(\bm{\sigma}) \iota^{(k)}_d(\bm{\sigma}) \lambda_d(\dd \bm{\eta})
\end{equation}
for any real $k>-1$, where
\begin{equation}
    \zeta_d(\bm{\sigma}) = \frac{\vol_{d-1}(\bm{\sigma}\cap K_d)}{\|\bm{\eta}\|\vol_d K_d}, \qquad \iota^{(k)}_d(\bm{\sigma}) = \int_{K_d} |\bm{\eta}^\top \vect{x}-1|^k \lambda_d(\dd\vect{x}).
\end{equation}
\end{proposition}
\begin{proof}[Proof of Theorem \ref{Thm:Canon}]
Let $\mathbb{X} = (\vect{X}_0,\ldots,\vect{X}_n)$ be a collection of random $n+1$ i.i.d. points taken uniformly from $K_d$, let $\solH_n = \hull(\mathbb{X})$ be their convex hull and $\Delta_n = \vol_d \solH_n$, then we have in general ($n\geq d$)
\begin{equation}
    v_n^{(k)}(K_d) = \frac{\Exx [\Delta_n^k]}{(\vol_d K_d)^k}.
\end{equation}
When $n=d$, $\solH_d$ is almost surely a $d$-simplex. That means that any $d$-tuple of points $\vect{X}_i$ from $\mathbb{X}$ form a facet. Let $\mathbb{X}' = (\vect{X}_1,\ldots,\vect{X}_d)$, $\bm{\sigma} = \mathcal{A}(\mathbb{X}')$ as in the statement of the theorem and let $\operatorname{dist}_{\bm{\sigma}}(\vect{X}_0)$ be the distance from $\bm{\sigma}$ to the point $\vect{X}_0$, then by base-height splitting,
\begin{equation}
    \Delta_d = \frac{1}{d}  \operatorname{dist}_{\bm{\sigma}}(\vect{X}_0)\Delta_{d-1},
\end{equation}
where
$\Delta_{d-1} = \vol_{d-1}\hull(\mathbb{X}')$. See Figure \ref{fig:CSI} below.
\begin{figure}[H]
    \centering     \includegraphics[width=0.35\textwidth]{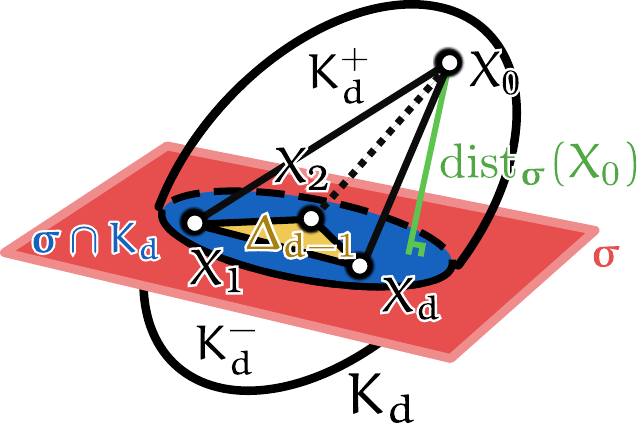}
    \caption{Base-height splitting}
    \label{fig:CSI}
\end{figure}

Fixing $\mathbb{X}'$, we get by conditioning,
\begin{equation}\label{Eq:CanonDec}
    v_d^{(k)}(K_d) = \frac{\Exx[ \Exx [\operatorname{dist}_{\bm{\sigma}}^k(\vect{X}_0) \mid \mathbb{X}'\,] \Delta_{d-1}^k ]}{d^k (\vol_d K_d)^k},
\end{equation}
where
\begin{equation}
    \Exx [\operatorname{dist}_{\bm{\sigma}}^k(\vect{X}_0) \mid \mathbb{X}'\,] = \frac{1}{\vol_d K_d} \int_{K_d} \operatorname{dist}_{\bm{\sigma}}^k(\vect{x}_0) \lambda_d(\dd\vect{x}_0)
\end{equation}
is the $k$-th distance moment from (fixed) $\bm{\sigma}$. If $\bm{\sigma}$ is parametrised Cartisianely, that means by $\bm{\eta} = (\eta_1,\ldots,\eta_d)^\top$ such that $\vect{x} \in \bm{\sigma}\Leftrightarrow\bm{\eta}^\top \vect{x} = 1$, we may write
\begin{equation}\label{Eq:DistEta}
\operatorname{dist}_{\bm{\sigma}}(\vect{x}_0) = |\bm{\eta}^\top \vect{x}_0-1|/\|\bm{\eta}\|
\end{equation}
and thus
\begin{equation}
    \Exx [\operatorname{dist}_{\bm{\sigma}}^k(\vect{X}_0) \mid \mathbb{X}'\,] = \frac{1}{\|\bm{\eta}\|^k\vol_d K_d} \int_{K_d} |\bm{\eta}^\top \vect{x}_0-1|^k \lambda_d(\dd\vect{x}_0).
\end{equation}
Note that since $\Exx [\operatorname{dist}_{\bm{\sigma}}^k(\vect{X}_0) \mid \mathbb{X}'\,]$ is only a function of $\bm{\sigma}$, we may use Blaschke-Petkantschin formula in Cartesian parametrisation (Corollary \ref{BP:zeta}), that is
\begin{equation*}
\expe{g(\bm{\sigma}) \Delta_{d-1}^k} = (d-1)!(\vol_d K)^{k+1} \int_{\mathbb{R}^d\setminus K_d^\circ} v_{d-1}^{(k+1)}(\bm{\sigma}_{\! K_d}) \zeta_d^{d+k+1}(\bm{\sigma}) g(\bm{\sigma}) \|\bm{\eta}\|^k \lambda_d(\dd \bm{\eta}),
\end{equation*}
where $\bm{\sigma}_{K_d} = \bm{\sigma}\cap K_d$. Selecting $g(\bm{\sigma}) = \Exx [\operatorname{dist}_{\bm{\sigma}}^k(\vect{X}_0) \mid \mathbb{X}'\,]$ and by definition of $\iota^{(k)}_d(\bm{\sigma})$, Equation \eqref{Eq:CanonDec} then becomes the desired assertion of the theorem.
\end{proof}
\begin{remark}
Note that from the theorem above we can obtain the limiting behaviour of $v_d^{(k)}(K_d)$ when $k \to (-1)^+$. Let's take a closer look on $\iota^{(k)}_d(\bm{\sigma})$ as in is the only singular term in $v_d^{(k)}(K_d)$. Let $y \in \bm{\sigma}$ be the closest point on $\bm{\sigma}$ to the origin, we have $h = \| \vect{y} \| = 1/\| \bm{\eta}\|$. Let $\vect{e}'_i$ form an orthonormal basis on $\bm{\sigma}$. Then, by Fubini's theorem with $\vect{u} = u_1 \vect{u}_1' + \ldots + u_{d-1} \vect{u}_{d-1}'$
\begin{equation}
    \iota^{(k)}_d(\bm{\sigma}) \!=\!\!  \!\int_{h-\varepsilon}^{h + \varepsilon}  \!\!\!\int_{\bm{\sigma} \cap K_d}  \!\!\!\!\!\| \bm{\eta} \|^k |x\!-\!h|^k \lambda_{d-1}(\dd\vect{u}) \dd x  \!+\!  O(1) \!=\! \frac{2\vol_{d-1} (\bm{\sigma}_{\! K_d})}{\|\bm{\eta}\|(k\!+\!1)} \!+\! O(1) \!=\! \frac{2\zeta_d(\bm{\sigma})}{(k\!+\!1)\vol_d K_d} \!+\! O(1)
\end{equation}
from which immediately
\begin{equation}
\lim_{k \to (-1)^+} (1+k)v_d^{(k)}(K_d) =  \frac{2\, d!}{\vol_d K_d}\int_{\mathbb{R}^d\setminus K_d^\circ} \zeta_d^{d+1}(\bm{\sigma}) \lambda_d(\dd \bm{\eta}).
\end{equation}
\end{remark}

\begin{remark}
Writing $\vect{x} = \vect{y} + t \hat{\bm{\eta}}+ \vect{u}$ and integrating over $\vect{u}$, we obtain an equivalent expression for $\iota_d^{(k)}(\bm{\sigma})$ in terms of geometric quantities, 
\end{remark}
\begin{equation}
    \iota_d^{(k)}(\bm{\sigma}) =\|\bm{\eta}\|^k \int_{-\infty}^\infty \vol_{d-1}((\bm{\sigma}+t\hat{\bm{\eta}})\cap K_d) |t|^k \ddd t
\end{equation}

\begin{remark}
By affine invariancy of volumetric moments and when $K_d = P_d$ is a polytope, we may take advantage of its symmetries to obtain
\begin{equation}
v_d^{(k)}(P_d) = \sum_{C \in \mathcal{C}(P_d)} w_C \, v_d^{(k)}(P_d)_C,
\end{equation}
where
\begin{equation}
 v_d^{(k)}(P_d)_C = \frac{(d-1)!}{ d^k}\int_{(\mathbb{R}^d\setminus P_d^\circ)_C} v_{d-1}^{(k+1)}(\bm{\sigma}\cap K_d) \,\zeta_d^{k+d+1}(\bm{\sigma}) \iota^{(k)}_d(\bm{\sigma}) \lambda_d(\dd \bm{\eta}).
\end{equation}
It may seem that finding the precise integration domains $(\mathbb{R}^d\setminus P_d^\circ)_C$ for various configurations is complicated. In fact, it is relatively easy. Recall that a configuration $C = P_d(S)$ is defined by the property of $\bm{\sigma}$ separating some given vertices from the set $S$ out of the set of all vertices $V$ of the polytope $P_d$.  The domain $(\mathbb{R}^d\setminus P_d^\circ)_C$ in $(\eta_1,\ldots,\eta_d)^\top$ is then the unique solution of the following inequalities
\begin{equation}\label{Eq:Separ}
\bm{\eta}^\top \vect{v} < 1 \text{ for all } \vect{v} \in S, \qquad \bm{\eta}^\top \vect{v} > 1\text{ for all } \vect{v} \in V\setminus S
\end{equation}
% Assume that one vertex $v_0$ of $P_d$ in $S$ is the origin, then $\bm{\eta}^T v_0 - 1 = -1 <0$.
or inequalities with $<$, $>$ flipped (we then take the union of those two options). Note that $\bm{\sigma}$ always separates $P_d$ into disjoint union $P_d^+ \sqcup P_d^-$, where
\begin{equation}\label{Eq:Pplusminus}
    P_d^+ = \{\vect x \in P_d \, | \, \bm{\eta}^\top \vect{x} < 1\}, \qquad P_d^- = \{\vect x \in P_d \, | \, \bm{\eta}^\top \vect{x} > 1\},
\end{equation}
which means that the computation of $\iota_d^{(k)}(\bm{\sigma})$ is also straightforward as
\begin{equation}
\iota^{(k)}_d(\bm{\sigma}) = \int_{P_d^+}  (1-\bm{\eta}^\top \vect{x})^k \lambda_d(\dd\vect{x}) + \int_{P_d^-} (\bm{\eta}^\top \vect{x}-1)^k \lambda_d(\dd\vect{x})
\end{equation}
for any real $k>-1$. When $k$ is an integer, let us denote
\begin{equation}
\iota^{(k)}_d(\bm{\sigma})_\mathrm{N} = \int_{P_d}  (\bm{\eta}^\top \vect{x}-1)^k \lambda_d(\dd\vect{x}),
\end{equation}
then, when $k$ is even, we have $\iota^{(k)}_d(\bm{\sigma})= \iota^{(k)}_d(\bm{\sigma})_\mathrm{N}$. For any integer $k$, we get by inclusion/exclusion
\begin{equation}
\iota^{(k)}_d(\bm{\sigma}) \!=\! \iota^{(k)}_d(\bm{\sigma})_\mathrm{N} \!-\! (1\!-\!(-1)^k)\!\! \int_{P_d^+} \!\!\! (\bm{\eta}^\top \vect{x}\!-\!1)^k \lambda_d(\dd\vect{x})\! =\! (-1)^k \iota^{(k)}_d(\bm{\sigma})_\mathrm{N} \!+\! (1-(-1)^k) \!\!\int_{P_d^-}\!\!\!  (\bm{\eta}^\top \vect{x}\!-\!1)^k \lambda_d(\dd\vect{x}).
\end{equation}
Lastly, note that $\zeta_d(\bm{\sigma})$ is a rational function of $\bm{\eta}$. To see this, we know that $\vol_d P_d^+$ is a rational function in $(\eta_1,\ldots,\eta_d)^\top$. From homogeneity (Remark \ref{Rem:HomZeta}),
\begin{equation}\label{Eq:ZetaEta}
\zeta_d(\bm{\sigma}) = - \frac{1}{\vol_d P_d} \sum_{j=1}^d \eta_j\frac{\partial\vol_d P_d^+}{\partial \eta_j} = \frac{1}{\vol_d P_d} \sum_{j=1}^d \eta_j\frac{\partial\vol_d P_d^-}{\partial \eta_j}
\end{equation}
which is also rational since differentiation preserves rationality.
\end{remark}

\begin{remark}
Fundamental Lemma of Convex Geometry tells us that a polytope is described equivalently either by linear inequalities or as a convex hull of its vertices (H- and V- representation equivalence). Hence, for example by linear programming techniques, we can deduce the vertices of $P_d^+$ from the inequalities for $P_d^+$ and vice versa. The same applies for the polytope $\bm{\sigma}\cap P_d$ whose number of vertices is $n_C$ by definition.
\end{remark}

\begin{example}\label{Ex:v1k1}
Consider a trivial example of $v_1^{(k)}(T_1)$, that is the $k$-th moment of a random line length. Parametrising $\bm{\eta} = (a)^\top$, $a>1$, we get $\zeta_1(\bm{\sigma}) = 1/a$,
\begin{equation}
\iota_1^{(k)}(\bm{\sigma}) = \int_0^1 |ax-1|^k \dd x = \frac{(a-1)^{k+1}+1}{a (1+k)}
\end{equation}
and thus by Theorem \ref{Thm:Canon} with $\mathbb{R}^1\setminus T_1^\circ = (1,\infty)$ and $\lambda_1(\dd\bm{\eta}) = \dd a$,
\begin{equation}
v_1^{(k)}(T_1) = \int_1^\infty \frac{(a-1)^{k+1}+1}{a^{k+3} (k+1)} \ddd a = \frac{2}{(1+k)(2+k)}.
\end{equation}
\end{example}

\subsubsection{Triangle area moments}
As a toy model, we deduce the volumetric moments $v_2^{(k)}(T_2)$ from the canonical section integral formula (Theorem \ref{Thm:Canon}). We obtain values shown in Table \ref{Tab:v2kT2}.
\begin{table}[H]
    \centering
    \begin{tabular}{|c|c|c|c|c|c|c|c|c|c|c|}
        \hline
        $k$ & $0$ & $1$ & $2$ & $3$ & $4$ & $5$ & $6$ & $7$ & $8$ & $9$ \\
        \hline
        \ru{1.3}
        $v_2^{(k)}(T_2)$ & $1$ & $\frac{1}{12}$ & $\frac{1}{72}$ & $\frac{31}{9000}$ & $\frac{1}{900}$ &
    $\frac{1063}{2469600}$ &
    $\frac{403}{2116800}$ &
    $\frac{211}{2268000}$ &
    $\frac{13}{2646000}$ & $\frac{2593}{93915360}$ \\[0.5em]
        \hline
    \end{tabular}
    \caption{Volumetric moments $v_2^{(k)}(T_2)$ (triangle area moments)}
    \label{Tab:v2kT2}
\end{table}

First, from affine invariancy, $v_2^{(k)}(T_2)$ must be the same as $v_2^{(k)}(\mathbb{T}_2)$, where
\begin{equation}
\mathbb{T}_2=\hull(\vect{0},\vect{e}_1,\vect{e}_2) = \hull([0,0],[1,0],[0,1])
\end{equation}
is the canonical triangle. Trivialy, we have $\vol_2 \mathbb{T}_2 = 1/2!=1/2$. Let $\bm{\eta}=(a,b)^\top$ be the Cartesian parametrisation of the line $\bm{\sigma} \in \mathbb{A}(2,1)$ such that $\vect{x}\in \bm{\sigma} \Leftrightarrow \bm{\eta}^\top\vect{x} = 1$. We have $\|\bm{\eta}\| = \sqrt{a^2+b^2}$. Based on symmetries $\mathcal{G}(T_2)$, there is only one realisable configuration. Moreover, thanks to affine invariancy, we can consider the only configuration $\mathrm{I}$ in $\mathcal{C}(\mathbb{T}_2)$. Table \ref{tab:TriConfs} shows specifically which sets $S$ of vertices are separated by a cutting plane $\bm{\sigma}$. The corresponding configurations in $T_2$ are shown in Figure \ref{fig:T2Con}.

\begin{minipage}[b]{0.22\textwidth}
\begin{table}[H]
    \centering
\begin{tabular}{|c|c|}
\hline
 $C$ & $\mathrm{I}$\\
 \hline
 \ru{1.0}$S$ & $[0,0]$ \\[0.1em]
 \hline 
 $w_C$ & $3$\\
 \hline
\end{tabular}
\vspace{1.5em}
\caption{Configurations $\mathcal{C}(\mathbb{T}_2)$.}
    \label{tab:TriConfs}
\end{table}
\end{minipage}
\hfill
\begin{minipage}[b]{0.28\textwidth}
\begin{figure}[H]
    \centering     \includegraphics[width=0.65\textwidth]{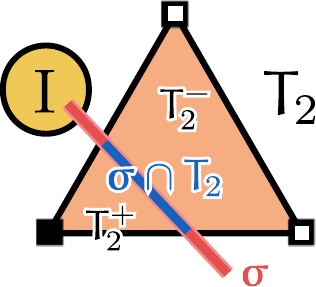}
    \vspace{-1em}
    \caption{Configurations $\mathcal{C}(T_2)$}
    \label{fig:T2Con}
\end{figure}
\end{minipage}
\hfill
\begin{minipage}[b]{0.28\textwidth}
\begin{figure}[H]
    \centering     \includegraphics[width=0.65\textwidth]{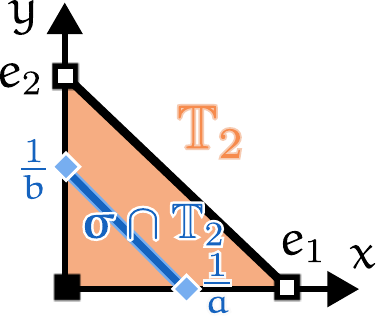}
    \caption{Configuration $\mathrm{I}$ in $\mathcal{C}(\mathbb{T}_2)$}
    \label{fig:bbT2ConI}
\end{figure}
\end{minipage}
\vspace{1em}

By Theorem \ref{Thm:Canon} and for any $C \in \mathcal{C}(\mathbb{T}_2)$,
\begin{equation}\label{Eq:SecIntArea}
 v_2^{(k)}(\mathbb{T}_2)_C = \frac{1}{ 2^k}\int_{(\mathbb{R}^2\setminus \mathbb{T}_2^\circ)_C} v_1^{(k+1)}(\bm{\sigma}\cap \mathbb{T}_2) \,\zeta_2^{k+3}(\bm{\sigma}) \iota^{(k)}_2(\bm{\sigma}) \lambda_2(\dd \bm{\eta}),
\end{equation}
where
\begin{equation}
    \zeta_2(\bm{\sigma}) = \frac{\vol_1(\bm{\sigma}\cap \mathbb{T}_2)}{\|\bm{\eta}\|\vol_2 \mathbb{T}_2}, \qquad \iota^{(k)}_2(\bm{\sigma}) = \int_{\mathbb{T}_2} |\bm{\eta}^\top \vect{x}-1|^k \lambda_2(\dd\vect{x}).
\end{equation}
To ensure $\bm{\sigma}$ separates only the point $[0,0]$ in Configuration $\mathrm{I}$, we must force the plane intersection coordinates $\frac{1}{a},\frac{1}{b}$ to lie in the interval $(0,1)$. Or, by Equation \eqref{Eq:Separ}, we get $a>1$ and $b>1$ directly. Any way, that means $(\mathbb{R}^2\setminus \mathbb{T}_2^\circ)_\mathrm{I} = (1,\infty)^2$ is our integration domain in $a,b$. See Figure \ref{fig:bbT2ConI}. Denote
\begin{equation}
    \mathbb{T}_2^{ab} = \hull([0,0],[1/a,0],[0,1/b]).
\end{equation}
The line $\bm{\sigma}$ splits $\mathbb{T}_2$ into disjoint union of two domains $\mathbb{T}_2^+\sqcup \mathbb{T}_2^-$, where the one closer to the origin is precisely $\mathbb{T}_2^+ = \mathbb{T}_2^{ab}$. Therefore,
\begin{equation}
\iota^{(k)}_2(\bm{\sigma}) = \int_{\mathbb{T}_2^{ab}}  (1-\bm{\eta}^\top \vect{x})^k \lambda_2(\dd\vect{x}) + \int_{\mathbb{T}_2\setminus\mathbb{T}_2^{ab}} (\bm{\eta}^\top \vect{x}-1)^k \lambda_2(\dd\vect{x}).
\end{equation}
This integral is easy to compute. In fact, for any real $k>-1$, we get
\begin{equation}\label{Eq:iota2k}
\iota^{(k)}_2(\bm{\sigma}) = \frac{b (a-1)^{k+2}-a (b-1)^{k+2}+a-b}{a b (a-b) (1+k) (2+k)}.
\end{equation}
Note that $\bm{\sigma} \cap \mathbb{T}_2 = \hull([1/a,0],[0,1/b])$ and thus
\begin{equation}
\vol_1(\bm{\sigma}\cap\mathbb{T}_2) = \frac{\sqrt{a^2+b^2}}{ab} = \frac{\|\bm{\eta}\|}{ab}
\end{equation}
and hence
\begin{equation}\label{Eq:ZetaTriangle}
    \zeta_2(\bm{\sigma}) = \frac{\vol_1(\bm{\sigma}\cap \mathbb{T}_2)}{\|\bm{\eta}\|\vol_2 \mathbb{T}_2} = \frac{2}{ab}.
\end{equation}
Moreover, by affine invariancy of volumetric moments and using Example \ref{Ex:v1k1},
\begin{equation}
v_1^{(k+1)}(\bm{\sigma}\cap \mathbb{T}_2) = v_1^{(k+1)}(T_1) = \frac{2}{(2+k)(3+k)}.
\end{equation}
By Equation \eqref{Eq:DecoCon} and by affine invariancy,
\begin{equation}
v_2^{(k)}(T_2) = \sum_{C \in \mathcal{C}(T_2)} w_C \, v_2^{(k)}(T_2)_C = 3 v_2^{(k)}(\mathbb{T}_2)_{\mathrm{I}},
\end{equation}
from which, we get by Equation \eqref{Eq:SecIntArea} for any real $k>-1$,
\begin{equation}
     v_2^{(k)}(T_2) = 48 \int_1^{\infty } \int_1^{\infty }\frac{b (a-1)^{2+k}-a (b-1)^{2+k}+a-b}{a^{k+4} b^{k+4}
   (a-b)(1+k) (2+k)^2 (3+k)}\ddd a \dd b.
\end{equation}
Let $a=1/x$ and $b=1/y$, then, after some simple manipulations,
\begin{equation}
     v_2^{(k)}(T_2) = \frac{48}{(1+k) (2+k)^2 (3+k)} \int_0^1 \int_0^1\frac{(1-x)^{2+k} \left(x^{2+k}-y^{2+k}\right)}{x-y}\ddd x  \ddd y
\end{equation}
for any real $k>-1$. This integral can be computed explicitly when $k$ is an integer. Dividing the numerator by $x-y$, we get
\begin{equation}
\begin{split}
     v_2^{(k)}(T_2) & = \frac{48}{(1+k) (2+k)^2 (3+k)} \sum_{j=0}^{k+1} \int_0^1 \int_0^1(1-x)^{2+k}x^j y^{k-j+1}\ddd x  \ddd y\\
     & = \frac{48}{(1+k) (2+k)^2 (3+k)} \sum_{j=0}^{k+1} \frac{1}{k-j+2}\int_0^1(1-x)^{2+k} x^j\ddd x,
\end{split}
\end{equation}
which is, of course, the Beta integral. Therefore, for any non-negative integer $k$,
\begin{equation}
v_2^{(k)}(T_2) = \frac{48 \, k!}{(2+k) (3+k)} \sum _{j=0}^{k+1} \frac{j!}{(k-j+2) (k+j+3)!}.
\end{equation}
This result is not new, in fact, it has been derived several times (although known in different forms), see Reed \cite{reed1974random}, Mathai \cite[p.~391]{mathai1999introduction} or Maesumi \cite{maesumi2019triangle}. 

\begin{comment}
Finally, let us mention that that the even moments are easy to obtain by integrating even powers of the area over the unit triangle, see Figure \ref{fig:secT3ConI} below.
\begin{figure}[H]
    \centering     \includegraphics[width=0.35\textwidth]{Pictures/pic_secT3ConI.pdf}
    \caption{Random triangle area $\Delta_2$ written as a determinant}
    \label{fig:secT3ConI}
\end{figure}
In general, writing the expectation as an integral, we have for even $k$ and $\vect{x}_i = (x_i,y_i)^\top, i=1,2,3$,
\begin{equation}
v_2^{(k)}(\mathbb{T}_2) = 2^{k+3} \int_{\mathbb{T}_2^3} \Delta_2^{k} \ddd \vect{x}_0\dd \vect{x}_1 \dd \vect{x}_2.
\end{equation}
\end{comment}

\subsubsection{Square area moments}
As another example, we deduce the volumetric moments $v_2^{(k)}(C_2)$ from Theorem \ref{Thm:Canon}. We obtain values shown in Table \ref{Tab:v2kC2}.
\begin{table}[!tbh]
    \centering
    \begin{tabular}{|c|c|c|c|c|c|c|c|c|c|}
        \hline
        $k$ & $1$ & $2$ & $3$ & $4$ & $5$ & $6$ & $7$ & $8$ & $9$ \\
        \hline
        \ru{1.3}
        $v_2^{(k)}(C_2)$ & $\frac{11}{144}$ & $\frac{1}{96}$ & $\frac{137}{72000}$ & $\frac{1}{2400}$ & $\frac{363}{3512320}$ & $\frac{761}{27095040}$ & $\frac{7129}{870912000}$ & $\frac{61}{24192000}$ & $\frac{83711}{103038566400}$ \\[0.5em]
        \hline
    \end{tabular}
    \caption{Volumetric moments $v_2^{(k)}(C_2)$ (square area moments)}
    \label{Tab:v2kC2}
\end{table}

We may parametrise $C_2$ with $\vol_2 C_2 = 1$ as
\begin{equation}
C_2=\hull([0,0],[1,0],[0,1],[1,1]),
\end{equation}
Let $\bm{\eta}=(a,b)^\top$ be the Cartesian parametrisation of the line $\bm{\sigma} \in \mathbb{A}(2,1)$ such that $\vect{x}\in \bm{\sigma} \Leftrightarrow \bm{\eta}^\top\vect{x} = 1$. We have $\|\bm{\eta}\| = \sqrt{a^2+b^2}$. Based on symmetries $\mathcal{G}(C_2)$, there are two configurations. Table \ref{tab:SquareConfs} shows specifically which sets $S$ of vertices are separated by a cutting plane $\bm{\sigma}$ in which configurations in our local representation of $C_2$ above. Note that there is an ambiguity how to select those vertices as long it is the same configuration.
\begin{table}[htb]
    \centering
\begin{tabular}{|c|c|c|c|}
\hline
 $C$ & $\mathrm{I}$ & $\mathrm{II}$\\
 \hline
 $S$ & $[0,0]$ & \begin{tabular}{c} $[0,0]$ \\ $[0,1]$ \end{tabular} \\
 \hline 
 $w_C$ & $4$ & $2$\\
 \hline
\end{tabular}
    \caption{Configurations $\mathcal{C}(C_2)$ in a local representation.}
    \label{tab:SquareConfs}
\end{table}

By Theorem \ref{Thm:Canon} and for any $C \in \mathcal{C}(C_2)$,
\begin{equation}\label{Eq:SecIntAreaSq}
 v_2^{(k)}(C_2)_C = \frac{1}{ 2^k}\int_{(\mathbb{R}^2\setminus C_2^\circ)_C} v_1^{(k+1)}(\bm{\sigma}\cap C_2) \,\zeta_2^{k+3}(\bm{\sigma}) \iota^{(k)}_2(\bm{\sigma}) \lambda_2(\dd \bm{\eta}),
\end{equation}
where
\begin{equation}
    \zeta_2(\bm{\sigma}) = \frac{\vol_1(\bm{\sigma}\cap C_2)}{\|\bm{\eta}\|\vol_2 C_2}, \qquad \iota^{(k)}_2(\bm{\sigma}) = \int_{C_2} |\bm{\eta}^\top \vect{x}-1|^k \lambda_2(\dd\vect{x}).
\end{equation}

\subsubsection*{Configuration I}
By Equation \eqref{Eq:Separ}, we get the following set of inequalities which ensure $\bm{\sigma}$ separates only the point $[0,0]$,
\begin{equation}
   0 < 1, \qquad a > 1, \qquad b > 1, \qquad a + b > 1,
\end{equation}
hence, our $a,b$ integration domain is $(\mathbb{R}^2\setminus C_2^\circ)_\mathrm{I} = (1,\infty)^2$. Denote
\begin{equation}
    \mathbb{T}_2^{ab} = \hull([0,0],[1/a,0],[0,1/b]),
\end{equation}
then the line $\bm{\sigma}$ splits $C_2$ into disjoint union of two domains $C_2^+\sqcup C_2^-$, where the one closer to the origin is precisely $C_2^+ = \mathbb{T}_2^{ab}$. Therefore,
\begin{equation}
\iota^{(k)}_2(\bm{\sigma}) = \int_{\mathbb{T}_2^{ab}}  (1-\bm{\eta}^\top \vect{x})^k \lambda_2(\dd\vect{x}) + \int_{C_2\setminus\mathbb{T}_2^{ab}} (\bm{\eta}^\top \vect{x}-1)^k \lambda_2(\dd\vect{x}).
\end{equation}
This integral is easy to compute. In fact, for any real $k>-1$, we get
\begin{equation}
\iota^{(k)}_2(\bm{\sigma}) = \frac{(a+b-1)^{k+2}-(a-1)^{k+2}-(b-1)^{k+2}+1}{a b (k+1) (k+2)}.
\end{equation}
By Equation \eqref{Eq:ZetaTriangle} from the $P_2 = \mathbb{T}_2$ case,
\begin{equation}
    \zeta_2(\bm{\sigma}) = \frac{\vol_1(\bm{\sigma}\cap C_2)}{\|\bm{\eta}\|\vol_2 C_2} = \frac{1}{ab}
\end{equation}
and by affine invariancy, as $\bm{\sigma} \cap C_2$ is a line segment,
\begin{equation}
v_1^{(k+1)}(\bm{\sigma}\cap C_2) = v_1^{(k+1)}(T_1) = \frac{2}{(2+k)(3+k)}.
\end{equation}
from which, we get by Equation \eqref{Eq:SecIntAreaSq} for any real $k>-1$,
\begin{equation}
     v_2^{(k)}(C_2)_\mathrm{I} = 2^{1-k} \int_1^{\infty } \int_1^{\infty }\frac{(a+b-1)^{k+2}-(a-1)^{k+2}-(b-1)^{k+2}+1}{a^{k+4} b^{k+4}
   (1+k) (2+k)^2 (3+k)}\ddd a  \ddd b.
\end{equation}
Integrating out $b$ and substituting $a=1/x$ and after some simplifications, we get
\begin{equation}
v_2^{(k)}(C_2)_\mathrm{I} = \frac{2^{1-k}}{(1+k) (2+k)^2 (3+k)^2} \int_0^1 \frac{1-x^{2+k}}{1-x} \ddd x,
\end{equation}
for any real $k>-1$. When $k$ is an integer, we get
\begin{equation}
v_2^{(k)}(C_2)_\mathrm{I} = \frac{16 H_{k+2}}{(1+k) (2+k)^2 (3+k)^2},
\end{equation}
where $H_k = \sum_{j=1}^k 1/j$ is the $k$-th \emph{harmonic number}\index{harmonic number}.

\subsubsection*{Configuration II}
By Equation \eqref{Eq:Separ}, we get the following set of inequalities which ensure $\bm{\sigma}$ separates points $[0,0]$ and $[0,1]$,
\begin{equation}
   0 < 1, \qquad a > 1, \qquad b < 1, \qquad a + b > 1,
\end{equation}
however, by symmetry, we may additionally require $b>0$. In fact, both options $b>0$ and $b<0$ give the same factor since they correspond to two possibilities where $\bm{\sigma}$ hits $\mathcal{A}([0,0],[0,1])$. Therefore we only consider the following integration half-domain (indicated by $*$)
\begin{equation}
(\mathbb{R}^2\setminus C_2^\circ)^{*}_\mathrm{II} = (1,\infty)\times (0,1)
\end{equation}
and in the end multiply the result twice. The plane $\bm{\sigma}$ splits $C_3$ into disjoint union of two domains $C_3^+\sqcup C_3^-$, where the one closer to the origin can be described as
\begin{equation}
    C_2^+ = \hull\left(\left[0,0\right],\left[\frac{1}{a},0\right],\left[\frac{1-b}{a},1\right],\left[0,1\right]\right),
\end{equation}
from which, by elementary geometry $\vol_2 C_2^+ = (2-b)/(2a)$ and as a consequence of Equation \eqref{Eq:ZetaEta},
\begin{equation}
\zeta_2(\bm{\sigma}) = - a \frac{\partial}{\partial a}\left(\frac{2-b}{2a}\right) - b \frac{\partial}{\partial b}\left(\frac{2-b}{2a}\right) =  \frac{1}{a}.
\end{equation}
Next, again, the following integrals
\begin{equation}
\iota^{(k)}_2(\bm{\sigma}) = \int_{C_2^+}  (1-\bm{\eta}^\top \vect{x})^k \lambda_2(\dd\vect{x}) + \int_{C_2\setminus C_2^+} (\bm{\eta}^\top \vect{x}-1)^k \lambda_2(\dd\vect{x}).
\end{equation}
are easy to compute for any real $k>-1$, we get
\begin{equation}
\iota^{(k)}_2(\bm{\sigma}) = \frac{(a+b-1)^{k+2}-(a-1)^{k+2}-(1-b)^{k+2}+1}{a b (k+1) (k+2)}.
\end{equation}
and by affine invariancy, as $\bm{\sigma} \cap C_2$ is again a line segment,
\begin{equation}
v_1^{(k+1)}(\bm{\sigma}\cap C_2) = v_1^{(k+1)}(T_1) = \frac{2}{(2+k)(3+k)}.
\end{equation}
from which, we get by Equation \eqref{Eq:SecIntAreaSq} for any real $k>-1$ (counted twice!),
\begin{equation}
     v_2^{(k)}(C_2)_\mathrm{II} = \frac{4}{2^k} \int_0^{\infty } \int_1^{\infty }\frac{(a+b-1)^{k+2}-(a-1)^{k+2}-(1-b)^{k+2}+1}{a^{k+4} b
   (1+k) (2+k)^2 (3+k)}\ddd a  \ddd b.
\end{equation}
Integrating out $a$ and after some simplifications, we get
\begin{equation}
v_2^{(k)}(C_2)_\mathrm{II} = \frac{2^{3-k}}{(1+k) (2+k)^2 (3+k)^2} \int_0^1 \frac{1-b^{2+k}}{1-b} \ddd b,
\end{equation}
for any real $k>-1$. When $k$ is an integer, we get
\begin{equation}
v_2^{(k)}(C_2)_\mathrm{II} = \frac{2^{3-k} H_{k+2}}{(1+k) (2+k)^2 (3+k)^2}.
\end{equation}

\subsubsection*{Contribution from all configurations
}
By Equation \eqref{Eq:DecoCon},
\begin{equation}
v_2^{(k)}(C_2) = \sum_{C \in \mathcal{C}(C_2)} w_C \, v_2^{(k)}(C_2)_C = 4 v_2^{(k)}(C_2)_{\mathrm{I}} + 2 v_2^{(k)}(C_2)_{\mathrm{II}},
\end{equation}
which gives for any real $k>-1$,
\begin{equation}
     v_2^{(k)}(C_2)= \frac{24}{2^k (1+k) (2+k)^2 (3+k)^2} \int_0^1 \frac{1-x^{k+2}}{1-x} \ddd x.
\end{equation}
For $k$ being an integer, we get
\begin{equation}
v_2^{(k)}(C_2)= \frac{24 H_{k+2}}{2^k (1+k) (2+k)^2 (3+k)^2} = \frac{24 \sum_{j=1}^{k+2} \frac{1}{j}}{2^k (1+k) (2+k)^2 (3+k)^2}.
\end{equation}
This result is also not new, see Reed \cite{reed1974random} or Henze \cite{henze1983random}.

\newpage
\section{Three-dimensional polytopes}
\subsection{Tetrahedron odd volumetric moments}
Let us investigate how we can obtain the volumetric moments $v_3^{(k)}(T_3)$. First, since $v_3^{(k)}(T_3)$ is an affine invariant, then it must be the same as $v_3^{(k)}(\mathbb{T}_3)$, where
\begin{equation}
\mathbb{T}_3=\hull(\vect{0},\vect{e}_1,\vect{e}_2,\vect{e}_3) = \hull([0,0,0],[1,0,0],[0,1,0],[0,0,1])
\end{equation}
is the canonical tetrahedron. We have $\vol_3 \mathbb{T}_3 = 1/3!=1/6$. Let $\bm{\eta}=(a,b,c)^\top$ be the Cartesian parametrisation of $\bm{\sigma} \in \mathbb{A}(3,2)$ such that $\vect{x}\in \bm{\sigma} \Leftrightarrow \bm{\eta}^\top\vect{x} = 1$. We have $\|\bm{\eta}\| = \sqrt{a^2+b^2+c^2}$. Based on symmetries $\mathcal{G}(T_3)$, there are two realisable configurations we need to consider (see its genealogy in Figure \ref{fig:TETRAHEDRON_GENEALOGY}). Moreover, thanks to affine invariancy, we can consider instead the two $\mathcal{C}(\mathbb{T}_3)$ configurations (see Table \ref{tab:TetraConfs} below, Figure shows the correspoding configurations on the non-deformed $T_3$).

\begin{minipage}[b]{0.45\textwidth}
\begin{table}[H]
    \centering
\begin{tabular}{|c|c|c|}
\hline
 $C$ & $\mathrm{I}$ & $\mathrm{II}$\\
 \hline
 \ru{1.0}$S$ & $[0,0,0]$ & \begin{tabular}{c} $[0,0,0]$ \\ $[0,0,1]$ \end{tabular} \\[0.1em]
 \hline 
 $w_C$ & $4$ & $3$\\
 \hline
\end{tabular}
\vspace{0.8em}
\caption{Configurations $\mathcal{C}(\mathbb{T}_3)$.}
    \label{tab:TetraConfs}
\end{table}
\end{minipage}
\hfill
\begin{minipage}[b]{0.50\textwidth}
\begin{figure}[H]
    \centering     \includegraphics[width=0.45\textwidth]{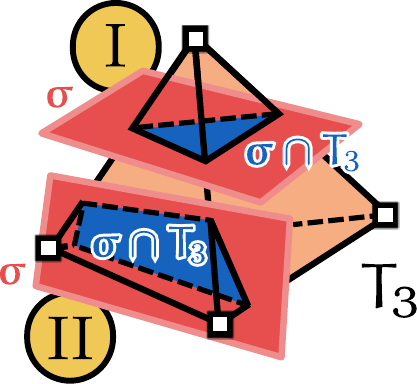}
    \vspace{-0.5em}
    \caption{Configurations $\mathcal{C}(T_3)$}
    \label{fig:T3Con}
\end{figure}
\end{minipage}
\vspace{1em}

By Theorem \ref{Thm:Canon} and for any $C \in \mathcal{C}(\mathbb{T}_3)$,
\begin{equation}\label{Eq:SecInt}
 v_3^{(k)}(\mathbb{T}_3)_C = \frac{2}{ 3^k}\int_{(\mathbb{R}^3\setminus \mathbb{T}_3^\circ)_C} v_2^{(k+1)}(\bm{\sigma}\cap \mathbb{T}_3) \,\zeta_3^{k+4}(\bm{\sigma}) \iota^{(k)}_3(\bm{\sigma}) \lambda_3(\dd \bm{\eta}),
\end{equation}
where
\begin{equation}
    \zeta_3(\bm{\sigma}) = \frac{\vol_2(\bm{\sigma}\cap \mathbb{T}_3)}{\|\bm{\eta}\|\vol_3 \mathbb{T}_3}, \qquad \iota^{(k)}_3(\bm{\sigma}) = \int_{\mathbb{T}_3} |\bm{\eta}^\top \vect{x}-1|^k \lambda_3(\dd\vect{x}).
\end{equation}
In order to distinguish between configurations, we also write $\zeta_3(\bm{\sigma})_C$ and $\iota^{(k)}_3(\bm{\sigma})_C$ instead of just $\zeta_3(\bm{\sigma})$ and $\iota^{(k)}_3(\bm{\sigma})$. Here, $C$ is only a subscript and does not imply any decomposition of those functions.

\begin{minipage}[b]{0.45\textwidth}
\begin{figure}[H]
    \centering     \includegraphics[width=0.65\textwidth]{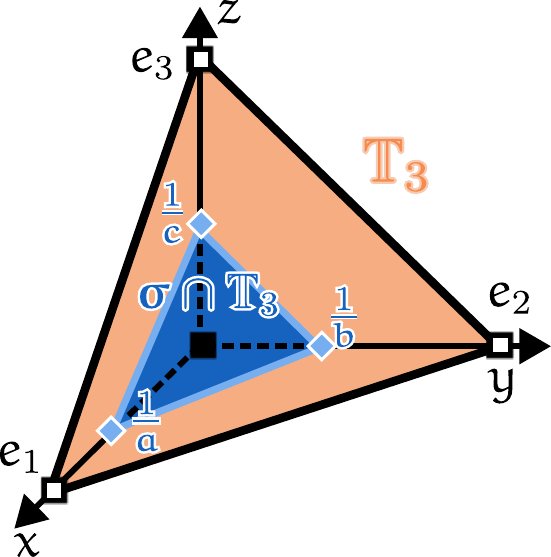}
    \caption{Configuration $\mathrm{I}$ in $\mathcal{C}(\mathbb{T}_3)$}
    \label{fig:bbT3ConI}
\end{figure}
\end{minipage}
\hfill
\begin{minipage}[b]{0.45\textwidth}
\begin{figure}[H]
    \centering     \includegraphics[width=0.65\textwidth]{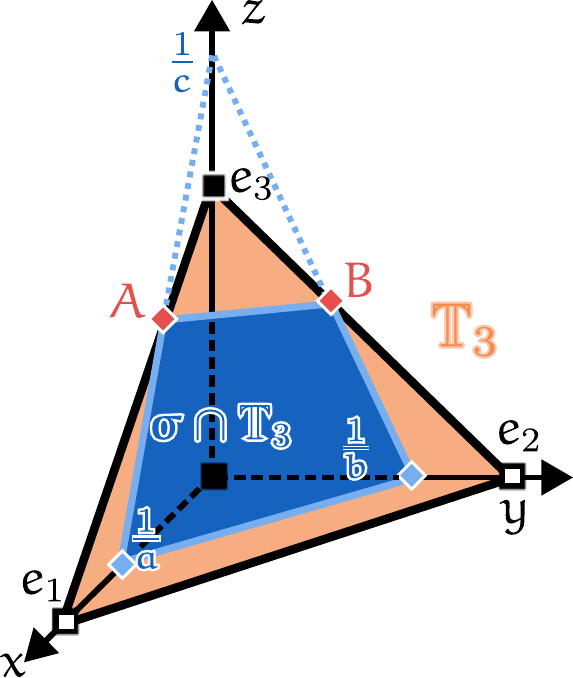}
    \caption{Configuration $\mathrm{II}$ in $\mathcal{C}(\mathbb{T}_3)$}
    \label{fig:bbT3ConII}
\end{figure}
\end{minipage}

\subsubsection{Configuration I}
To ensure $\bm{\sigma}$ separates only the point $[0,0,0]$, plugging the remaining points into Equation \eqref{Eq:Separ}, we get $a>1$, $b>1$ and $c>1$. That means $(\mathbb{R}^3\setminus \mathbb{T}_3^\circ)_\mathrm{I} = (1,\infty)^3$ is our integration domain in $a,b,c$. See Figure \ref{fig:bbT3ConI}. Denote
\begin{equation}
    \mathbb{T}_3^{abc} = \hull([0,0,0],[1/a,0,0],[0,1/b,0],[0,0,1/c]).
\end{equation}
The plane $\bm{\sigma}$ splits $\mathbb{T}_3$ into disjoint union of two domains $\mathbb{T}_3^+\sqcup \mathbb{T}_3^-$, where the one closer to the origin is precisely $\mathbb{T}_3^+ = \mathbb{T}_3^{abc}$. Therefore, by inclusion/exclusion,
\begin{equation}
\begin{split}
\iota^{(k)}_3(\bm{\sigma})_\mathrm{I}  & = \int_{\mathbb{T}_3^+}  (1-\bm{\eta}^\top \vect{x})^k \lambda_3(\dd\vect{x}) + \int_{\mathbb{T}_3^-} (\bm{\eta}^\top \vect{x}-1)^k \lambda_3(\dd\vect{x})\\
& = \int_{\mathbb{T}_3} (\bm{\eta}^\top \vect{x}-1)^k \lambda_3(\dd\vect{x}) - (1-(-1)^k) \int_{\mathbb{T}_3^{abc}} (\bm{\eta}^\top \vect{x}-1)^k \lambda_3(\dd\vect{x}).
\end{split}
\end{equation}
for any $k$ integer. These integrals are easy to compute. Mathematica Code \ref{code:iota} computes $\iota^{(k)}_3(\bm{\sigma})_\mathrm{I}$ for various values of $k$. Running the code for $k=1,2,3$, we get
\begin{align}
\label{Eq:iota31}\iota^{(1)}_3(\bm{\sigma})_{\mathrm{I}}& =\frac{1}{24} \left(\frac{2}{a b c}+a+b+c-4\right),\\
\iota^{(2)}_3(\bm{\sigma})_{\mathrm{I}}& =\frac{1}{60} \left(a^2+a b+b c+a c+b^2+c^2-5 a-5 b-5 c+10\right),\\
\begin{split}
   \label{Eq:iota33}\iota^{(3)}_3(\bm{\sigma})_{\mathrm{I}} & = \frac{1}{120} \bigg{(}\frac{2}{a b c}+15 a+15 b+15 c-6 a^2-6 b^2-6 c^2-6 a b-6 a c-6 b c\\
   & \hspace{3.4em}+a^2 b+a b^2+a^2 c+b^2 c+a c^2+b c^2+a^3+b^3+c^3+a b
   c-20\bigg{)}.
\end{split}
\end{align}
In fact, we can also deduce a general formula for $\iota_3^{(k)}(\bm{\sigma})$. It is easy to show
\begin{equation}
\begin{split}
\iota^{(k)}_3(\bm{\sigma})_\mathrm{I} & = \int_{\mathbb{T}_3} (ax_1+b x_2+c x_3-1)^k  - \frac{(-1)^k-1}{a b c} (1-x_1-x_2-x_3)^k \lambda_3(\dd\vect{x})\\
& \textstyle = \frac{ 1}{(k+1)(k+2)(k+3)} \left( \frac{1}{a b c}+\frac{(a-1)^{3+k}}{a (a-b) (a-c)}+\frac{(b-1)^{3+k}}{b (b-a)(b-c)}+\frac{(c-1)^{3+k}}{c (c-a) (c-b)} \right).
\end{split}
\end{equation}
Denote $T_2^{abc}$ as the triangle $\hull([1/a,0,0],[0,1/b,0],[0,0,1/c])$. Then
the intersection of the plane $\bm{\sigma}$ with $\mathbb{T}_3$ is precisely $T_2^{abc}$. That is,
\begin{equation}
    \bm{\sigma} \cap \mathbb{T}_3 = T_2^{abc}.
\end{equation}
By Equation \eqref{Eq:DistEta}, the distance from $T_2^{abc}$ to the origin is $\operatorname{dist}_{\bm{\sigma}}(\vect{0}) = 1/\| \bm{\eta} \|$. By base-height splitting,
\begin{equation}
    \frac{\vol_3 \mathbb{T}_3}{abc} = \vol_3\mathbb{T}_3^+ = \frac{1}{3} \operatorname{dist}_{\bm{\sigma}}(\vect{0}) \vol_2 T_2^{abc} = \frac{\vol_2(\bm{\sigma}\cap \mathbb{T}_3)}{3 \| \bm{\eta}\|},
\end{equation}
from which we immediately get
\begin{equation}
    \zeta_3(\bm{\sigma})_{\mathrm{I}} = \frac{\vol_2(\bm{\sigma}\cap \mathbb{T}_3)}{\|\bm{\eta}\|\vol_3 \mathbb{T}_3} = \frac{3}{abc}.
\end{equation}
Finally, by scale affinity (we have $n_\mathrm{I} = 3$),
\begin{equation}
v_2^{(k+1)}(\bm{\sigma}\cap \mathbb{T}_3) = v_2^{(k+1)}(T_2^{abc}) = v_2^{(k+1)}(T_2),
\end{equation}
which implies for $k=1,2,3$ that (see Table \ref{Tab:v2kT2})
\begin{equation}
v_2^{(2)}(\bm{\sigma}\cap \mathbb{T}_3) = \frac{1}{72}, \qquad v_2^{(3)}(\bm{\sigma}\cap \mathbb{T}_3) = \frac{31}{9000}, \qquad v_2^{(4)}(\bm{\sigma}\cap \mathbb{T}_3) = \frac{1}{900}.
\end{equation}

Putting everything into the integral in Equation \eqref{Eq:SecInt}, we get when $k=1$,
\begin{align}
     v_3^{(1)}(\mathbb{T}_3)_\mathrm{I} & = \frac{3}{32} \int_1^\infty  \int_1^\infty  \int_1^\infty \frac{2+abc(a+b+c-4)}{a^6 b^6 c^6} \ddd a \ddd b\ddd c = \frac{3}{2000}.
\end{align}
For higher values of $k$, we get
\begin{equation}
\begin{aligned}
   & v_3^{(2)}(\mathbb{T}_3)_\mathrm{I} \!=\! \tfrac{279}{4000000}, && v_3^{(3)}(\mathbb{T}_3)_\mathrm{I} \!=\! \tfrac{37193}{6174000000}, && v_3^{(4)}(\mathbb{T}_3)_\mathrm{I} \!=\! \tfrac{681383}{847072800000}, && v_3^{(5)}(\mathbb{T}_3)_\mathrm{I} \!=\! \tfrac{3674957}{25092716544000}.
\end{aligned}   
\end{equation}

\subsubsection{Configuration II}
In this scenario, $\bm{\sigma}$ separates two points $[0,0,0]$ and $[0,0,1]$ from $\mathbb{T}_3$. By Equation \eqref{Eq:Separ}, we get $a>1$, $b>1$ and $c<1$. We can split the condition for $c$ into to cases: either $0<c<1$ or $c<0$. In fact, both options give the same factor since they are symmetrical as they correspond to two possibilities where $\bm{\sigma}$ might intersect $\mathcal{A}([0,0,0],[0,0,1])$. Therefore we only consider the integration half-domain (indicated by $*$)
\begin{equation}
(\mathbb{R}^3\setminus \mathbb{T}_3^\circ)^{*}_\mathrm{II} = (1,\infty)^2 \times (0,1)
\end{equation}
and in the end multiply the result twice. See Figure \ref{fig:bbT3ConII}. The plane $\bm{\sigma}$ splits $\mathbb{T}_3$ into disjoint union of two domains $\mathbb{T}_3^+\sqcup \mathbb{T}_3^-$, where $\mathbb{T}_3^+$ being the one closer to the origin. Denote
\begin{align}
    \mathbb{T}_3^{abc} = & \hull\left(\left[0,0,0\right],\left[\frac{1}{a},0,0\right],\left[0,\frac{1}{b},0\right],\left[0,0,\frac{1}{c}\right]\right),\\
    \mathbb{T}_{3*}^{abc} = & \hull\left(\left[0,0,1\right],\left[\frac{1-c}{a-c},0,\frac{a-1}{a-c}\right],\left[0,\frac{1-c}{b-c},\frac{b-1}{b-c}\right],\left[0,0,\frac{1}{c}\right]\right),
\end{align}
then we can write
\begin{equation}
\mathbb{T}_3^+ = \mathbb{T}_{3}^{abc}\setminus \mathbb{T}_{3*}^{abc} = \hull\big{(}[0,0,0],[\tfrac{1}{a},0,0],[0,\tfrac{1}{b},0],[\tfrac{1-c}{a-c},0,\tfrac{a-1}{a-c}],[0,\tfrac{1-c}{b-c},\tfrac{b-1}{b-c}]\big{)}
\end{equation}
and thus, by inclusion/exclusion
\begin{equation}
\begin{split}
\iota^{(k)}_3(\bm{\sigma})_\mathrm{II}  & = \int_{\mathbb{T}_3} (\bm{\eta}^\top \vect{x}-1)^k \lambda_3(\dd\vect{x}) - (1-(-1)^k) \int_{\mathbb{T}_3^{abc}} (\bm{\eta}^\top \vect{x}-1)^k \lambda_3(\dd\vect{x})\\
& + (1-(-1)^k) \int_{\mathbb{T}_{3*}^{abc}} (\bm{\eta}^\top \vect{x}-1)^k \lambda_3(\dd\vect{x}).
\end{split}
\end{equation}
for any $k$ integer. These integrals are again easy to compute. Following Mathematica Code \ref{code:iota2} computes $\iota^{(k)}_3(\bm{\sigma})_\mathrm{II}$ for various values of $k$. Running the code for $k=1$ and $k=3$, we obtain
\begin{align}
\iota^{(1)}_3(\bm{\sigma})_{\mathrm{II}} = \iota^{(1)}_3(\bm{\sigma})_{\mathrm{I}} - \frac{(1-c)^4}{12 c (a-c) (b-c)},&&
\iota^{(3)}_3(\bm{\sigma})_\mathrm{II} = \iota^{(3)}_3(\bm{\sigma})_\mathrm{I} - \frac{(1-c)^6}{60 c (a-c) (b-c)}.
\end{align}
where the functions $\iota^{(1)}_3(\bm{\sigma})_{\mathrm{I}}$ and $\iota^{(3)}_3(\bm{\sigma})_{\mathrm{I}}$ are given by Equations \eqref{Eq:iota31} and \eqref{Eq:iota33}. In general case for any $k$ integer, we have
\begin{equation}
    \iota^{(k)}_3(\bm{\sigma})_\mathrm{II} = \iota^{(k)}_3(\bm{\sigma})_\mathrm{I} - \frac{(1-(-1)^k)(1-c)^{3+k}}{(k+1)(k+2)(k+3) c (a-c) (b-c)}.
\end{equation}
For $k$ even, we have $\iota_3^{(k)}(\bm{\sigma})_\mathrm{II} = \iota_3^{(k)}(\bm{\sigma})_\mathrm{I} = \iota_3^{(k)}(\bm{\sigma})_\mathrm{N}$ since the part with $1-(-1)^k$ vanishes. However, since the even metric moments $v_3^{(k)}(T_3)$ are trivial to compute by integration alone, we will proceed by assuming $k$ is odd. Denote 
\begin{align}
    T_2^{abc} = & \hull\left(\left[\frac{1}{a},0,0\right],\left[0,\frac{1}{b},0\right],\left[0,0,\frac{1}{c}\right]\right),\\
    T_{2*}^{abc} = & \hull\left(\left[\frac{1-c}{a-c},0,\frac{a-1}{a-c}\right],\left[0,\frac{1-c}{b-c},\frac{b-1}{b-c}\right],\left[0,0,\frac{1}{c}\right]\right),
\end{align}
we have for the intersection of $\bm{\sigma}$ with $\mathbb{T}_3$,
\begin{equation}
    \bm{\sigma} \cap \mathbb{T}_3 = T_2^{abc}\setminus T_{2*}^{abc} = \hull\big{(}[\tfrac{1}{a},0,0],[0,\tfrac{1}{b},0],[\tfrac{1-c}{a-c},0,\tfrac{a-1}{a-c}],[0,\tfrac{1-c}{b-c},\tfrac{b-1}{b-c}]\big{)},
\end{equation}
so $n_\mathrm{II} = 4$. By scale affinity
\begin{equation}
v_2^{(k+1)}(\bm{\sigma}\cap \mathbb{T}_3) = v_2^{(k+1)}(T_2^{abc}\setminus T_{2*}^{abc}) = v_2^{(k+1)}(\mathbb{U}_2^{\alpha\beta}),
\end{equation}
where $\mathbb{U}_2^{\alpha\beta} = \mathbb{T}_2 \setminus \mathbb{T}_2^{\frac{1}{\alpha},\frac{1}{\beta}}= \hull([0,0],[1,0],[0,1])\setminus\hull([0,0],[\alpha,0][0,\beta])$ is a canonical truncated triangle with
\begin{equation}\label{Eq:AlphaBetaTrans}
    \alpha = \frac{a (1-c)}{a-c}, \qquad \beta = \frac{b (1-c)}{b-c}.
\end{equation}

See Figure \ref{fig:secT3ConII} below for an illustration of $\mathbb{U}_2^{\alpha\beta}$ and its volumetric moments.
\begin{figure}[H]
    \centering     \includegraphics[width=0.30\textwidth]{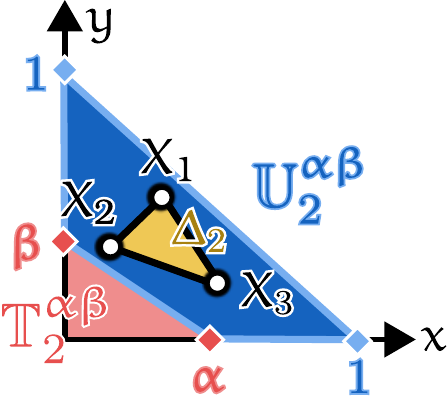}
    \caption{Mean section moments in the second $\mathcal{C}(\mathbb{T}_3)$ configuration}
    \label{fig:secT3ConII}
\end{figure}

Since $\vol_2 \mathbb{U}_2^{\alpha\beta} = \tfrac{1}{2}(1-\alpha\beta)$, we can write in general,
\begin{equation}
v_2^{(k+1)}(\mathbb{U}_2^{\alpha\beta}) = \left(\frac{2}{1-\alpha\beta}\right)^{k+4} \int_{(\mathbb{U}_2^{\alpha\beta})^3} \Delta_2^{k+1} \ddd \vect{x}_0\dd \vect{x}_1 \dd \vect{x}_2,
\end{equation}
We would like to find $v_2^{(k+1)}(\mathbb{U}_2^{\alpha\beta})$ for odd $k$. This is, luckily, trivial, since we are now integrating even powers of
\begin{equation}
\Delta_2 = \frac{1}{2!}\left|\det(\vect{x}_1-\vect{x}_0\,|\,\vect{x}_2-\vect{x}_0)\right|.
\end{equation}
The calculation can be carried out in Mathematica using Code \ref{code:vII}, which exploits the symmetries and uses inclusion/exclusion. Running the code for $k=1$ and $k=3$, we obtain
\begin{align}
v_2^{(2)}(\mathbb{U}_2^{\alpha\beta}) & \!= \! \frac{\alpha ^4 \beta ^4\!-\!8 \alpha ^3 \beta ^3\!+\!8 \alpha ^3 \beta ^2\!-\!4 \alpha ^3 \beta \!+\!8 \alpha ^2 \beta ^3\!-\!10 \alpha ^2 \beta ^2\!+\!8 \alpha ^2 \beta \!-\!4 \alpha  \beta ^3\!+\!8 \alpha  \beta ^2\!-\!8 \alpha  \beta \!+\!1}{72(1-\alpha  \beta )^4},\\
    v_2^{(4)}(\mathbb{U}_2^{\alpha\beta}) & \!=\! \frac{\!\!\begin{Bmatrix}
   46 \alpha ^3\!-\!31 \alpha ^3 \beta ^5 \beta ^4\!-\!6 \alpha  \beta ^5\!+\!18 \alpha  \beta ^4\!+\!32 \alpha  \beta ^2\!-\!19 \alpha  \beta \!-\!31 \alpha  \beta ^3\!-\!19 \alpha ^5 \beta ^5\!+\!32 \alpha ^5 \beta ^4\\
   \!+\!\alpha ^6 \beta ^6\!-\!6 \alpha ^5 \beta\!+\!18 \alpha ^5 \beta ^2 \!-\!31 \alpha ^3 \beta\!+\!32 \alpha ^4 \beta^5\!-\!47 \alpha ^4 \beta ^4\!+\!46 \alpha ^4 \beta ^3\!-\!34 \alpha ^4 \beta ^2\!+\!18 \alpha ^4 \beta \!\\
   \!-\!31 \alpha ^5 \beta ^3\!-\!50\alpha ^3 \beta ^3\!+\!46 \alpha ^3 \beta ^2\!+\!18 \alpha ^2 \beta ^5\!-\!34 \alpha ^2 \beta ^4\!+\!46 \alpha ^2 \beta ^3\!-\!47 \alpha ^2\beta ^2\!+\!32 \alpha ^2 \beta\!+\!1
   \end{Bmatrix}\!\!}{900 (1-\alpha\beta)^6}.
\end{align}
Finally, by definition (alternatively by Equation \eqref{Eq:ZetaEta})
\begin{equation}
    \zeta_3(\bm{\sigma})_\mathrm{II} = \frac{\vol_2(\bm{\sigma}\cap \mathbb{T}_3)}{\| \bm{\eta}\| \vol_3 \mathbb{T}_3} = (1-\alpha\beta)\frac{\vol_2 T_2^{abc}}{\| \bm{\eta}\| \vol_3 \mathbb{T}_3} = (1-\alpha\beta) \zeta_3(\bm{\sigma})_\mathrm{I}.
\end{equation}
Before we proceed to evaluate the final integral, we make the following change of variables $(a,b,c) \to (\alpha,\beta,c)$ via transformation Equations \eqref{Eq:AlphaBetaTrans}, which transform the integration half-domain into
\begin{equation}
(\mathbb{R}^3\setminus \mathbb{T}_3^\circ)^{*}_\mathrm{II}\,|_{\alpha,\beta,c} = (1-c,1)^2 \times (0,1).
\end{equation}
Note that, if $c$ is treated as a parameter, the variables $a,b$ depend on $\alpha,\beta$ separately. As a consequence,
\begin{equation}
    \dd a = \frac{c(1-c) \ddd \alpha}{(1-c-\alpha)^2}, \qquad \dd b = \frac{c(1-c) \ddd \beta}{(1-c-\beta)^2}
\end{equation}
and thus one has for the of transformation of measure
\begin{equation}
\lambda_3(\dd \bm{\eta}) = \dd a \ddd b  \ddd c = \frac{c^2(1-c)^2 \ddd \alpha  \ddd \beta\ddd c}{(1-c-\alpha)^2(1-c-\beta)^2}.
\end{equation}
Putting everything into the integral in Equation \eqref{Eq:SecInt} with prefactor $2$, we get when $k=1$,
\begin{equation}
\begin{split}
     & v_3^{(1)}(\mathbb{T}_3)_\mathrm{II} = \frac{3}{16} \int_0^1  \int_{1-c}^1 \int_{1-c}^1 \frac{3 (1-c)^2 (1-c-\alpha )^3 (1-c-\beta )^3 (1-\alpha  \beta )}{16 c^{13} \alpha ^5 \beta ^5}\times\\
     & \left(\frac{2 (1-c-\alpha ) (1-c-\beta ) \left(1-(1-c)^2 \alpha  \beta \right)}{c^3 \alpha  \beta}+c \left(1-\frac{\alpha }{1-c-\alpha }-\frac{\beta }{1-c-\beta }\right)-4\right)\times\\
     & \big{(}1\!-\!8\alpha  \beta \!+\!8 \alpha ^2 \beta \!-\!4 \alpha ^3 \beta \!+\!8 \alpha  \beta ^2\!-\!10 \alpha ^2 \beta ^2\!+\!8\alpha ^3 \beta ^2\!-\!4 \alpha  \beta ^3\!+\!8 \alpha ^2 \beta ^3\!-\!8 \alpha ^3 \beta ^3\!+\!\alpha ^4 \beta^4\big{)} \ddd \alpha \ddd\alpha\ddd c.\\
\end{split}
\end{equation}
Integrating out $\alpha,\beta$ can be done relatively easily, we end up with
\begin{equation}
v_3^{(1)}(\mathbb{T}_3)_\mathrm{II}  =  \int_0^1  \frac{c^2 p_0+4800 c (1-c)^2 p_1 \ln (1-c)+3600 (1-c)^2 p_2 \ln ^2(1-c)}{19200 c^{16}} \dd c,
\end{equation}
where
\begin{align}
\begin{split}
  p_0     & = 8265600\!-\!49593600 c\!+\!111530400 c^2\!-\!103044000 c^3 \!-\!10353200 c^4\!+\!114147200 c^5 \\
    &\!-\!115229200 c^6\!+\!58917200 c^7\!-\!17280824 c^8\!+\!2861248 c^9\!-\!220122 c^{10}\!-\!702 c^{11}\!+\!213 c^{12},
\end{split}
\\[1ex]
\begin{split}
p_1  & = 3444\!-\!15498 c\!+\!22076 c^2\!-\!4942 c^3\!-\!18060 c^4\!+\!21343 c^5 \!-\!11086 c^6\!+\!3147 c^7\!-\!496 c^8\!+\!36 c^9,
\end{split}
\\[1ex]
\begin{split}
p_2    & = 2296\!-\!11480 c\!+\!19692 c^2\!-\!9888 c^3\!-\!11350 c^4\!+\!20442 c^5 \!-\!13971 c^6\!+\!5296 c^7\!-\!1191 c^8\!+\!154 c^9\!-\!9 c^{10}.
\end{split}
\end{align}
The last $c$ integration can be carried out by Mathematica (alternatively, we can use derivatives of the Beta function). We get
\begin{equation}
    v_3^{(1)}(\mathbb{T}_3)_\mathrm{II} = \frac{217}{54000}-\frac{\pi ^2}{45045}.
\end{equation}
For higher values of $k$, the integration possesses similar difficulty, we got
\begin{align}
   v_3^{(3)}(\mathbb{T}_3)_\mathrm{II} & = \frac{105199}{9261000000}+\frac{79 \pi ^2}{7274767500},&&
   v_3^{(5)}(\mathbb{T}_3)_\mathrm{II}  = \frac{1890871}{9601804800000}-\frac{547 \pi^2}{26831987910000}.
    \end{align}

\subsubsection{Contribution from all configurations}
By Equation \eqref{Eq:DecoCon} and by affine invariancy,
\begin{equation}
v_3^{(k)}(T_3) = \sum_{C \in \mathcal{C}(T_3)} w_C \, v_3^{(k)}(T_3)_C = 4 v_3^{(k)}(\mathbb{T}_3)_{\mathrm{I}} + 3v_3^{(k)}(\mathbb{T}_3)_{\mathrm{II}},
\end{equation}
from which, immediately, we get Buchta's and Reitzner's \cite{buchta1992expected}, Mannion's \cite{mannion1994volume} and Philip's \cite{philip2006tetrahedron} result for $v_3^{(1)}(T_3)$ and also some of its further generalisations
\begin{align}
   v_3^{(1)}(T_3) & = \frac{13}{720} - \frac{\pi^2}{15015} \approx 0.01739823925,\\[2ex]
   \label{Eq:v33T3}v_3^{(3)}(T_3) & = \frac{733}{12600000}+\frac{79 \pi ^2}{2424922500} \approx 0.0000584961, \\[2ex]
   v_3^{(5)}(T_3) & = \frac{5125739}{4356374400000}-\frac{547 \pi ^2}{8943995970000} \approx 0.000001176003.
\end{align}

\subsection{Octahedron odd volumetric moments}
By affine invariancy, it does not matter how large is the volume of an octahedron as long as the octahedron stays regular. Hence, we may select the following representation of a regular octahedron
\begin{equation}
    O_3 = \hull([\pm 1,0,0],[0,\pm 1,0],[0,0,\pm 1]),
\end{equation}
which has $\vol_3 O_3 = 4/3$. According to its genealogy $\mathcal{C}(O_3)$, it has three configurations as shown in Figure \ref{fig:OCTAHE_GENEALOGY}. Table \ref{tab:Octavol} shows specifically which sets $S$ of vertices are separated by a cutting plane $\bm{\sigma}$ in which configurations in our local representation of $O_3$ above. Note that there is an ambiguity how to select those vertices as long it is the same configuration.
\begin{table}[htb]
    \centering
\begin{tabular}{|c|c|c|c|}
\hline
 $C$ & $\mathrm{I}$ & $\mathrm{II}$ & $\mathrm{III}$\\
 \hline
 $S$ & $[0,0,1]$ & \begin{tabular}{c} $[1,0,0]$ \\ $[0,1,0]$ \end{tabular} & \begin{tabular}{c} $[1,0,0]$\\ $[0,1,0]$\\ $[0,0,1]$ \end{tabular}\\
 \hline 
 $w_C$ & $6$ & $12$ & $4$\\
 \hline
\end{tabular}
    \caption{Configurations $\mathcal{C}(O_3)$ in a local representation.}
    \label{tab:Octavol}
\end{table}

By Theorem \ref{Thm:Canon} and for any $C \in \mathcal{C}(O_3)$,
\begin{equation}\label{Eq:SecIntOcta}
v_3^{(k)}(O_3)_C = \frac{2}{ 3^k}\int_{(\mathbb{R}^3\setminus O_3^\circ)_C} v_2^{(k+1)}(\bm{\sigma}\cap O_3) \,\zeta_3^{k+4}(\bm{\sigma}) \iota^{(k)}_3(\bm{\sigma}) \lambda_3(\dd \bm{\eta}),
\end{equation}
where
\begin{equation}
    \zeta_3(\bm{\sigma}) = \frac{\vol_2(\bm{\sigma}\cap O_3)}{\|\bm{\eta}\|\vol_3 O_3}, \qquad \iota^{(k)}_3(\bm{\sigma}) = \int_{O_3} |\bm{\eta}^\top \vect{x}-1|^k \lambda_3(\dd\vect{x}).
\end{equation}
We can describe the relation $\vect{x}=(x,y,z)^\top \in O_3$ by the following set of eight linear inequalities (all of them keep $\vect{0}\in O_3$)
\begin{equation}\label{Eq:OctaIneq}
\begin{split}
    & x + y + z < 1, \quad -x + y + z < 1, \quad x + y - z < 1, \quad -x + y - z < 1,\\
    & x - y + z < 1, \quad -x - y + z < 1, \quad x - y - z < 1, \quad -x - y - z < 1.
\end{split}
\end{equation}

\subsubsection{Configuration I}
First, we find $(\mathbb{R}^3\setminus O_3^\circ)_\mathrm{I}$. By Equation \eqref{Eq:Separ}, plugging the configurations points from $S$ into $\bm{\eta}^\top \vect{x} > 1$ and from $V\setminus S$ into $\bm{\eta}^\top \vect{x} < 1$ (flipped inequalities give the empty set), we get that $a,b,c$ must satisfy
\begin{equation}
    c>1, \quad -c<1, \quad a<1, \quad -a<1, \quad b<1, \quad -b<1,
\end{equation}
so $(\mathbb{R}^3\setminus O_3^\circ)_\mathrm{I}  =  (-1,1)^2\times (1,\infty)$. Next, $\bm{\sigma}$ splits $O_3$ into $O_3^+\sqcup O_3^-$. We can parametrize those domains by simultaneously solving Equation \eqref{Eq:Pplusminus} and \eqref{Eq:OctaIneq}. From those inequalities, we get by linear programming
\begin{equation}\label{Eq:OctaIplus}
    O_3^+ \!=\! \hull\!\big{(}[-1, \!0,\! 0],\! [0,\! -1,\! 0],\! [0,\! 0,\! -1],\! [0,\! 0,\! 1],\! [0,\! 1,\! 0],\! [\tfrac{b-1}{b-a},\!\tfrac{a-1}{a-b},\!0],\![\tfrac{b+1}{a+b},\!\tfrac{1-a}{a+b},\!0],\! [\tfrac{c-1}{c-a},\!0,\!\tfrac{a-1}{a-c}],\! [\tfrac{c+1}{a+c},\!0,\!\tfrac{1-a}{a+c}]\big{)}.
\end{equation}
Note that a simultaneous system of inequalities can be reduced using the eponymous \texttt{Reduce} command in Mathematica (used also in the case above). As a direct consequence of this parametrisation, we get
\begin{equation}
\vol_3 O_3^+ = \frac{2 \left(c^4+3 c^3-3 c^2-2 a^2 c^2-2 b^2 c^2+c+2 a^2 b^2\right)}{3 (c-a) (c+a) (c-b) (c+b)}
\end{equation}
from which, by Equation \eqref{Eq:ZetaEta},
\begin{equation}
    \zeta_3(\bm{\sigma})_\mathrm{I} = \frac{3 c (c-1)^2}{2 (c-a) (c+a) (c-b) (c+b)}.
\end{equation}
Also, thanks to our parametrisation, we get% https://mathematica.stackexchange.com/questions/19392/integration-over-region-given-by-inequality
\begin{equation}
\iota^{(k)}_3(\bm{\sigma})_\mathrm{I} = \int_{O_3^+}  (1-\bm{\eta}^\top \vect{x})^k \lambda_3(\dd\vect{x}) + \int_{O_3^-} (\bm{\eta}^\top \vect{x}-1)^k \lambda_3(\dd\vect{x})
\end{equation}
for any real $k>-1$ almost for free, namely for $k=1$ and $k=3$,
\begin{align}
    \iota_3^{(1)}(\bm{\sigma})_\mathrm{I} & = \frac{c^5+6 c^3+4 a^2 b^2-4 c^2 \left(1+a^2+b^2\right)+c}{3 (c-a) (c+a) (c-b) (c+b)},\\
    \iota_3^{(3)}(\bm{\sigma})_\mathrm{I} & = \frac{\!c^7 \!+\!15 c^5\!+\!15 c^3 \!-\!6c^2\!+\!6 a^4 c^2\!-\!20 b^2 c^2\!-\!6 b^4 c^2\!-\!20 a^2 c^2 \!-\! 6 a^2 b^2 c^2 \!+\! c \!+\!20 a^2 b^2\!+\!6 a^4 b^2\!+\! 6 a^2 b^4
    \!}{15 (c-a) (c+a) (c-b) (c+b)}
\end{align}
and also $n_\mathrm{I} = 4$ since
\begin{equation}
    \bm{\sigma} \cap O_3 = \hull\big{(} [\tfrac{b-1}{b-a},\tfrac{a-1}{a-b},0],[\tfrac{b+1}{a+b},\tfrac{1-a}{a+b},0], [\tfrac{c-1}{c-a},0,\tfrac{a-1}{a-c}], [\tfrac{c+1}{a+c},0,\tfrac{1-a}{a+c}] \big{)}.
\end{equation}
We can use a computer to deduce the following even moments
\begin{align}
    v_2^{(2)}(\bm{\sigma}\cap O_3) = \frac{3 c^4\!+\!c^2(a^2\!+\!b^2)\!-\!a^2 b^2}{288 c^4},&&
    v_2^{(4)}(\bm{\sigma}\cap O_3) = \frac{\begin{Bmatrix}
        12c^6\!+\!17a^2c^4\!-\!14a^2 b^2c^2\!+\!3 a^4c^2\\
        \!+\!17 b^2c^4\!+\!3b^4c^2\!-\! 3 a^4 b^2 \!-\! 3 a^2 b^4
    \end{Bmatrix}}{28800 c^6}.
\end{align}
Therefore, putting everything together,
\begin{align}
\begin{split}
    v_3^{(1)}(O_3)_\mathrm{I} = & \frac{3}{512}\int_1^\infty \int_{-1}^1 \int_{-1}^1 \frac{c \, (c-1)^{10} \left(c^2 a^2 + c^2 b^2 - a^2 b^2 + 3 c^4\right)}{(c-a)^6 (c+a)^6 (c-b)^6 (c+b)^6} \\
   & \times \left(4 a^2 b^2 - 4 a^2 c^2 -4 b^2 c^2 - 4 c^2+c^5+6 c^3+c\right) \ddd a  \ddd b  \ddd c,
\end{split}
\end{align}
similarly for $v_3^{(3)}(O_3)_\mathrm{I}$. Integration in Mathematica then reveals
\begin{align}
    v_3^{(1)}(O_3)_\mathrm{I} & = \frac{2569561}{230400} - \frac{ 11571 \pi^2}{10240},\\
    v_3^{(3)}(O_3)_\mathrm{I} & = \frac{3260724307264561}{433954160640000}-\frac{109143647 \pi ^2}{143360000},\\
    v_3^{(5)}(O_3)_\mathrm{I} & = \frac{1306914286180250262095927}{59965827237606850560000}-\frac{3676076446537 \pi ^2}{1664719257600}.
\end{align}

\subsubsection{Configuration II}
By Equation \eqref{Eq:Separ}, plugging the configurations points from $S$ into $\bm{\eta}^\top \vect{x} > 1$ and from $V\setminus S$ into $\bm{\eta}^\top \vect{x} < 1$ (flipped inequalities give empty set), we get that $a,b,c$ must satisfy
\begin{equation}
    c<1, \quad -c<1, \quad a>1, \quad -a<1, \quad b>1, \quad -b<1,
\end{equation}
so $(\mathbb{R}^3\setminus O_3^\circ)_\mathrm{II}  =  (1,\infty)^2\times (-1,1)$. Next, $\bm{\sigma}$ splits $O_3$ into $O_3^+\sqcup O_3^-$. We can parametrize those domains by simultaneously solving Equation \eqref{Eq:Pplusminus} and \eqref{Eq:OctaIneq}. Then, by linear programming, we get $n_\mathrm{II} = 6$ since we obtained
\begin{equation}
\begin{split}
    O_3^+ = & \hull\big{(} [-1, 0, 0], [0, -1, 0], [0, 0, -1], [0, 0, 1], [0,\tfrac{c-1}{c-b},\tfrac{b-1}{b-c}],\\
    & [0,\tfrac{c+1}{b+c},\tfrac{1-b}{b+c}], [\tfrac{1-b}{a+b},\tfrac{a+1}{a+b},0], [\tfrac{b+1}{a+b},\tfrac{1-a}{a+b},0],[\tfrac{c-1}{c-a},0,\tfrac{a-1}{a-c}],[\tfrac{c+1}{a+c},0,\tfrac{1-a}{a+c}]\big{)},
\end{split}
\end{equation}
from which, using Mathematica,
\begin{equation}
\vol_3 O_3^+ = \frac{2 \begin{Bmatrix}
3 a^2 b^2+a^3 b^2+a^2 b^3-c^2+3 a c^2-3 a^2 c^2-a^3 c^2 + 3 b c^2 \\
-3 a b c^2-a^2 b c^2-3 b^2 c^2-a b^2 c^2-b^3 c^2+2 a c^4+2 b c^4-a b
\end{Bmatrix}}{3 (a+b) (a-c) (b-c) (a+c) (b+c)}
\end{equation}
which further yields, by Equation \eqref{Eq:ZetaEta},
\begin{equation}
    \zeta_3(\bm{\sigma})_\mathrm{II} = \frac{3 \left(2 a c^2+2 b c^2-a b+a^2 b^2-a^2 c^2-b^2 c^2-a b c^2-c^2\right)}{2 (a+b) (a-c) (b-c) (a+c) (b+c)}.
\end{equation}
Next, for $k=1$, we obtain
\begin{align}
    \iota_3^{(1)}(\bm{\sigma})_\mathrm{II}  =  \frac{\begin{Bmatrix}
        a^4 b^2-a^4 c^2+a^3 b^3-a^3 b c^2+a^2 b^4-a^2 b^2 c^2+6 a^2 b^2-6 a^2 c^2-a b\\
        -a b^3 c^2-6 a b c^2+4 a c^4+4 a c^2-b^4 c^2-6 b^2 c^2+4 b c^4+4 b c^2-c^2
    \end{Bmatrix}}{3 (a + b) (a - c) (b - c) (a + c) (b + c)}.
\end{align}
As $\iota_3^{(3)}(\bm{\sigma})$, $v_2^{(2)}(\bm{\sigma}\cap O_3)$ and $v_2^{(4)}(\bm{\sigma}\cap O_3)$ are rather long, we are not listing them here. Putting everything together and integrating over $a,b,c$, we get
\begin{align}
    v_3^{(1)}(O_3)_\mathrm{II} & = \frac{72588071 \pi ^2}{92252160}-\frac{12023076361}{1548288000},\\
    v_3^{(3)}(O_3)_\mathrm{II} & = \frac{38809663388059 \pi ^2}{95351832576000}-\frac{830108924076197}{206644838400000},\\
    v_3^{(5)}(O_3)_\mathrm{II} & = \frac{24706383193486257481 \pi^2}{22106368864419840000}-\frac{6614474327656066615169519}{599658272376068505600000}.
\end{align}

\subsubsection{Configuration III}
By Equation \eqref{Eq:Separ}, $a,b,c$ must satisfy
\begin{equation}
    c>1, \quad -c<1, \quad a>1, \quad -a<1, \quad b>1, \quad -b<1,
\end{equation}
or with $<$ and $>$ flipped,
\begin{equation}
    c<1, \quad -c>1, \quad a<1, \quad -a>1, \quad b<1, \quad -b>1,
\end{equation}
so $(\mathbb{R}^3\setminus O_3^\circ)_\mathrm{II}  =  ((-\infty,-1)\cup (1,\infty))^3$. By symmetry, we may integrate only over half-domain $(\mathbb{R}^3\setminus O_3^\circ)^{*}_\mathrm{II}  =  (1,\infty)^3$. Next, $O_3^+\sqcup O_3^-$, where, by simultaneously solving Equations \eqref{Eq:Pplusminus} and \eqref{Eq:OctaIneq} and by linear programming,
\begin{equation}
\begin{split}
    O_3^+ = \hull\big{(} &[-1, 0, 0], [0, -1, 0], [0, 0, -1], [0,\tfrac{1-c}{b+c},\tfrac{b+1}{b+c}], [0, \tfrac {c + 1} {b + c}, \tfrac {1 - b} {b + c}],\\
    & [\tfrac {1 - b} {a + b}, \tfrac {a + 1} {a + b}, 0], [\tfrac {b + 1} {a + b}, \tfrac {1 - a} {a + b}, 0], [\tfrac {1 - c} {a + c}, 0, \tfrac {a + 1} {a + c}], [\tfrac {c + 1} {a + c}, 0, \tfrac {1 - a} {a + c}]\big{)},
\end{split}
\end{equation}
which means $n_\mathrm{III}=6$. Using Mathematica,
\begin{equation}
\vol_3 O_3^+ = \frac{2 \left(3 a b+a^2 b+a b^2+3 a c+a^2 c+3 b c+2 a b c+b^2 c+a c^2+b c^2-1\right)}{3 (a+b) (a+c) (b+c)}
\end{equation}
from which, by Equation \eqref{Eq:ZetaEta},
\begin{equation}
    \zeta_3(\bm{\sigma})_\mathrm{III} = \frac{3 (a b+a c+b c-1)}{2 (a+b) (a+c) (b+c)}.
\end{equation}
Next, for $k=1$ and $k=3$, we obtained
\begin{align}
    \iota_3^{(1)}(\bm{\sigma})_\mathrm{III} & = \frac{\begin{Bmatrix}
    6 a b-1+a^3 b+a^2 b^2+a b^3+6 a c+a^3 c+6 b c\\
    +2 a^2 b c+2 a b^2 c+b^3 c+a^2 c^2+2 a b c^2+b^2 c^2+a c^3+b c^3    
    \end{Bmatrix}}{3
   (a+b) (a+c) (b+c)}
\end{align}
and
\begin{align}
    v_2^{(2)}(\bm{\sigma}\!\cap\! O_3) &  \!=\!  \frac{\begin{Bmatrix}
    3 a b\!-\!1\!-\!6 a^2\!+\!20 a^3 b\!-\!6 b^2\!-\!18 a^2 b^2\!-\!6 a^4 b^2\!+\!20 a b^3\!+\!3 a c\!-\!6 a^2 b^4\! -\!21 a^4 b^4\!\\
    +\!3 a^5 b^5\!+\!18 a^3 b^3\!+\!20 a^3 c\!+\!3 b c\!-\!12 a^2 b c \!-\!12 a^4 b c\!-\!12 a b^2 c\!+\!18 a^3 b^2 c\!+\!20 b^3 c\!\\
    +\!18 a^2 b^3 c\!-\!48 a^4 b^3 c\!-\!12 a b^4 c\!-\!48 a^3 b^4 c\!+\!15 a^5 b^4 c\!+\!15 a^4 b^5 c \!-\!6 c^2-\!18 a^2 c^2\!\\
    \!-\!6 a^4 c^2\!-\!12 a b c^2\!+\!18 a^3 b c^2\!-\!18 b^2 c^2\!-\!6 b^4 c^2-\!54 a^4 b^2 c^2\!+\!18 a b^3 c^2\!-\!84 a^3 b^3 c^2 \!\\
    \!+\!30 a^5 b^3 c^2\!+\!108 a^2 b^2 c^2\!-\!54 a^2 b^4 c^2\!+\!54 a^4 b^4 c^2\!+\!30 a^3 b^5 c^2\!+\!30 a^5 b^2 c^3\!+\!18 a^3 c^3\\
    \!+\!20 b c^3\!+\!18 a^2 b c^3\!-\!48 a^4 b c^3\!+\!18 a b^2 c^3\!-\!84 a^3 b^2 c^3\!+\!20 a c^3\!+\!18 b^3 c^3\!-\!84 a^2 b^3 c^3\!\\
    \!-\!6 b^2 c^4\!-\!48 a b^4 c^3\!+\!78 a^3 b^4 c^3\!+\!30 a^2 b^5 c^3\!-\!6 a^2 c^4\!-\!21 a^4 c^4 \!-\!12 a b c^4\!-\!48 a^3 b c^4\!\\
    \!+\!15 a^5 b c^4\!+\!54 a^2 b^4 c^4\!-\!54 a^2 b^2 c^4\!+\!54 a^4 b^2 c^4\!+\!78 a^4 b^3 c^3\!+\!78 a^3 b^3 c^4+\!15 a b^5 c^4\!\!\\
    \!-\!21 b^4 c^4\!+\!3 b^5 c^5\!+\!3 a^5 c^5\!-\!48 a b^3 c^4 \!+\!15 a^4 b c^5\!+\!30 a^3 b^2 c^5\!+\!30 a^2 b^3 c^5\!+\!15 a b^4 c^5
    \end{Bmatrix}}{288 (a b+a c+b c\!-\!1)^5},
\end{align}
As $\iota_3^{(3)}(\bm{\sigma})$ and $v_2^{(4)}(\bm{\sigma}\cap O_3)$ are rather long, we are not listing them here. Putting everything together and integrating over $a,b,c$ and multiplying by the factor of two (as $(1,\infty)^3$ is only a half-domain of integration),
\begin{align}
    v_3^{(1)}(O_3)_\mathrm{III} & = \frac{376079789}{57344000}-\frac{2721 \pi ^2}{4096},\\
    v_3^{(3)}(O_3)_\mathrm{III} & = \frac{752252545541087}{964342579200000}-\frac{90646167 \pi ^2}{1146880000},\\
    v_3^{(5)}(O_3)_\mathrm{III} & = \frac{3995047725382306264583}{9994304539601141760000}-\frac{4195233727 \pi ^2}{103582531584}.
\end{align}

\subsubsection{Contribution from all configurations}
By Equation \eqref{Eq:DecoCon},
\begin{equation}
v_3^{(k)}(O_3) = \sum_{C \in \mathcal{C}(O_3)} w_C \, v_3^{(k)}(O_3)_C = 6 v_3^{(k)}(O_3)_{\mathrm{I}} + 12v_3^{(k)}(O_3)_{\mathrm{II}}+4v_3^{(k)}(O_3)_{\mathrm{III}},
\end{equation}
from which, immediately
\begin{align}
   v_3^{(1)}(O_3) & = \frac{19297 \pi ^2}{3843840}-\frac{6619}{184320} \approx 0.013637411,\\[2ex]
   \label{Eq:v33O3}v_3^{(3)}(O_3) & = \frac{1628355709 \pi ^2}{19864965120000}-\frac{81932629}{103219200000} \approx 0.0000152505, \\[2ex]
   v_3^{(5)}(O_3) & = \textstyle\frac{6356364544399 \pi ^2}{1611922729697280000}-\frac{205491225433}{5287025049600000}\approx 5.215748\cdot 10^{-8}.
\end{align}

\subsection{Cube odd volumetric moments}
We use the following standard representation of the unit cube ($\vol_3 C_3 = 1$),
\begin{equation}
    C_3 \! =\! \hull([0,0,0],\![1,0,0],\![0,1,0],\![0,0,1],\![0,1,1],\![1,0,1],\![1,1,0],\![1,1,1]).
\end{equation}
According to its genealogy $\mathcal{C}(C_3)$, it has five configurations as shown in Figure \ref{fig:CUBE_GENEALOGY}. Table \ref{tab:Cubevol} shows specifically which sets $S$ of vertices in which configurations are separated by a cutting plane $\bm{\sigma}$ in our standard representation of $C_3$ above.
\begin{table}[htb]
    \centering
\begin{tabular}{|c|c|c|c|c|c|}
\hline
 $C$ & $\mathrm{I}$ & $\mathrm{II}$ & $\mathrm{III}$ & $\mathrm{IV}$ & $\mathrm{V}$\\
 \hline
 $S$ & $[0,0,0]$ & \begin{tabular}{c} $[0,0,0]$ \\ $[0,0,1]$ \end{tabular} & \begin{tabular}{c} $[0,0,0]$\\ $[1,0,0]$\\ $[0,1,0]$ \end{tabular} & \begin{tabular}{c} $[0,0,0]$\\ $[1,0,0]$\\ $[0,1,0]$\\ $[0,0,1]$ \end{tabular} & \begin{tabular}{c} $[0,0,0]$\\ $[1,0,0]$\\ $[0,1,0]$\\ $[1,1,0]$ \end{tabular}\\
 \hline 
 $w_C$ & $8$ & $12$ & $24$ & $4$ & $3$\\
 \hline
\end{tabular}
    \caption{Configurations $\mathcal{C}(C_3)$ in the standard representation of $C_3$.}
    \label{tab:Cubevol}
\end{table}

By Theorem \ref{Thm:Canon} and for any $C \in \mathcal{C}(C_3)$,
\begin{equation}\label{Eq:SecIntCube}
v_3^{(k)}(C_3)_C = \frac{2}{ 3^k}\int_{(\mathbb{R}^3\setminus C_3^\circ)_C} v_2^{(k+1)}(\bm{\sigma}\cap C_3) \,\zeta_3^{k+4}(\bm{\sigma}) \iota^{(k)}_3(\bm{\sigma}) \lambda_3(\dd \bm{\eta}),
\end{equation}
where
\begin{equation}
    \zeta_3(\bm{\sigma}) = \frac{\vol_2(\bm{\sigma}\cap C_3)}{\|\bm{\eta}\|\vol_3 C_3}, \qquad \iota^{(k)}_3(\bm{\sigma}) = \int_{C_3} |\bm{\eta}^\top \vect{x}-1|^k \lambda_3(\dd\vect{x}).
\end{equation}
We can describe the relation $\vect{x}=(x,y,z)^\top \in C_3$ by the following set of three linear inequalities
\begin{equation}\label{Eq:CubeIneq}
    0< x < 1, \qquad 0 < y < 1, \qquad 0 < z <1.
\end{equation}
For Configuration $\mathrm{I}$, by Equation \eqref{Eq:Separ}, $a,b,c$ must satisfy
\begin{equation}
    a>1, \quad b>1, \quad a+b>1, \quad a+c>1, \quad b+c>1, \quad a+b+c>1,
\end{equation}
so $(\mathbb{R}^3\!\setminus\! C_3^\circ)_\mathrm{I} = (1,\infty)^3$. Similarly for other configurations. Since the analysis is similar as in the case of $P_3$ being a regular octahedron $O_3$, we only list the results from all configurations, see Table \ref{tab:CubeAll}.
\begin{table}[H]
    \centering
\begin{tabular}{|c|c|c|c|}
\hline
 $\!C\!$ & $v_3^{(1)}(C_3)_C$ & $v_3^{(3)}(C_3)_C$ & $v_3^{(5)}(C_3)_C$\\
 \hline
 \ru{1.1}$\!\mathrm{I}\!$ & $\frac{391}{82944000}$ & $\frac{8717}{1800338400000}$ & $\frac{932274811}{50575353828920524800}$\\[0.6ex]
 \hline 
 \ru{1.1}$\!\mathrm{II}\!$ & $\frac{34309}{186624000}$ & $\frac{648789871}{3089380694400000}$ & $\frac{36816619074923}{51228618815877414912000}$\\[0.6ex]
 \hline
 \ru{1.1}$\!\mathrm{III}\!$ & $\!\!\!\frac{3191 \pi ^2}{207360}\!-\!\frac{792503149}{5225472000}\!\!\!$ & $\!\!\!\frac{182029 \pi^2}{195955200}\!-\!\frac{113292736592927}{12357522777600000}\!\!\!$ & $\frac{213033619 \pi^2}{634894848000}-\frac{47144185844633987}{14235866239795200000}$\\[0.6ex]
 \hline
 \ru{1.1}$\!\mathrm{IV}\!$ & $\frac{198785357}{217728000}-\frac{71 \pi ^2}{768}$ & $\!\frac{22659798780677}{411917425920000}\!-\!\frac{910157 \pi^2}{163296000}\!$ & $\!\!\!\frac{26487208076498306317}{1921073205595403059200}\!-\!\frac{27814438817 \pi ^2}{19910302433280}\!\!\!$\\[0.6ex]
 \hline
 \ru{1.1}$\!\mathrm{V}\!$ & $\frac{7}{5184}$ & $\frac{29}{21870000}$ & $\frac{22473091}{6271745266483200}$\\[0.6ex]
 \hline
\end{tabular}
    \caption{Sections integrals in various configurations $\mathcal{C}(C_3)$.}
    \label{tab:CubeAll}
\end{table}

Summing up the contributions from all configurations with appropriate weights,
\begin{equation}
\begin{split}
v_3^{(k)}(C_3) & = \!\!\!\! \sum_{C \in \mathcal{C}(C_3)} w_C \, v_3^{(k)}(C_3)_C = 8 v_3^{(k)}(C_3)_{\mathrm{I}} + 12v_3^{(k)}(C_3)_{\mathrm{II}} \\
& +  24v_3^{(k)}(C_3)_{\mathrm{III}} + 4v_3^{(k)}(C_3)_{\mathrm{IV}} + 3v_3^{(k)}(C_3)_{\mathrm{V}},
\end{split}
\end{equation}
from which immediately
\begin{align}
   v_3^{(1)}(C_3) & = \frac{3977}{216000}-\frac{\pi ^2}{2160} \approx 0.01384277574,\\[2ex]
   v_3^{(3)}(C_3) & = \frac{8411819}{450084600000}-\frac{\pi ^2}{3402000} \approx 0.0000157883 , \\[2ex]
    v_3^{(5)}(C_3) & = \textstyle \frac{306749173351 \pi ^2}{124439390208000}-\frac{2225580641145943786613}{91479676456923955200000} \approx 3.673225\cdot 10^{-7}.
\end{align}
We find it striking that even though an octahedron has fewer number of configurations than a cube, the value $v_3^{(1)}(C_3)$ has been obtained by Zinani \cite{zinani2003expected}, but the octahedron case $v_3^{(1)}(O_3)$ was unknown. Keep in mind that the configurations are the same in our canonical approach as well as in the original method using Efron's section formula.

\newpage
\section{Higher-dimensional polytopes}
\subsection{Pentachoron odd volumetric moments}
By a \emph{pentachoron}\index{pentachoron}, we mean a $4$-simplex\index{4-simplex}. The regular pentachoron is then $T_4$. The analysis is somewhat analogous to the three-dimensional case. Now, we obtain the volumetric moments $v_4^{(k)}(T_4)$ for odd $k$. First, since $v_4^{(k)}(T_4)$ is an affine invariant, it must be the same as $v_4^{(k)}(\mathbb{T}_4)$, where
\begin{equation}
\mathbb{T}_4 = \hull([0,0,0,0],[1,0,0,0],[0,1,0,0],[0,0,1,0],[0,0,0,1])
\end{equation}
is the canonical pentachoron\index{pentachoron!canonical}. We have $\vol_4 \mathbb{T}_4 = 1/4!=1/24$. Let $\bm{\eta}=(a,b,c,d)^\top$ be the Cartesian parametrisation of $\bm{\sigma} \in \mathbb{A}(4,3)$ such that $\vect{x}\in \bm{\sigma} \Leftrightarrow \bm{\eta}^\top\vect{x} = 1$. We have $\|\bm{\eta}\| = \sqrt{a^2+b^2+c^2+d^2}$. Based on symmetries $\mathcal{G}(T_4)$, there are two realisable configurations we need to consider. Thanks to affine invariancy, we can again consider instead the two $\mathcal{C}(\mathbb{T}_4)$ configurations (see Table \ref{tab:PentaConfs} below).
\begin{table}[htb]
    \centering
\begin{tabular}{|c|c|c|}
\hline
 $\mathbb{T}_4$ & $\mathrm{I}$ & $\mathrm{II}$\\
 \hline
 \ru{1.7}$S$ & $[0,0,0,0]$ & \begin{tabular}{c} $[0,0,0,0]$ \\ $[0,0,0,1]$ \end{tabular} \\[0.9em]
 \hline 
 $w_C$ & $5$ & $10$\\
 \hline
\end{tabular}
    \caption{Configurations $\mathcal{C}(\mathbb{T}_4)$.}
    \label{tab:PentaConfs}
\end{table}

By Theorem \ref{Thm:Canon} and for any $C \in \mathcal{C}(\mathbb{T}_4)$,
\begin{equation}\label{Eq:SecInt4}
 v_4^{(k)}(\mathbb{T}_4)_C = \frac{6}{ 4^k}\int_{(\mathbb{R}^4\setminus \mathbb{T}_4^\circ)_C} v_3^{(k+1)}(\bm{\sigma}\cap \mathbb{T}_4) \,\zeta_4^{k+5}(\bm{\sigma}) \iota^{(k)}_4(\bm{\sigma}) \lambda_4(\dd \bm{\eta}),
\end{equation}
where
\begin{equation}
    \zeta_4(\bm{\sigma}) = \frac{\vol_3(\bm{\sigma}\cap \mathbb{T}_4)}{\|\bm{\eta}\|\vol_4 \mathbb{T}_4}, \qquad \iota^{(k)}_4(\bm{\sigma}) = \int_{\mathbb{T}_4} |\bm{\eta}^\top \vect{x}-1|^k \lambda_4(\dd\vect{x}).
\end{equation}
Again, in order to distinguish between configurations, we write $\zeta_4(\bm{\sigma})_C$ and $\iota^{(k)}_4(\bm{\sigma})_C$ instead of just $\zeta_4(\bm{\sigma})$ and $\iota^{(k)}_4(\bm{\sigma})$.

\subsubsection{Configuration I}
To ensure $\bm{\sigma}$ separates only the point $[0,0,0,0]$, we get from Equation \eqref{Eq:Separ}, that $a>1$, $b>1$, $c>1$ and $d>1$. That means $(\mathbb{R}^4\setminus \mathbb{T}_4^\circ)_\mathrm{I} = (1,\infty)^4$ is our integration domain in $a,b,c,d$. Denote
\begin{equation}
    \mathbb{T}_4^{abcd} = \hull([0,0,0,0],[\tfrac{1}{a},0,0,0],[0,\tfrac{1}{b},0,0],[0,0,\tfrac{1}{c},0],[0,0,0,\tfrac{1}{d}]).
\end{equation}
The hyperplane $\bm{\sigma}$ splits $\mathbb{T}_4$ into disjoint union of two domains $\mathbb{T}_4^+\sqcup \mathbb{T}_4^-$, where the one closer to the origin is precisely $\mathbb{T}_4^+ = \mathbb{T}_4^{abcd}$. Therefore
\begin{equation}
\begin{split}
\iota^{(k)}_4(\bm{\sigma})_\mathrm{I}  & = \int_{\mathbb{T}_4^+}  (1-\bm{\eta}^\top \vect{x})^k \lambda_4(\dd\vect{x}) + \int_{\mathbb{T}_4^-} (\bm{\eta}^\top \vect{x}-1)^k \lambda_4(\dd\vect{x})\\
& = \int_{\mathbb{T}_4} (\bm{\eta}^\top \vect{x}-1)^k \lambda_4(\dd\vect{x}) - (1-(-1)^k) \int_{\mathbb{T}_4^{abcd}} (\bm{\eta}^\top \vect{x}-1)^k \lambda_4(\dd\vect{x}).
\end{split}
\end{equation}
for any $k$ integer. These integrals are easy to compute. Mathematica Code \ref{code:4diota} computes $\iota^{(k)}_4(\bm{\sigma})_\mathrm{I}$ for various values of $k$. Running the code for $k=1$ and $k=3$, we obtain
\begin{align}
\label{Eq:iota41}\iota^{(1)}_4(\bm{\sigma})_\mathrm{I}& =\frac{1}{120} \left(\frac{2}{a b c d}+a+b+c+d-5\right),\\[2ex]
\begin{split}
   \label{Eq:iota43}\iota^{(3)}_4(\bm{\sigma})_\mathrm{I} & = \frac{1}{840} \bigg{(}\frac{2}{a b c d}+a^3+a^2 b+a^2 c+a^2 d-7 a^2+a b^2+a b c+a b d-7 a b+a c^2\\
   &+a c d-7 a c+a d^2-7 a d+21 a+21 b+b^2 c+b^2 d-7 b^2-7 b d+b c d-7 b c\\[-0.3ex]
   & +b^3+b c^2+b d^2+c^3+c^2 d-7 c^2+c d^2-7 c d+21 c+d^3-7 d^2+21 d-35\bigg{)}.
\end{split}
\end{align}
Denote $T_3^{abcd}=\hull([1/a,0,0,0], [0,1/b,0,0],[0,0,1/c,0],[0,0,$\ $0,1/d])$, then
the intersection of the hyperplane $\bm{\sigma}$ with $\mathbb{T}_4$ is precisely tetrahedron $T_3^{abcd}$. That is,
\begin{equation}
    \bm{\sigma} \cap \mathbb{T}_4 = T_3^{abcd}.
\end{equation}
By Equation \eqref{Eq:DistEta}, the distance from $T_3^{abcd}$ to the origin is $\operatorname{dist}_{\bm{\sigma}}(\vect{0}) = 1/\| \bm{\eta} \|$. By base-height splitting,
\begin{equation}
    \frac{\vol_4 \mathbb{T}_4}{abc} = \vol_4\mathbb{T}_4^+ = \frac{1}{4} \operatorname{dist}_{\bm{\sigma}}(\vect{0}) \vol_3 T_3^{abcd} = \frac{\vol_3(\bm{\sigma}\cap \mathbb{T}_4)}{4 \| \bm{\eta}\|},
\end{equation}
from which we immediately get
\begin{equation}
    \zeta_4(\bm{\sigma})_\mathrm{I} = \frac{\vol_3(\bm{\sigma}\cap \mathbb{T}_4)}{\|\bm{\eta}\|\vol_4 \mathbb{T}_4} = \frac{4}{abcd}.
\end{equation}
Finally, by scale affinity (we have $n_\mathrm{I} = 4$),
\begin{equation}
v_3^{(k+1)}(\bm{\sigma}\cap \mathbb{T}_4) = v_3^{(k+1)}(T_3^{abcd}) = v_3^{(k+1)}(T_3),
\end{equation}
which implies for $k=1,2,3$ that (see Table \ref{tab:Evenkres} and Equation \eqref{Eq:v33T3}),
\begin{equation}
\begin{aligned}
& v_3^{(2)}(\bm{\sigma}\cap \mathbb{T}_4) = \tfrac{3}{4000}, && \quad v_3^{(3)}(\bm{\sigma}\cap \mathbb{T}_4) = \tfrac{733}{12600000}+\tfrac{79 \pi ^2}{2424922500}, && \quad v_3^{(4)}(\bm{\sigma}\cap \mathbb{T}_4) = \tfrac{871}{123480000}.
\end{aligned}
\end{equation}

Putting everything into the integral in Equation \eqref{Eq:SecInt4}, we get when $k=1$,
\begin{equation}
     v_4^{(1)}(\mathbb{T}_4)_\mathrm{I}  = \frac{24}{625} \int_1^\infty   \int_1^\infty  \int_1^\infty  \int_1^\infty  \frac{2+a b c d (a+b+c+d-5)}{a^7 b^7 c^7 d^7} \ddd a \ddd b\ddd c \ddd d = \frac{1}{16875}.
\end{equation}
For $k=3$ and $k=5$, we get
\begin{equation}
\begin{aligned}
   & \textstyle v_4^{(3)}(\mathbb{T}_4)_\mathrm{I} = \frac{26061191}{1600967592000000}, && \qquad \textstyle v_4^{(5)}(\mathbb{T}_4)_\mathrm{I} = \frac{27909940019}{504189521813376000000}.
\end{aligned}   
\end{equation}

\subsubsection{Configuration II}
In this scenario, $\bm{\sigma}$ separates two points $[0,0,0,0]$ and $[0,0,0,1]$ from $\mathbb{T}_4$. By Equation \eqref{Eq:Separ}, we get $a>1$, $b>1$, $c>1$ and $d<1$. We can split the condition for $d$ into to cases: either $0<d<1$ or $d<0$. In fact, both options give the same factor since they are symmetrical as they correspond to two possibilities where $\bm{\sigma}$ hits $\mathcal{A}([0,0,0,0],[0,0,0,1])$. Therefore we only consider the integration half-domain
\begin{equation}
(\mathbb{R}^4\setminus \mathbb{T}_4^\circ)^{*}_\mathrm{II} = (1,\infty)^3 \times (0,1)
\end{equation}
and in the end multiply the result twice. The hyperplane $\bm{\sigma}$ splits $\mathbb{T}_4$ into disjoint union of two domains $\mathbb{T}_4^+\sqcup \mathbb{T}_4^-$, where $\mathbb{T}_4^+$ being the one closer to the origin. Denote
\begin{align}
    \mathbb{T}_4^{abcd} = & \hull\left(\left[0,0,0,0\right],\left[\tfrac{1}{a},0,0,0\right],\left[0,\tfrac{1}{b},0,0\right],\left[0,0,\tfrac{1}{c},0\right],\left[0,0,0,\tfrac{1}{d}\right]\right),\\
    \mathbb{T}_{4*}^{abcd} = & \hull\left(\left[0,0,0,1\right],\left[\tfrac{1-d}{a-d},0,0,\tfrac{a-1}{a-d}\right],\left[0,\tfrac{1-d}{b-d},0,\tfrac{b-1}{b-d}\right],\left[0,0,\tfrac{1-d}{c-d},\tfrac{c-1}{c-d}\right],\left[0,0,0,\tfrac{1}{d}\right]\right),
\end{align}
then we can write $\mathbb{T}_4^+ = \mathbb{T}_{4}^{abcd}\setminus \mathbb{T}_{4*}^{abcd}$ and thus, by inclusion/exclusion
\begin{equation}
\begin{split}
\iota^{(k)}_4(\bm{\sigma})_\mathrm{II}  & = \int_{\mathbb{T}_4} (\bm{\eta}^\top \vect{x}-1)^k \lambda_4(\dd\vect{x}) - (1-(-1)^k) \int_{\mathbb{T}_4^{abcd}} (\bm{\eta}^\top \vect{x}-1)^k \lambda_4(\dd\vect{x})\\
& + (1-(-1)^k) \int_{\mathbb{T}_{4*}^{abcd}} (\bm{\eta}^\top \vect{x}-1)^k \lambda_4(\dd\vect{x}).
\end{split}
\end{equation}
for any $k$ integer. These integrals are again easy to compute. Mathematica Code \ref{code:4diota2} computes $\iota^{(k)}_4(\bm{\sigma})_\mathrm{II}$ for various values of $k$. Running the code for $k=1$ and $k=3$, we obtain
\begin{align}
\iota^{(1)}_4(\bm{\sigma})_{\mathrm{II}} \!=\! \iota^{(1)}_4(\bm{\sigma})_{\mathrm{I}} \!-\! \frac{(1-d)^5}{60 d (a\!-\!d) (b\!-\!d) (c\!-\!d)}, &&
\iota^{(3)}_4(\bm{\sigma})_\mathrm{II} \!=\! \iota^{(3)}_4(\bm{\sigma})_\mathrm{I} \!-\! \frac{(1-d)^7}{420 d (a\!-\!d) (b\!-\!d) (c\!-\!d)}.
\end{align}
where the functions $\iota^{(1)}_4(\bm{\sigma})_{\mathrm{I}}$ and $\iota^{(3)}_4(\bm{\sigma})_{\mathrm{I}}$ are given by Equations \eqref{Eq:iota41} and \eqref{Eq:iota43} from the configuration $\mathrm{I}$. By denoting 
\begin{align}
    T_3^{abcd} = & \hull\left(\left[\tfrac{1}{a},0,0,0\right],\left[0,\tfrac{1}{b},0,0\right],\left[0,0,0,\tfrac{1}{c}\right],\left[0,0,0,\tfrac{1}{d}\right]\right),\\
    T_{3*}^{abcd} = & \hull\left(\left[\tfrac{1-d}{a-d},0,0,\tfrac{a-1}{a-d}\right],\left[0,\tfrac{1-d}{b-d},0,\tfrac{b-1}{b-d}\right],\left[0,0,\tfrac{1-d}{c-d},\tfrac{c-1}{c-d}\right],\left[0,0,0,\tfrac{1}{d}\right]\right),
\end{align}
we have for the intersection of $\bm{\sigma}$ with $\mathbb{T}_4$,
\begin{equation}
\begin{split}
    \bm{\sigma} \cap \mathbb{T}_4 & = T_3^{abcd}\setminus T_{3*}^{abcd} = \hull\big{(}\left[\tfrac{1}{a},0,0,0\right],\left[0,\tfrac{1}{b},0,0\right],\left[0,0,0,\tfrac{1}{c}\right],\\
    & \left[\tfrac{1-d}{a-d},0,0,\tfrac{a-1}{a-d}\right],\left[0,\tfrac{1-d}{b-d},0,\tfrac{b-1}{b-d}\right],\left[0,0,\tfrac{1-d}{c-d},\tfrac{c-1}{c-d}\right]\big{)},
\end{split}
\end{equation}
so $n_\mathrm{II} = 6$. By scale affinity
\begin{equation}
v_3^{(k+1)}(\bm{\sigma}\cap \mathbb{T}_4) = v_3^{(k+1)}(T_3^{abcd}\setminus T_{3*}^{abcd}) = v_3^{(k+1)}(\mathbb{U}_3^{\alpha\beta\gamma}),
\end{equation}
where
\begin{equation*}
\mathbb{U}_3^{\alpha\beta\gamma} = \mathbb{T}_3 \setminus \mathbb{T}_3^{\frac{1}{\alpha},\frac{1}{\beta},\frac{1}{\gamma}}= \hull([0,0,0],[1,0,0],[0,1,0],[0,0,1])\setminus\hull([0,0,0],[\alpha,0,0],[0,\beta,0],[0,0,\gamma])
\end{equation*}
is a canonical truncated tetradedron with
\begin{equation}\label{Eq:AlphaBetaGammaTrans}
    \alpha = \frac{a (1-d)}{a-d}, \qquad \beta = \frac{b (1-d)}{b-d}, \qquad \gamma = \frac{c (1-d)}{c-d}.
\end{equation}
See Figure \ref{fig:secT4ConII} below for an illustration of $\mathbb{U}_3^{\alpha\beta\gamma}$ and its volumetric moments.
\begin{figure}[H]
    \centering     \includegraphics[width=0.30\textwidth]{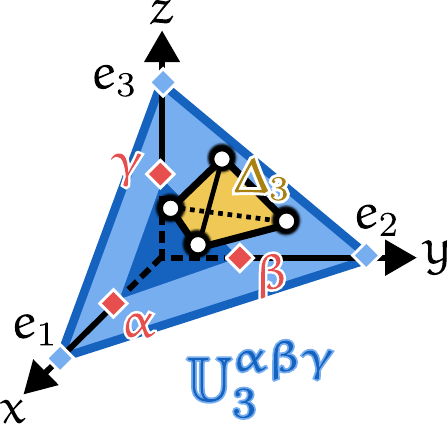}
    \caption{Mean section moments in the second $\mathcal{C}(\mathbb{T}_4)$ configuration}
    \label{fig:secT4ConII}
\end{figure}

Since $\vol_3 \mathbb{U}_3^{\alpha\beta\gamma} = \tfrac{1}{3!}(1-\alpha\beta\gamma)$, we can write in general,
\begin{equation}
v_3^{(k+1)}(\mathbb{U}_3^{\alpha\beta\gamma}) = \left(\frac{6}{1-\alpha\beta\gamma}\right)^{k+5} \int_{(\mathbb{U}_3^{\alpha\beta\gamma})^4} \Delta_3^{k+1} \ddd \vect{x}_0\dd \vect{x}_1 \dd \vect{x}_2\dd\vect{x}_3,
\end{equation}
We would like to find $v_3^{(k+1)}(\mathbb{U}_3^{\alpha\beta\gamma})$ for odd $k$. This is, luckily, trivial, since we are now integrating even powers of
\begin{equation}
\Delta_3 = \frac{1}{3!}\left|\det(\vect{x}_1-\vect{x}_0\,|\,\vect{x}_2-\vect{x}_0\,|\,\vect{x}_3-\vect{x}_0)\right|.
\end{equation}
The calculation can be carried out in Mathematica using Code \ref{code:4dvII}, which exploits the symmetries and uses inclusion/exclusion. Running the code for $k=1$, we get
\begin{equation}
    v_3^{(2)}(\mathbb{U}_3^{\alpha\beta\gamma}) = \frac{3\begin{Bmatrix}
    \!1\!+\!9 \alpha ^4 \beta ^2 \gamma ^2\!-\!16 \alpha ^5 \beta ^5 \gamma ^5\!+\!10 \alpha ^5 \beta ^5 \gamma ^4\!+\!10 \alpha ^5 \beta^4 \gamma ^5\!-\!4 \alpha ^5 \beta ^5 \gamma ^3\!-\!2 \alpha ^5 \beta ^4 \gamma ^4\!\\
    \!-\!10 \alpha ^2 \beta ^4 \gamma ^3\!+\!10 \alpha ^4 \beta ^5 \gamma ^5\!-\!10\alpha ^2 \beta ^3 \gamma ^4\!+10 \alpha ^2 \beta  \gamma\!+\!9 \alpha ^4 \beta ^4 \gamma ^2\!-\!10 \alpha ^3 \beta ^2 \gamma ^4\!\\
    \!-\!10 \alpha ^4 \beta ^3 \gamma ^2\!+\!9 \alpha ^4 \beta^2 \gamma ^4\!-\!10 \alpha ^4 \beta ^2 \gamma ^3\!-\!10\alpha ^3 \beta ^4 \gamma ^2\!-\!4 \alpha ^3 \beta ^5 \gamma ^5\!+\!2 \alpha ^3 \beta ^4 \gamma ^3\!\\
    \!-\!2\alpha ^4 \beta ^4 \gamma ^5\!+\!2 \alpha ^3 \beta ^3 \gamma ^4\!+\!2 \alpha ^3 \beta ^3 \gamma ^2\!+\!2 \alpha ^4 \beta ^3 \gamma ^3\!+\!2 \alpha ^3 \beta^2 \gamma ^3\!-\!4 \alpha ^3 \beta  \gamma\! +\!\alpha ^6 \beta ^6 \gamma ^6\!\\
    \!-\!2 \alpha ^2 \beta \gamma ^2 \!-\!4 \alpha  \beta ^3 \gamma\!-\!2 \alpha  \beta ^2 \gamma ^2\!+\!2 \alpha ^2 \beta ^3 \gamma ^3\!+\!9 \alpha ^2 \beta ^4 \gamma ^4\!+\!9 \alpha ^2 \beta ^2 \gamma ^4\!-\!2 \alpha ^2 \beta ^2 \gamma \!\\
    \!-\!4 \alpha ^5 \beta ^3 \gamma ^5\!+\!9 \alpha ^2 \beta ^4 \gamma ^2\!-\!2 \alpha ^4 \beta ^5 \gamma ^4\!+\!10 \alpha  \beta ^2 \gamma \!-\!4 \alpha  \beta \gamma ^3\!+\!10 \alpha  \beta  \gamma ^2\!-\!16 \alpha  \beta  \gamma \!    
    \end{Bmatrix}}{4000(1-\alpha  \beta \gamma )^6},
\end{equation}

Finally, by definition (alternatively by Equation \eqref{Eq:ZetaEta})
\begin{equation}
    \zeta_4(\bm{\sigma})_\mathrm{II} = \frac{\vol_3(\bm{\sigma}\cap \mathbb{T}_4)}{\| \bm{\eta}\| \vol_4 \mathbb{T}_4} = (1-\alpha\beta\gamma)\frac{\vol_3 T_3^{abc}}{\| \bm{\eta}\| \vol_4 \mathbb{T}_4} = (1-\alpha\beta\gamma) \zeta_4(\bm{\sigma})_\mathrm{I}.
\end{equation}
Before we proceed to evaluate the final integral, we make the following change of variables $(a,b,c,d) \to (\alpha,\beta,\gamma,d)$ via transformation Equations \eqref{Eq:AlphaBetaGammaTrans}, which transform the integration half-domain into
\begin{equation}
(\mathbb{R}^4\setminus \mathbb{T}_4^\circ)^{*}_\mathrm{II}\,|_{\alpha,\beta\gamma,d} = (1-d,1)^3 \times (0,1).
\end{equation}
Note that, if $d$ is treated as a parameter, the variables $a,b,c$ depend on $\alpha,\beta,\gamma$ separately. As a consequence,
\begin{equation}
    \dd a = \frac{d(1-d) \ddd \alpha}{(1-d-\alpha)^2}, \qquad \dd b = \frac{d(1-d) \ddd \beta}{(1-d-\beta)^2}, \qquad \dd c = \frac{d(1-d) \ddd \gamma}{(1-d-\alpha)^2}
\end{equation}
and thus one has for the of transformation of measure
\begin{equation}
\lambda_4(\dd \bm{\eta}) = \dd a \ddd b  \ddd c \ddd d = \frac{d^3(1-d)^3 \ddd \alpha  \ddd \beta\ddd\gamma\ddd d}{(1-d-\alpha)^2(1-d-\beta)^2(1-d-\gamma)^2}.
\end{equation}
Putting everything into the integral in Equation \eqref{Eq:SecInt4}, we get when $k=1$ and after integrating out $\alpha, \beta, \gamma$,
\begin{equation}
\begin{split}
& v_4^{(1)}(\mathbb{T}_4)_\mathrm{II} = \frac{1}{1406250} \int_0^1 \big{(} d^3 p_0+180 d^2 (1-d)^3 p_1 \ln (1-d)\\
& +10800 d (1-d)^3 p_2 \ln ^2(1-d)+216000 (1-d)^3 p_3 \ln ^3(1-d) \big{)} \frac{\dd d}{d^{25}},
\end{split}
\end{equation}
where
\bgroup
\allowdisplaybreaks
\begin{align}
\begin{split}
  p_0     & = 32480784000-324807840000 d+1556229024000 d^2-4749037776000 d^3\\
  & +10279357367400 d^4-16555175611200 d^5+20253161331700 d^6-18987688381900 d^7\\
  &+13740024940130 d^8-7798431753680 d^9+3604300565845 d^{10}-1440768739775 d^{11}\\
  &+518639866862 d^{12}-161581999478 d^{13}+39317696413 d^{14}-6685392751 d^{15}\\
  & +700753210 d^{16}-34837616 d^{17} +6112 d^{18}-3272 d^{19}+784 d^{20},
\end{split}
\\[2ex]
\begin{split}
p_1  & = 541346400-4060098000 d+14794437000 d^2-34585687500 d^3\!+\!56747312360 d^4\!\\
&-\!67139592080 d^5\!+\!57686267770 d^6\!-\!36408101115 d^7+17574730626 d^8-7114914681 d^9\\
&+\!2659305113 d^{10}\!-\!888330365 d^{11}\!+\!229856455 d^{12}\!-\!40385468 d^{13}\!+\!4279933 d^{14}\!-\!213224 d^{15},
\end{split}\\[2ex]
\begin{split}
p_2  & = 9022440\!-\!72179520 d\!+\!279656230 d^2\!-\!694452010 d^3\!+\!1216036193 d^4-1552509188 d^5\\
&+1460599749 d^6-1021377960 d^7+544097150 d^8-234903968 d^9+90292498d^{10}\\
&-32050399 d^{11}+9632345 d^{12}-2161105 d^{13}+327799 d^{14}-30254 d^{15}+1312 d^{16},
\end{split}
\\[2ex]
\begin{split}
p_3    & = 150374-1278179 d+5249902 d^2-13810685 d^3+25712115 d^4-\!35209551 d^5\!\\
&+\!35968805 d^6\!-\!27633760 d^7\!+\!16221440 d^8\!-\!7575685 d^9+3035423 d^{10}-1117957d^{11}\\
&+369741 d^{12}-99030 d^{13}+19440 d^{14}-2588 d^{15}+211 d^{16}-8 d^{17}.
\end{split}
\end{align}
\egroup
The last $d$ integration can be carried out by Mathematica (or tediously using Beta function derivatives). We get
\begin{equation}
    v_4^{(1)}(\mathbb{T}_4)_\mathrm{II} = \frac{89}{270000}-\frac{2173 \pi ^2}{520269750}.
\end{equation}
For higher values of $k$, the integration possesses similar difficulty, we got
\begin{equation}
    \begin{split}
   v_4^{(3)}(\mathbb{T}_4)_\mathrm{II} & = \textstyle\frac{3947568673}{80048379600000000}+\frac{63065881 \pi ^2}{396699961407750000},\\[1ex]
   v_4^{(5)}(\mathbb{T}_4)_\mathrm{II} & = \textstyle\frac{700536944899}{7058653305387264000000}-\frac{1262701803371 \pi ^2}{35570432728713733250400000}.
    \end{split}
\end{equation}

\subsubsection{Contribution from all configurations}
By Equation \eqref{Eq:DecoCon} and by affine invariancy,
\begin{equation}
v_4^{(k)}(T_4) = \sum_{C \in \mathcal{C}(T_4)} w_C \, v_4^{(k)}(T_4)_C = 5 v_4^{(k)}(\mathbb{T}_4)_{\mathrm{I}} + 10v_4^{(k)}(\mathbb{T}_4)_{\mathrm{II}},
\end{equation}
from which immediately
\begin{align}
   v_4^{(1)}(T_4) & = \frac{97}{27000}-\frac{2173 \pi ^2}{52026975} \approx 0.0031803708487,\\[2ex]
   v_4^{(3)}(T_4) & = \frac{1955399}{3403417500000}+\frac{63065881 \pi ^2}{39669996140775000} \approx 5.9023\cdot 10^{-7}, \\[2ex]
   v_4^{(5)}(T_4) & = \textstyle\frac{12443146181}{9803685146371200000}-\frac{1262701803371 \pi ^2}{3557043272871373325040000} \approx 1.26573\cdot 10^{-9}.
\end{align}
Monte-Carlo simulation shows that the value $v_4^{(1)}(T_4)$ fits withing the $95\%$ confidence interval $(0.00318034, 0.00318043)$ obtained from $4\times 10^{10}$ trials of randomly generated $4$-simplices in $T_4$.

\vspace{1em}
Moreover, by Buchta's relation (Equation \eqref{Eq:Buchta}), we get the value of mean $4$-volume of a convex hull of $6$ points in the unit pentachoron as
\begin{equation}
    v_5^{(1)}(T_4) = 3 v_4^{(1)}(T_4) = \frac{97}{9000}-\frac{2173 \pi ^2}{17342325} \approx 0.00954111.
\end{equation}

\subsection{Hexadecachoron first volumetric moment}
The \emph{hexadecachoron}\index{hexadecachoron} or a $16$-cell\index{16-cell} are alternative names of $4$-orthoplex\index{4-orthoplex} $O_4$, a polychoron with standard representation with $\vol_4 O_4 = 2/3$,
\begin{equation}\label{Eq:4orthodef}
O_4 = \hull([1,\!0,\!0,\!0],\![0,\!1,\!0,\!0][0,\!0,\!1,\!0][0,\!0,\!0,\!1],\![-1,\!0,\!0,\!0],\![0,\!-1,\!0,\!0][0,\!0,\!-1,\!0][0,\!0,\!0,\!-1]).
\end{equation}
The symmetry group $\mathcal{G}(O_4)$ is isomorphic to Coxeter group $\mathcal{B}_4$ of order $|\mathcal{B}_4| = 384$. We can describe the symmetry group using its four generators (one reflection, two rotations and one double rotation) of permutations acting on vertices indexed as in Equation \eqref{Eq:4orthodef}. In cycle notation\index{cycle notation} (excluding fixed points), we have
\begin{equation}
   \mathcal{G}(O_4) = \langle (48),(2367),(1256),(1256)(3478)
   \rangle < \mathcal{S}_8,
\end{equation}
where $\langle \cdot \rangle$ denotes the \emph{algebraic closure}\index{algebraic closure} and $<$ the relation of being a subgroup. From this group, we can generate $14$ configurations, out of which only $4$ are realisable and section equivalent. These consist the genealogy $\mathcal{C}(C_4)$. Table \ref{tab:4orthovol} shows specifically which sets $S$ of vertices in which configurations are separated by a cutting plane $\bm{\sigma}$ in our standard representation of $O_4$ in Equation \eqref{Eq:4orthodef}.
\begin{table}[htb]
    \centering
\begin{tabular}{|c|c|c|c|c|}
\hline
 $C$ & $\mathrm{I}$ & $\mathrm{II}$ & $\mathrm{III}$ & $\mathrm{IV}$ \\
 \hline
 $S$ & $[0,0,0,1]$ & \begin{tabular}{c} $[0,0,1,0]$ \\ $[0,0,0,1]$ \end{tabular} & \begin{tabular}{c} $[0,1,0,0]$\\ $[0,0,1,0]$\\ $[0,0,0,1]$ \end{tabular} & \begin{tabular}{c} $[1,0,0,0]$\\ $[0,1,0,0]$\\ $[0,0,1,0]$\\ $[0,0,0,1]$ \end{tabular} \\
 \hline 
 $w_C$ & $8$ & $24$ & $32$ & $16$\\
 $n_C$ & $6$ & $10$ & $12$ & $0$\\
 \hline
\end{tabular}
    \caption{Configurations $\mathcal{C}(O_4)$ in the standard representation of $O_4$.}
    \label{tab:4orthovol}
\end{table}
By similar treatment as in the case of $O_3$, we can easily find inequalities which describe $O_4^+$ and thus $\bm{\sigma}\cap O_4^+$. We only list the section integrals obtained from all configurations, see Table \ref{tab:4orthoAll}.
\begin{table}[H]
    \centering
\begin{tabular}{|c|c|}
\hline
\ru{1.1} $C$ & $v_4^{(1)}(O_4)_C$ \\
 \hline
 \ru{1.1}$\mathrm{I}$ & $\frac{2400441939 \zeta (3)}{320000}-\frac{71765769458062825751339}{8136689713152000000}+\frac{173050612310219 \pi^2}{3547315200000}-\frac{127327345788535137 \ln 2}{130068224000000}$ \\[0.6ex]
 \hline 
 \ru{1.1}$\mathrm{II}$ & $\frac{11577920188509587165389181}{2072472081039360000000}-\frac{13611484420925379 \zeta
   (3)}{2928808960000}+\frac{6998756
   6888072781151 \pi ^2}{1461358518681600000}-\frac{71866300533 \pi ^2 \ln 2}{1040060000}$ \\[0.6ex]
 \hline
 \ru{1.1}$\mathrm{III}$ & (not yet derived) \\[0.6ex]
 \hline
 \ru{1.1}$\mathrm{IV}$ & (not yet derived) \\[0.6ex]
 \hline
\end{tabular}
    \caption{Sections integrals in various configurations $\mathcal{C}(O_4)$.}
    \label{tab:4orthoAll}
\end{table}

\begin{remark}
As of now, we have not found the expressions for $v_4^{(1)}(O_4)_C$ for configurations $C\in \{\mathrm{III,IV}\}$, we have succeeded in writing them as explicit double integrals, but the shear scope of them have not enabled us to calculate using our own computers. However, we think this might be doable and will be part of our future papers. We have also attempted to find higher odd moments, however, the section integrals became too complicated. The third and the fifth moment are in principle derivable but it would be extraordinarily time consuming. We found at least in the first configuration
\begin{equation}
\begin{split}
    v_4^{(3)}(O_4)_\mathrm{I} & = \textstyle\frac{8928188080691679 \zeta
   (3)}{7867596800000}-\frac{13757679936170496961418065762637875149511}{10097679414187456780038045696000000000} \\[-0.2ex]
 & \textstyle+\frac{420783881199433246283869 \pi ^2}{1357358340088791040000000}-\frac{138200770459501589499358193329 \ln 2}{20380735476433197465600000000}
\end{split}    
\end{equation}
\end{remark}

\subsection{Tesseract odd volumetric moments}
By \emph{tesseract}\index{tesseract}, we mean $C_4$ ($4$-cube)\index{4-cube}. The standard representation of the unit tesseract with $\vol_4 C_4 = 1$ is
\begin{equation}\label{Eq:4cubedef}
\begin{split}
    C_4  = \hull(& [0,0,0,0],\![1,0,0,0],\![0,1,0,0],\![0,0,1,0],\![0,0,0,1],\![1,1,0,0],\![1,0,1,0],\![1,0,0,1],\\
    & [0,1,1,0],\![0,1,0,1],\![0,0,1,1],\![1,1,1,0],[1,1,0,1],\![1,0,1,1],\![0,1,1,1],\![1,1,1,1]).
\end{split}
\end{equation}
The symmetry group $\mathcal{G}(C_4)$ is isomorphic to Coxeter group $\mathcal{B}_4$ of order $|\mathcal{B}_4| = 384$. We can describe the symmetry group using its four generators (one reflection, two rotations and one double rotation) of permutations acting on vertices indexed as in Equation \eqref{Eq:4cubedef}. In cycle notation\index{cycle notation}, we have
\begin{equation}
\begin{split}
   \mathcal{G}(C_4) = \langle & (1,5)(2,8)(3,10)(4,11)(6,13)(7,14)(9,15)(12,16)(1,3,9,4)(2,6,12,7)\\
   & (5,10,15,11)(8,13,16,14)(1,2,6,3)(4,7,12,9)(5,8,13,10)\\
   & (11,14,16,15)(1,7,16,10)(2,12,15,5)(3,4,14,13)(6,9,11,8)
   \rangle < \mathcal{S}_{16}.
\end{split}
\end{equation}
From this group, we can generate $402$ configurations, out of which $14$ are realisable and section equivalent. These consist the genealogy $\mathcal{C}(C_4)$. Table \ref{tab:Tessvol} shows specifically which sets $S$ of vertices in which configurations are separated by a cutting plane $\bm{\sigma}$ in our standard representation of $C_4$ in Equation \eqref{Eq:4cubedef}.

\begin{table}[!tbh]
    \centering
\begin{tabular}{|c|c|c|c|c|c|c|c|}
\hline
 \!\!$C$\!\! & $\mathrm{I}$ & $\mathrm{II}$ & $\mathrm{III}$ & $\mathrm{IV}$ & $\mathrm{V}$ & $\mathrm{VI}$ & $\mathrm{VII}$\\
 \hline
 \!\!$S$\!\! & $[0,0,0,0]$ & \begin{tabular}{c} $[0,0,0,0]$ \\ $[1,0,0,0]$ \end{tabular} & \begin{tabular}{c} $[0,0,0,0]$\\ $[1,0,0,0]$\\ $[0,1,0,0]$ \end{tabular} & \begin{tabular}{c} $[0,0,0,0]$\\ $[1,0,0,0]$\\ $[0,1,0,0]$\\ $[1,1,0,0]$ \end{tabular} & \begin{tabular}{c} $[0,0,0,0]$\\ $[1,0,0,0]$\\ $[0,1,0,0]$\\ $[0,0,1,0]$ \end{tabular} & \begin{tabular}{c} $[0,0,0,0]$\\ $[1,0,0,0]$\\ $[0,1,0,0]$\\ $[0,0,1,0]$ \\ $[0,0,0,1]$ \end{tabular} & \begin{tabular}{c} $[0,0,0,0]$\\ $[1,0,0,0]$\\ $[0,1,0,0]$\\ $[0,0,1,0]$ \\ $[1,1,0,0]$ \end{tabular}\\
 \hline 
 \!\!$w_C$\!\! & $16$ & $32$ & $96$ & $24$ & $64$ & $16$ & $192$\\
 \!\!$n_C$\!\! & $4$ & $6$ & $8$ & $8$ & $10$ & $12$ & $10$\\
 \hline
 \hline
\!\!$C$\!\! & $\mathrm{VIII}$ & $\mathrm{IX}$ & $\mathrm{X}$ & $\mathrm{XI}$ & $\mathrm{XII}$ & $\mathrm{XIII}$ & $\mathrm{XIV}$\\
 \hline
 \!\!$S$\!\! & \!\begin{tabular}{c} $[0,0,0,0]$\\ $[1,0,0,0]$\\ $[0,1,0,0]$\\ $[0,0,1,0]$ \\ $[1,0,1,0]$ \\ $[1,1,0,0]$ \end{tabular}\!\!\!\! & \!\!\!\!\begin{tabular}{c} $[0,0,0,0]$\\ $[1,0,0,0]$\\ $[0,1,0,0]$\\ $[0,0,1,0]$ \\ $[0,0,0,1]$ \\ $[1,1,0,0]$ \end{tabular}\!\!\!\! & \!\!\!\!\begin{tabular}{c} $[0,0,0,0]$\\ $[1,0,0,0]$\\ $[0,1,0,0]$\\ $[0,0,1,0]$ \\ $[0,1,1,0]$ \\ $[1,0,1,0]$ \\ $[1,1,0,0]$ \end{tabular}\!\!\!\! & \!\!\!\!\begin{tabular}{c} $[0,0,0,0]$\\ $[1,0,0,0]$\\ $[0,1,0,0]$\\ $[0,0,1,0]$ \\ $[0,0,0,1]$ \\ $[1,0,1,0]$ \\ $[1,1,0,0]$ \end{tabular}\!\!\!\! & \!\!\!\!\begin{tabular}{c} $[0,0,0,0]$\\ $[1,0,0,0]$\\ $[0,1,0,0]$\\ $[0,0,1,0]$ \\$[0,1,1,0]$\\ $[1,0,1,0]$\\ $[1,1,0,0]$\\ $[1,1,1,0]$ \end{tabular}\!\!\!\! & \!\!\!\!\begin{tabular}{c} $[0,0,0,0]$\\ $[1,0,0,0]$\\ $[0,1,0,0]$\\ $[0,0,1,0]$ \\$[0,0,0,1]$\\ $[1,1,0,0]$\\ $[1,0,1,0]$\\ $[1,0,0,1]$ \end{tabular}\!\!\!\! & \!\!\!\!\begin{tabular}{c} $[0,0,0,0]$\\ $[1,0,0,0]$\\ $[0,1,0,0]$\\ $[0,0,1,0]$ \\$[0,0,0,1]$\\ $[0,1,1,0]$\\ $[1,0,1,0]$\\ $[1,1,0,0]$ \end{tabular}\!\!\!\!\\
 \hline 
 \!\!$w_C$\!\! & $96$ & $96$ & $64$ & $192$ & $4$ & $32$ & $64$\\
 \!\!$n_C$\!\! & $10$ & $12$ & $10$ & $12$ & $8$ & $12$ & $12$\\
 \hline
\end{tabular}
    \caption{Configurations $\mathcal{C}(C_4)$ in the standard representation of $C_4$.}
    \label{tab:Tessvol}
\end{table}

By similar treatment as in the case of $O_4$, we can easily find inequalities which describe $C_4^+$ and thus $\bm{\sigma}\cap C_4^+$. We only list the section integrals obtained from all configurations, see Table \ref{tab:4cubeAll}. Also, for brevity, we only enlist the first volumetric moment, although we found also $v_4^{(3)}(C_4)_C$ for all configurations. For example $v_4^{(3)}(C_4)_\mathrm{I} = \frac{573495143}{783231158555529707520000}$. It turns out the last configuration $\mathrm{XIV}$ is tricky to integrate. In the end, one has to use the identity involving \emph{trilogarithms}\index{trilogarithm} found (rediscovered) by Shobhit Bhatnagar \cite{trilogLi}, the identity states that
\begin{equation}
\operatorname{Li}_3\left(-\frac{1}{3}\right)-2 \operatorname{Li}_3\left(\frac{1}{3}\right) = -\frac{\ln^3 3}{6}+\frac{\pi^2}{6}\ln 3-\frac{13\zeta(3)}{6}.
\end{equation}.
\begin{table}[H]
    \centering
\begin{tabular}{|c|c|}
\hline
\ru{1.1} $C$ & $v_4^{(1)}(C_4)_C$ \\
 \hline
 \ru{1.1}$\mathrm{I}$ & $\frac{65598041}{3386742443900928000000}$ \\[0.6ex]
 \hline 
 \ru{1.1}$\mathrm{II}$ & $\frac{102608713871}{3292649334374400000}$ \\[0.6ex]
 \hline
 \ru{1.1}$\mathrm{III}$ & $\frac{256081766015430731}{345728180109312000000}-\frac{6302191 \pi ^2}{83980800000}$ \\[0.6ex]
 \hline
 \ru{1.1}$\mathrm{IV}$ & $\frac{7383631}{1862358220800}$ \\[0.6ex]
 \hline
 \ru{1.1}$\mathrm{V}$ & $\frac{74369 \zeta (3)}{92160000}-\frac{15427192177655450593}{2304854534062080000000}+\frac{31318807 \pi
   ^2}{149299200000}+\frac{482072643302197 \ln 2}{91462481510400000}$ \\[0.6ex]
 \hline
 \ru{1.1}$\mathrm{VI}$ & $-\frac{1663466629 \zeta (3)}{622080000}-\frac{210954160717218293347879}{6338349968670720000000}-\frac{133847 \pi
   ^2}{124416000}+\frac{2007170664939114317 \ln 2}{38109367296000000}$ \\[0.6ex]
 \hline
 \ru{1.1}$\mathrm{VII}$ & $\frac{388451 \zeta (3)}{29859840}+\frac{596684331816745397}{29933175767040000000}+\frac{4354897 \pi
   ^2}{1343692800000}-\frac{23489337302150729 \ln 2}{457312407552000000}$ \\[0.6ex]
 \hline
 \ru{1.1}$\mathrm{VIII}$ & $\frac{188122446351063331}{10975497781248000000}-\frac{1170683 \pi ^2}{671846400}+\frac{221036483033 \ln 2}{2494431313920000}$ \\[0.6ex]
 \hline
 \ru{1.1}$\mathrm{IX}$ & $-\frac{618197167 \zeta (3)}{1866240000}+\frac{373791108546507725849549}{38030099812024320000000}+\frac{74238971 \pi
   ^2}{671846400000}-\frac{1333435310218723619 \ln 2}{97995515904000000}$ \\[0.6ex]
 \hline
 \ru{1.1}$\mathrm{X}$ & $\frac{2274497329 \zeta (3)}{69120000}-\frac{21609245552433862937}{4390199112499200000}-\frac{1523317655658026279 \ln
2}{30487493836800000}$ \\[0.6ex]
 \hline
 \ru{1.1}$\mathrm{XI}$ & $\frac{24570427 \zeta (3)}{55296000}-\frac{157440595529232693016981}{76060199624048640000000}+\frac{47205929 \pi
   ^2}{24883200000}+\frac{3002774140883958709 \ln 2}{1371937222656000000}$ \\[0.6ex]
 \hline
 \ru{1.1}$\mathrm{XII}$ & $\frac{17}{311040}$ \\[0.6ex]
 \hline
 \ru{1.1}$\mathrm{XIII}$ & $\frac{746581063847040871}{6602447884032000000}-\frac{641346209 \pi ^2}{55987200000}$ \\[0.6ex]
 \hline
 \ru{1.1}$\mathrm{XIV}$ & $-\frac{41203109797 \zeta (3)}{622080000}+\frac{10605967272168022814803}{1152427267031040000000}-\frac{12193153 \pi
   ^2}{27993600000}+\frac{4645960252158518597 \ln 2}{45731240755200000}$ \\[0.6ex]
 \hline
\end{tabular}
    \caption{Sections integrals in various configurations $\mathcal{C}(C_4)$.}
    \label{tab:4cubeAll}
\end{table}

By Equation \eqref{Eq:DecoCon}, considering the contributions from all configurations,
\begin{equation}
\begin{split}
v_4^{(k)}(C_4) & = \!\!\! \sum_{C \in \mathcal{C}(C_4)} \!\! w_C \, v_4^{(k)}(C_4)_C = \! 16 v_4^{(k)}(C_4)_{\mathrm{I}} \! + \! 32v_4^{(k)}(C_4)_{\mathrm{II}} \! + \! 96v_4^{(k)}(C_4)_{\mathrm{III}}  \!+\! 24v_4^{(k)}(C_4)_{\mathrm{IV}} \\
& \!+\! 64v_4^{(k)}(C_4)_{\mathrm{V}} + 16v_4^{(k)}(C_4)_{\mathrm{VI}} + 192v_4^{(k)}(C_4)_{\mathrm{VII}} + 96v_4^{(k)}(C_4)_{\mathrm{VIII}} + 96v_4^{(k)}(C_4)_{\mathrm{IX}}\\
&  + 64v_4^{(k)}(C_4)_{\mathrm{X}} + 192v_4^{(k)}(C_4)_{\mathrm{XI}} + 4v_4^{(k)}(C_4)_{\mathrm{XII}} + 32v_4^{(k)}(C_4)_{\mathrm{XIII}} + 64v_4^{(k)}(C_4)_{\mathrm{XIV}},
\end{split}
\end{equation}
from which immediately
\begin{align}
   \!\!\!v_4^{(1)}(C_4) & \!=\! \textstyle\frac{31874628962521753237}{1058357013719040000000}-\frac{26003 \pi^2}{1399680000}+\frac{610208 \ln 2}{1913625} -\frac{536557 \zeta (3)}{2592000}\approx 0.00212952943564458\\
   \!\!\!v_4^{(3)}(C_4) & \!= \!\textstyle\frac{19330626155629115959}{1682723192209145856000000}\!-\!\frac{52276897 \pi^2}{216801070940160000}\!+\!\frac{10004540239 \ln 2}{77977156950000}\!-\!\frac{6155594561 \zeta (3)}{73741860864000}\!\approx\! 7.5157\!\cdot\! 10^{-8}.\!\!
%   v_4^{(5)}(C_4) & = \textstyle  \approx .
\end{align}

\FloatBarrier

\subsection{Hexateron odd volumetric moments}
By the \emph{hexateron}\index{hexateron}, we mean $T_5$ ($5$-simplex\index{5-simplex}). By affine invariancy, we may consider
\begin{equation}
    \mathbb{T}_5 = \hull(\vect{0},\vect{e}_1,\vect{e}_2,\vect{e}_3,\vect{e}_4,\vect{e}_5)
\end{equation}
with configurations and $\mathcal{C}(T_5)$ weights given by Table \ref{tab:Hexater}.

\begin{table}[H]
    \centering
\begin{tabular}{|c|c|c|c|}
\hline
 $C$ & $\mathrm{I}$ & $\mathrm{II}$ & $\mathrm{III}$\\
 \hline
 $S$ & $[0,0,0,0,0]$ & \begin{tabular}{c} $[0,0,0,0,0]$ \\ $[0,0,0,0,1]$ \end{tabular} & \begin{tabular}{c} $[0,0,0,0,0]$ \\ $[0,0,0,1,0]$ \\ $[0,0,0,0,1]$ \end{tabular}\\
 \hline 
 $w_C$ & $6$ & $15$ & $10$\\%[-0.5ex]
% $n_C$ & $5$ & $8$ & $9$\\
 \hline
\end{tabular}
    \caption{Configurations $\mathcal{C}(\mathbb{T}_5)$ in a local representation with $\mathcal{C}(T_5)$ weights.}
    \label{tab:Hexater}
\end{table}

By Theorem \ref{Thm:Canon} and for any $C \in \mathcal{C}(\mathbb{T}_5)$,
\begin{equation}\label{Eq:SecIntHexa}
v_5^{(k)}(\mathbb{T}_5)_C = \frac{24}{ 5^k}\int_{(\mathbb{R}^5\setminus \mathbb{T}_5^\circ)_C} v_4^{(k+1)}(\bm{\sigma}\cap \mathbb{T}_5) \,\zeta_5^{k+6}(\bm{\sigma}) \iota^{(k)}_5(\bm{\sigma}) \lambda_5(\dd \bm{\eta}),
\end{equation}
where
\begin{equation}
    \zeta_5(\bm{\sigma}) = \frac{\vol_4(\bm{\sigma}\cap \mathbb{T}_5)}{\|\bm{\eta}\|\vol_5 \mathbb{T}_5}, \qquad \iota^{(k)}_5(\bm{\sigma}) = \int_{\mathbb{T}_5} |\bm{\eta}^\top \vect{x}-1|^k \lambda_5(\dd\vect{x}).
\end{equation}
Configurations $\mathrm{I}$ and $\mathrm{II}$ are analogous to the first two configurations of $\mathbb{T}_3$ and $\mathbb{T}_4$, we have $n_\mathrm{I} = 5$ and $n_\mathrm{II} = 2n_\mathrm{I}-2 = 8$ (truncated 4-simplex). The last configuration $\mathrm{III}$, for which we have $n_\mathrm{III} = 9$, has no analogue in lower dimensions. However, by similar procedure as before, we obtained contributions from all configurations, see Table \ref{tab:HexaAll}.
\begin{table}[H]
    \centering
\begin{tabular}{|c|c|}
\hline
 $C$ & $v_5^{(1)}(\mathbb{T}_5)_C$ \\
 \hline
 \ru{1.1}$\mathrm{I}$ & $\frac{5}{2722734}$ \\[0.6ex]
 \hline 
 \ru{1.1}$\mathrm{II}$ & $\frac{12732911}{653456160000}-\frac{1394234873 \pi ^2}{3353951824423200}+\frac{1622 \pi ^4}{2707566616755}$ \\[0.6ex]
 \hline
 \ru{1.1}$\mathrm{III}$ & $\frac{146034151}{3920736960000}-\frac{3546684881 \pi ^2}{3353951824423200}+\frac{4904 \pi ^4}{386795230965}$ \\[0.6ex]
 \hline
\hline
 $C$ & $v_5^{(3)}(\mathbb{T}_5)_C$ \\
 \hline
 \ru{1.1}$\mathrm{I}$ & $\frac{9097367105}{359796813461446459392}$ \\[0.6ex]
 \hline 
 \ru{1.1}$\mathrm{II}$ & $\frac{25351944803581}{245954852952160665600000}+\frac{204046383487590493 \pi ^2}{98081004264127779106308096000}+\frac{13583435573 \pi
   ^4}{17098021963979168381769600}$ \\[0.6ex]
 \hline
 \ru{1.1}$\mathrm{III}$ & $\frac{173514729599507}{874506143829904588800000}-\frac{12027338819078269 \pi ^2}{9341048025155026581553152000}+\frac{1191143596913 \pi
   ^4}{11398681309319445587846400}$ \\[0.6ex]
 \hline
\end{tabular}
    \caption{Sections integrals in various configurations $\mathcal{C}(\mathbb{T}_5)$.}
    \label{tab:HexaAll}
\end{table}
As a consequence, summing up the contributions from all configurations and by affine invariancy,
\begin{equation}
v_5^{(k)}(T_5) = \sum_{C \in \mathcal{C}(T_5)} w_C \, v_5^{(k)}(T_5)_C = 6 v_5^{(k)}(\mathbb{T}_5)_{\mathrm{I}} + 15v_5^{(k)}(\mathbb{T}_5)_{\mathrm{II}} + 10v_5^{(k)}(\mathbb{T}_5)_{\mathrm{III}},
\end{equation}
from which immediately
\begin{align}
    v_5^{(1)}(T_5) & = \frac{2207}{3265920}-\frac{244129 \pi ^2}{14522729760}+\frac{73522 \pi ^4}{541513323351} \approx 0.00052308272,\\
    v_5^{(3)}(T_5) & = \textstyle\frac{362173019}{98363448852480000}+\frac{10217818563857 \pi ^2}{557436796045056999751680}+\frac{602363516243 \pi^4}{569934065465972279392320} \approx 3.96585\cdot10^{-9}.
\end{align}

\begin{remark}
Higher volumetric moments are difficult to compute. For the fifth moment, we would need $v_5^{(5)}(\mathbb{T}_5)_\mathrm{III}$. However, even $v_5^{(3)}(\mathbb{T}_5)_\mathrm{III}$ was already extremely difficult to compute (the file we worked with exceeded 1GB of storage memory). The intricacy of the third configuration stems partly from its asymmetry and from lacking the decoupling substitution $(a \to \alpha, b \to \beta,c\to \gamma, d \to \delta)$, which we found in the second configuration of $T_4$ (and which generalises as well into higher dimensions) and which enables us to integrate out $\alpha,\beta,\gamma,\delta$ immediately. We have not attempted to obtain the fifth moment, such calculation is surely within our grasp but the shear monstrosity of $v^{(6)}(\bm{\sigma} \cap \mathbb{T}_5)$ in Configuration $\mathrm{III}$ discourages us to finish the computation.

\subsection{Heptapeton first volumetric moment}
By the \emph{heptapeton}\index{heptapeton}, we mean $T_6$ ($6$-simplex\index{6-simplex}). By affine invariancy, we may consider
\begin{equation}
    \mathbb{T}_6 = \hull(\vect{0},\vect{e}_1,\vect{e}_2,\vect{e}_3,\vect{e}_4,\vect{e}_5,\vect{e}_6)
\end{equation}
with configurations and $\mathcal{C}(T_6)$ weights given by Table \ref{tab:Heptapent}.

\begin{table}[htb]
    \centering
\begin{tabular}{|c|c|c|c|}
\hline
 $C$ & $\mathrm{I}$ & $\mathrm{II}$ & $\mathrm{III}$\\
 \hline
 $S$ & $[0,0,0,0,0,0]$ & \begin{tabular}{c} $[0,0,0,0,0,0]$ \\ $[0,0,0,0,0,1]$ \end{tabular} & \begin{tabular}{c} $[0,0,0,0,0,0]$ \\ $[0,0,0,0,0,1]$ \\ $[0,0,0,0,1,0]$ \end{tabular}\\
 \hline 
 $w_C$ & $7$ & $21$ & $35$\\%[-0.5ex]
% $n_C$ & $6$ & $10$ & $b$\\
 \hline
\end{tabular}
    \caption{Configurations $\mathcal{C}(\mathbb{T}_6)$ in a local representation with $\mathcal{C}(T_6)$ weights.}
    \label{tab:Heptapent}
\end{table}
By Theorem \ref{Thm:Canon} and for any $C \in \mathcal{C}(\mathbb{T}_6)$,
\begin{equation}\label{Eq:SecIntHepta}
v_6^{(k)}(\mathbb{T}_6)_C = \frac{120}{6^k}\int_{(\mathbb{R}^6\setminus \mathbb{T}_6^\circ)_C} v_5^{(k+1)}(\bm{\sigma}\cap \mathbb{T}_6) \,\zeta_6^{k+7}(\bm{\sigma}) \iota^{(k)}_6(\bm{\sigma}) \lambda_6(\dd \bm{\eta}),
\end{equation}
where
\begin{equation}
    \zeta_6(\bm{\sigma}) = \frac{\vol_5(\bm{\sigma}\cap \mathbb{T}_6)}{\|\bm{\eta}\|\vol_6 \mathbb{T}_6}, \qquad \iota^{(k)}_6(\bm{\sigma}) = \int_{\mathbb{T}_6} |\bm{\eta}^\top \vect{x}-1|^k \lambda_6(\dd\vect{x}).
\end{equation}
Configurations $\mathrm{I}$ and $\mathrm{II}$ are analogous to the first two configurations of $\mathbb{T}_3$, $\mathbb{T}_4$ and $\mathbb{T}_5$, we have $n_\mathrm{I} = 6$ and $n_\mathrm{II} = 2n_\mathrm{I}-2 = 10$ (truncated 5-simplex). The last configuration $\mathrm{III}$ is analogous to third configuration of $\mathbb{T}_5$. We have $n_\mathrm{III} = 12$. Thanks to this similarity, since we already know how to handle this configuration in the $\mathbb{T}_5$ case, we obtained contributions of all $\mathbb{T}_6$ configurations, see Table \ref{tab:HeptaAll}.
\begin{table}[H]
    \centering
\begin{tabular}{|c|c|}
\hline
 $C$ & $v_6^{(1)}(\mathbb{T}_6)_C$ \\
 \hline
 \ru{1.1}$\mathrm{I}$ & $\frac{45}{963780608}$ \\[0.6ex]
 \hline 
 \ru{1.1}$\mathrm{II}$ & $\frac{3826171}{4182119424000}-\frac{12560362004329 \pi ^2}{443562265371500795520}+\frac{6607326855286 \pi ^4}{85176183364279644451815}$ \\[0.6ex]
 \hline
 \ru{1.1}$\mathrm{III}$ & $\frac{71529389}{24395696640000}-\frac{4625576448278719 \pi ^2}{33267169902862559664000}+\frac{432402941059748 \pi ^4}{141960305607132740753025}$ \\[0.6ex]
 \hline
\end{tabular}
    \caption{Sections integrals in various configurations $\mathcal{C}(\mathbb{T}_6)$.}
    \label{tab:HeptaAll}
\end{table}
As a consequence, summing up the contributions from all configurations and by affine invariancy,
\begin{equation}
v_6^{(k)}(T_6) = \sum_{C \in \mathcal{C}(T_6)} w_C \, v_6^{(k)}(T_6)_C = 7 v_6^{(k)}(\mathbb{T}_6)_{\mathrm{I}} + 21v_6^{(k)}(\mathbb{T}_6)_{\mathrm{II}} + 35v_6^{(k)}(\mathbb{T}_6)_{\mathrm{III}},
\end{equation}
from which immediately
\begin{equation}
    v_6^{(1)}(T_6) = \tfrac{26609}{217818720}-\tfrac{3396146609 \pi ^2}{621871356506400}+\tfrac{1318349152898 \pi ^4}{12180206401298390455} \approx 0.00007880487647920397.
\end{equation}
We have not attempted to derive the higher moments. We leave this for our readers and humbly add that this task will be extraordinarily difficult.

\newpage
\section{Final remarks}
We have seen that the metric moments $v_n^{(k)}(P_d)$ having $n=d$ can be computed for all odd $k$ via our canonical section integral method whereas for $n>1$ and $d=3, k=1$ we could use Efron's formula. A natural question arises: How can we compute $v_n^{(k)}(P_d)$ for odd $k>1$ and $n>d$? Or when $d\geq 4$? Another obvious question is to deduce the volumetric moments $v_d^{(k)}(T_d)$ for $d\geq 6$. When $d=7$, there are four section equivalent configurations $C=\mathrm{I},\mathrm{II},\mathrm{III},\mathrm{IV}$ in $\mathcal{C}(T_7)$. Evaluating the section integral $v_7^{(1)}(T_7)_\mathrm{IV}$ for the fourth configuration is beyond the capabilities of our computer. At least, since $\bm{\sigma} \cap T_d$ is always a $T_{d-1}$ simplex in the first configuration of $T_d$, that is $n_\mathrm{I} = d$ with $w_\mathrm{I} = d+1$. By Theorem \ref{Thm:Canon},
\begin{equation}
 v_d^{(k)}(T_d)_\mathrm{I} = v_{d-1}^{(k+1)}(T_{d-1}) \frac{(d-1)!}{ d^k} \int_{\mathbb{R}^d\setminus K_d^\circ} \zeta_d^{k+d+1}(\bm{\sigma}) \iota^{(k)}_d(\bm{\sigma}) \lambda_d(\dd \bm{\eta})
\end{equation}
since $v_{d-1}^{(k+1)}(T_{d-1})$ are constants. More specifically, for $k=1$ by using Reed's formula, we found the following surprising relation
\begin{equation}
 v_d^{(1)}(T_d)_\mathrm{I} = 2v_d^{(2)}(T_d) = \frac{2(d!)}{(d+1)^d(d+2)^d}.
\end{equation}
\end{remark}
Based on the result we have seen so far for $d$-simplices, we conjecture
\begin{equation}
v_{r+1}^{(k)}(T_{r+1})=\sum_{s=0}^{\lfloor r/2 \rfloor} p^{(k)}_{rs} \pi^{2s}
\end{equation}
for some rationals $p^{(k)}_{rs}$ and $r=0,1,2,3,\ldots$ Since $\mathcal{G}(T_d) = \mathcal{S}_{d+1}$ (any permutation of vertices is a valid symmetry), we have for the weights $o_C = \binom{d+1}{|C|}$.

\begin{acknowledgment}{Acknowledgments.}
The work presented here spans five years, from the end of 2020 to mid-2024, corresponding to my PhD programme, which I pursued in Prague under the supervision of Jan Rataj. I owe a great deal to his patience and invaluable guidance throughout this period. I deeply appreciate the freedom I was granted, which allowed me to focus fully on my research. During this time, there were several key achievements that marked my progress. I would like to thank Anna Gusakova and Zakhar Kabluchko for the fruitful discussions we had during my stay in Münster, which greatly contributed to the development of my work.
\end{acknowledgment}

\newpage
% BIBLIOGRAPHY - BIBTEX
%\bibliography{bibliography}

% BIBLIOGRAPHY - BIBLATEX
\printbibliography[heading=bibintoc]

\clearpage

\appendix
\section{Selected genealogies}\label{Apx:Gen}

\vspace{-0.5em}
Configurations $\mathcal{C}(P)$ derived from the empty configuration\index{configuration!empty} $\mathrm{N}$ (no points selected) by succesively adding an extra vertex ($\mathrm{I},\mathrm{II},\mathrm{III},$ etc.). Genealogic decomposition is used to decompose affine functionals $F(K)$ as $\sum_{C \in \mathcal{C}(P)} w_C F(K)_C$. Each configuration is characterised by selection $S$ of vertices (figures), by section equivalent weights $w_C$ and the number of vertices of $\bm{\sigma} \cap P$, which is the order $n_C$.

\begin{minipage}[b]{0.45\textwidth}
\begin{figure}[H]
    \centering     \includegraphics[width=0.5\textwidth]{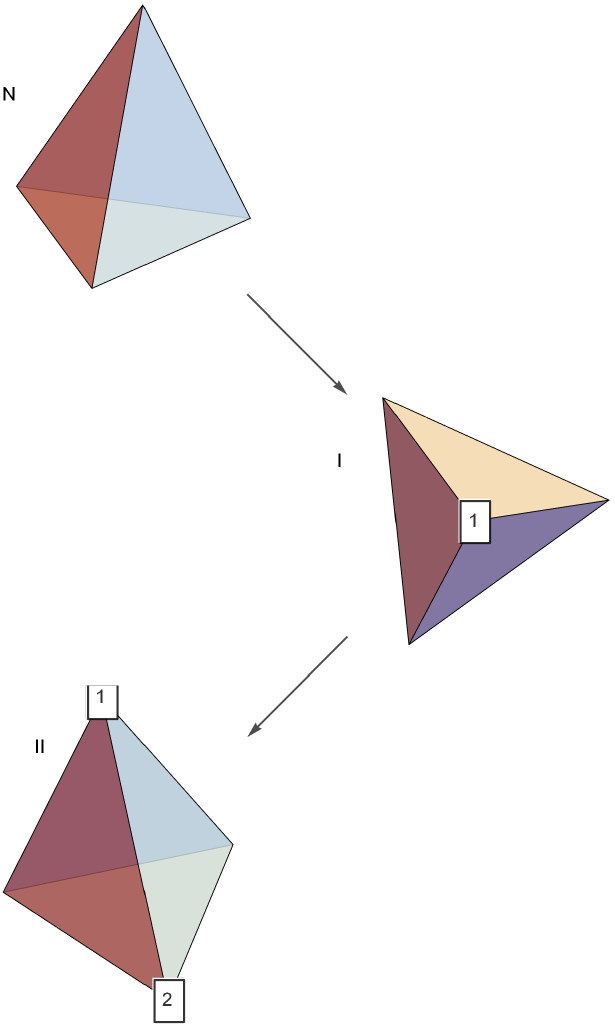}
    \vspace{-4em}
    \\
    \hspace{6em}\begin{tabular}{|c|c|c|}
    \hline
        $C$ & I & II \\
    \hline
        $w_C$ & 4 & 6\\
        $n_C$ & 3 & 4\\
    \hline
    \end{tabular}
    \caption{Tetrahedron genealogy}
    \label{fig:TETRAHEDRON_GENEALOGY}
\end{figure}
\end{minipage}
\begin{minipage}[b]{0.60\textwidth}
\begin{figure}[H]
    \centering     
    \vspace{-0.5em}\hspace{4em}\includegraphics[width=0.3\textwidth]{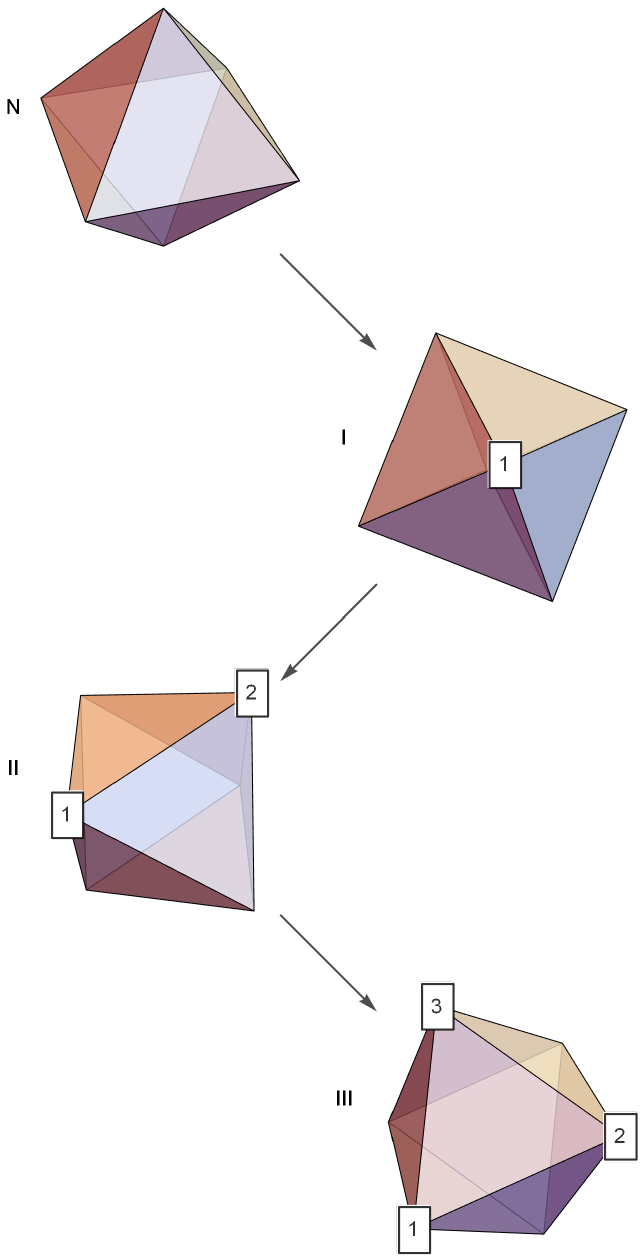}
    \vspace{-3em}
    \\
    \hspace{-5em}\begin{tabular}{|c|c|c|c|}
    \hline
        $C$ & I & II & III \\
    \hline
        $w_C$ & 6 & 12 & 4\\
        $n_C$ & 4 & 6 & 6\\
    \hline
    \end{tabular}    \caption{Octahedron genealogy}
    \label{fig:OCTAHE_GENEALOGY}
\end{figure}
\end{minipage}

\begin{minipage}[b]{0.45\textwidth}
\begin{figure}[H]
    \centering     \includegraphics[width=0.6\textwidth]{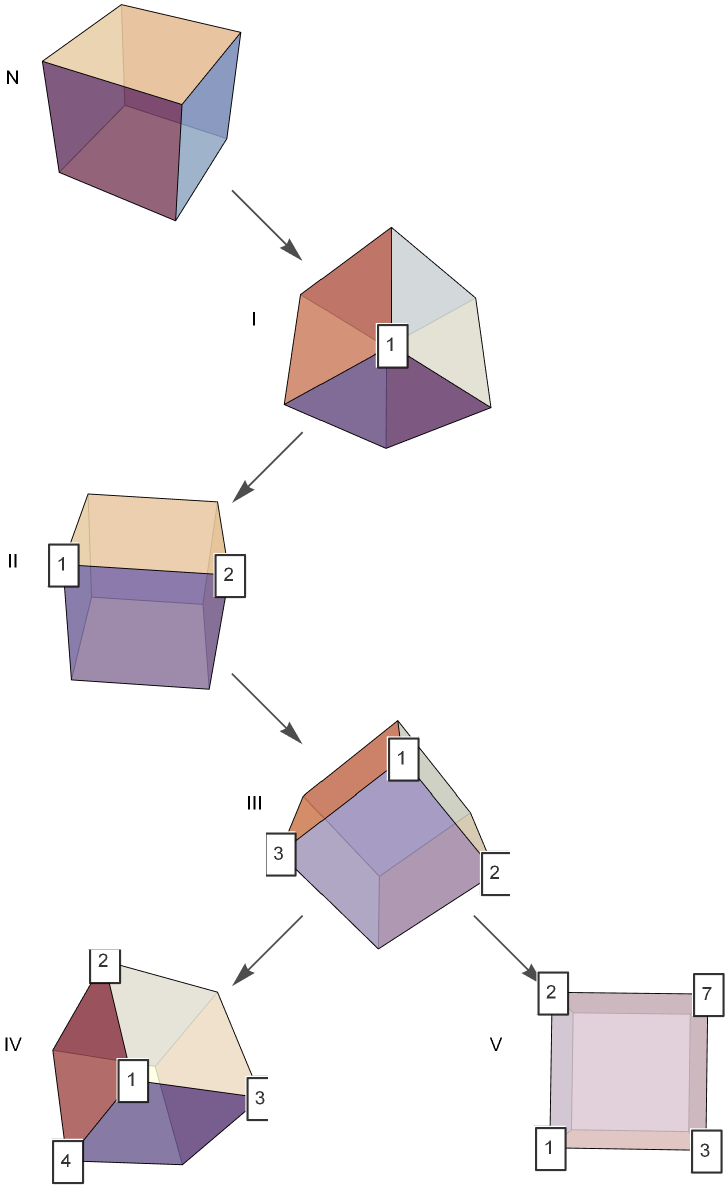}
    \vspace{1em}
    \\
    \begin{tabular}{|c|c|c|c|c|c|}
    \hline
        $C$ & I & II & III & IV & V \\
    \hline
        $w_C$ & 8 & 12 & 24 & 4 & 3 \\
        $n_C$ & 3 & 4 & 5 & 6 & 4 \\
    \hline
    \end{tabular}
    \caption{Cube genealogy}
    \label{fig:CUBE_GENEALOGY}
\end{figure}
\end{minipage}
\begin{minipage}[b]{0.55\textwidth}
\begin{figure}[H]
    \centering     \includegraphics[width=0.65\textwidth]{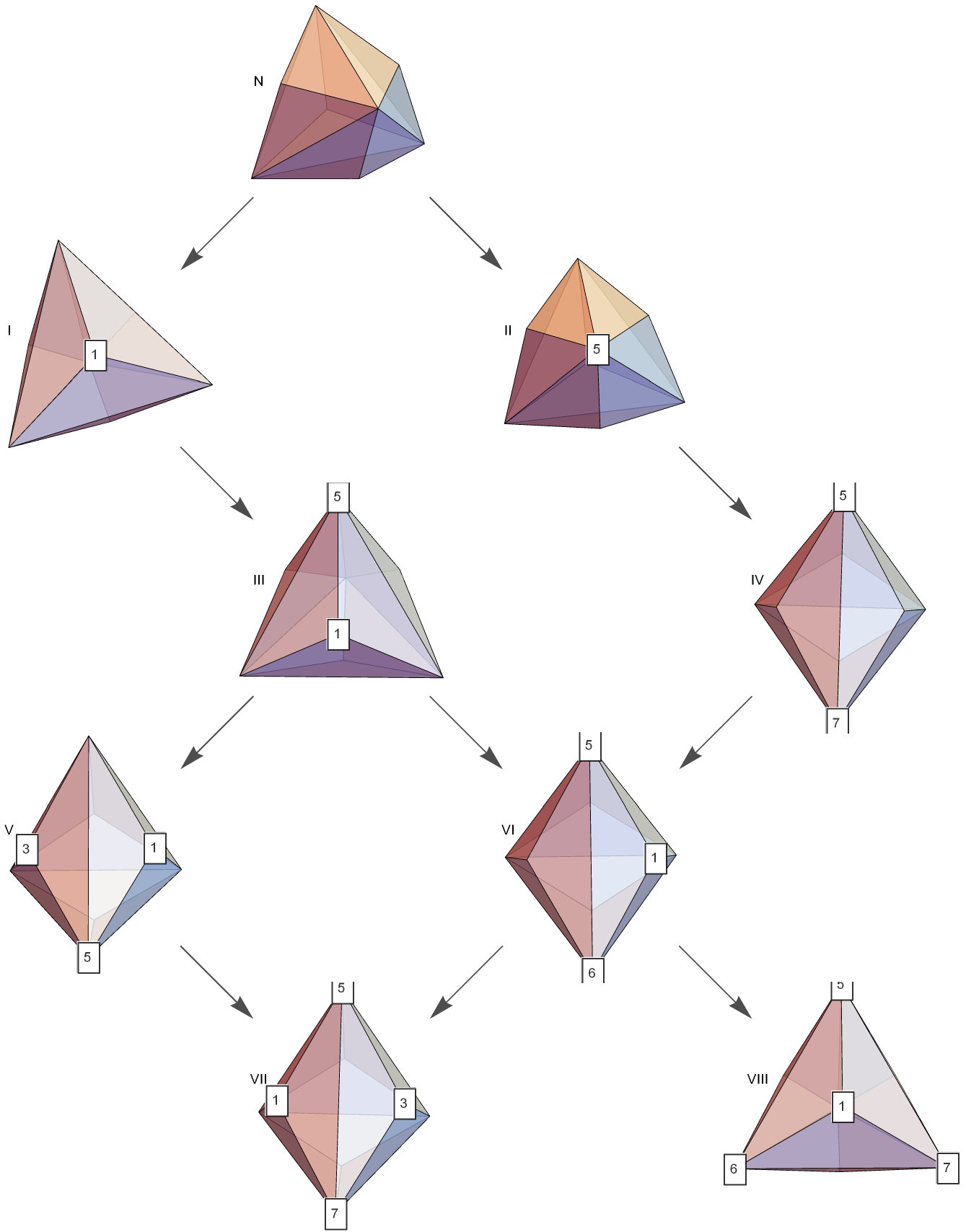}
    \vspace{0.7em}
    \\
    \begin{tabular}{|c|c|c|c|c|c|c|c|c|}
    \hline
        $C$ & I & II & III & \!IV\! & \!V\! & VI & \!VII\! & \!VIII\!\\
    \hline
        $ w_C$ & 4 & 4 & 12 & \!6\! & \!12\! & 12 & \!3\! & \!4\!\\
        $ n_C$ & 3 & 6 & 7 & \!10\! & \!8\! & 9 & \!8\! & \!9\!\\
    \hline
\end{tabular}
    \caption{Triakis tetrahedron genealogy}
    \label{fig:TRIAKIS_GENEALOGY}
\end{figure}
\end{minipage}

\begin{minipage}[b]{0.42\textwidth}
\vspace{-4em}
\begin{figure}[H]
    \centering     \includegraphics[width=0.92\textwidth]{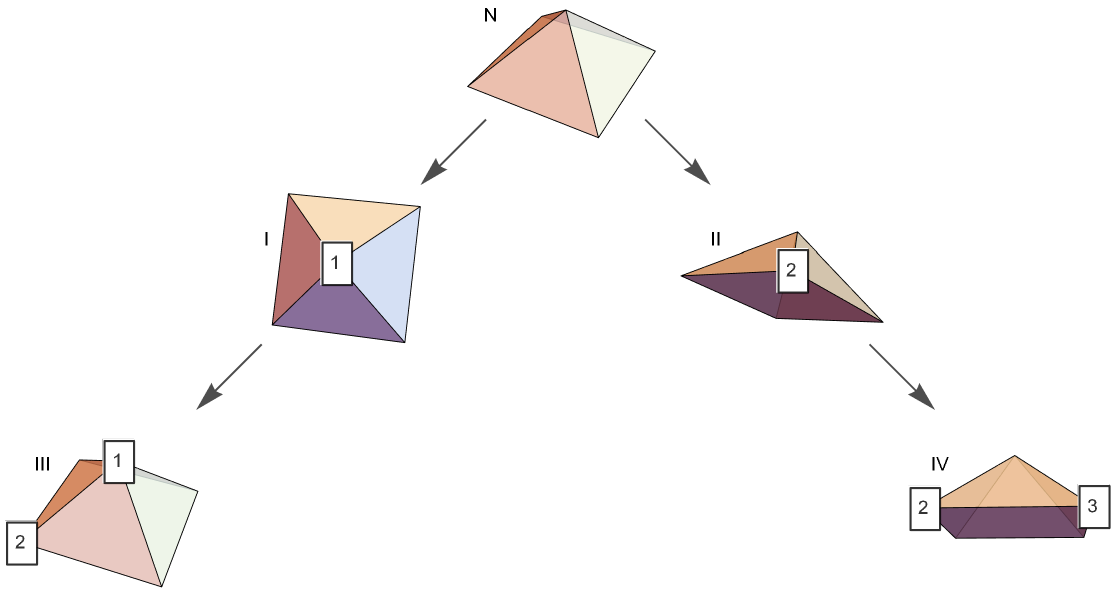}
    \\
    \hspace{6em}\begin{tabular}{|c|c|c|c|c|}
    \hline
        $C$ & I & II & III & IV\\
    \hline
        $w_C$ & 1 & 4 & 4 & 4\\
        $n_C$ & 4 & 3 & 5 & 4\\
    \hline
    \end{tabular}
    \caption{Square pyramid genealogy}
    \label{fig:PYRAMID_GENEALOGY}
\end{figure}
\vspace{-3em}
\begin{figure}[H]
    \centering     \includegraphics[width=0.92\textwidth]{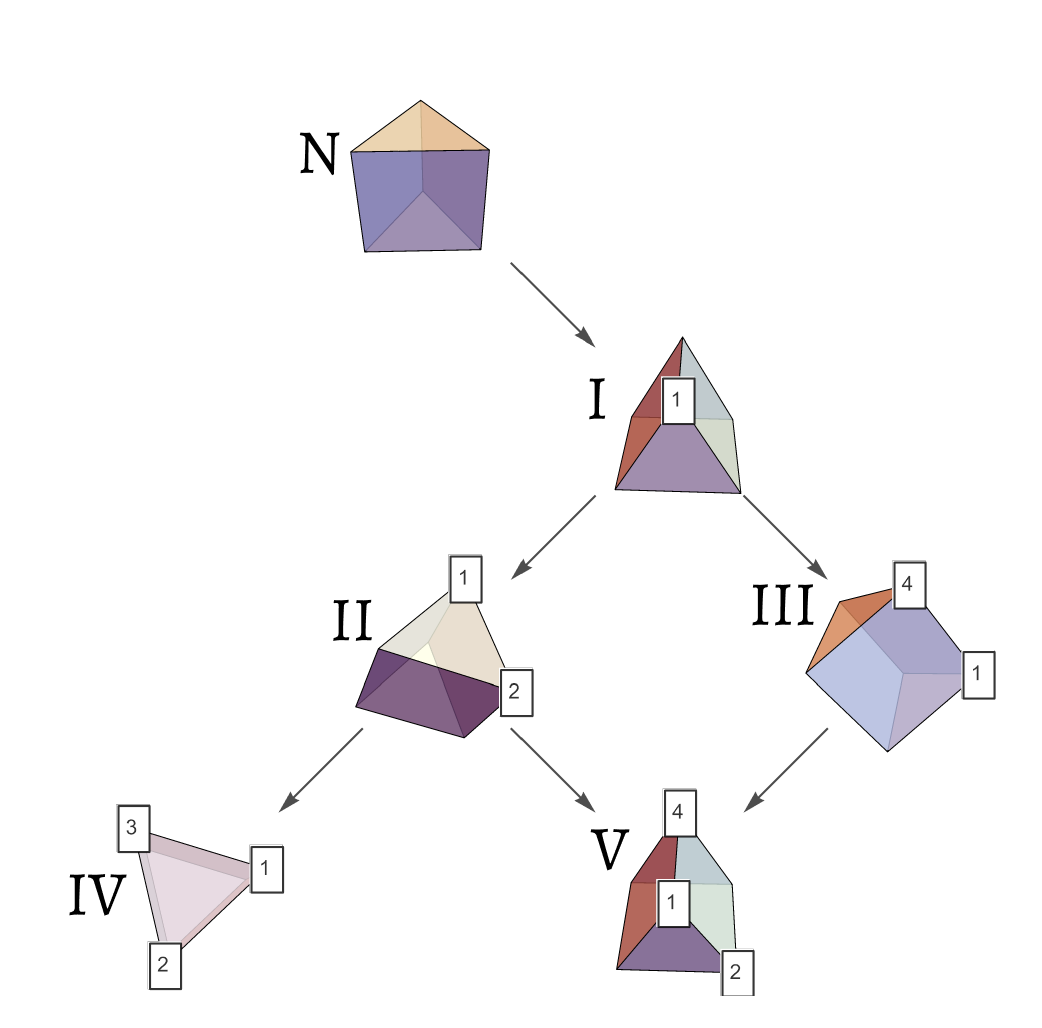}
    \vspace{0em}
    \\
    \begin{tabular}{|c|c|c|c|c|c|}
    \hline
        $C$ & I & II & III & IV & V\\
    \hline
        $ w_C$ & 6 & 6 & 3 & 1 & 6\\
        $ n_C$ & 3 & 4 & 4 & 3 & 5\\
    \hline
\end{tabular}
    \caption{Triangular prism genealogy}
    \label{fig:TRIPRISM_GENEALOGY}
\end{figure}
\vspace{-2em}
\begin{figure}[H]
    \centering     
    \includegraphics[width=0.92\textwidth]{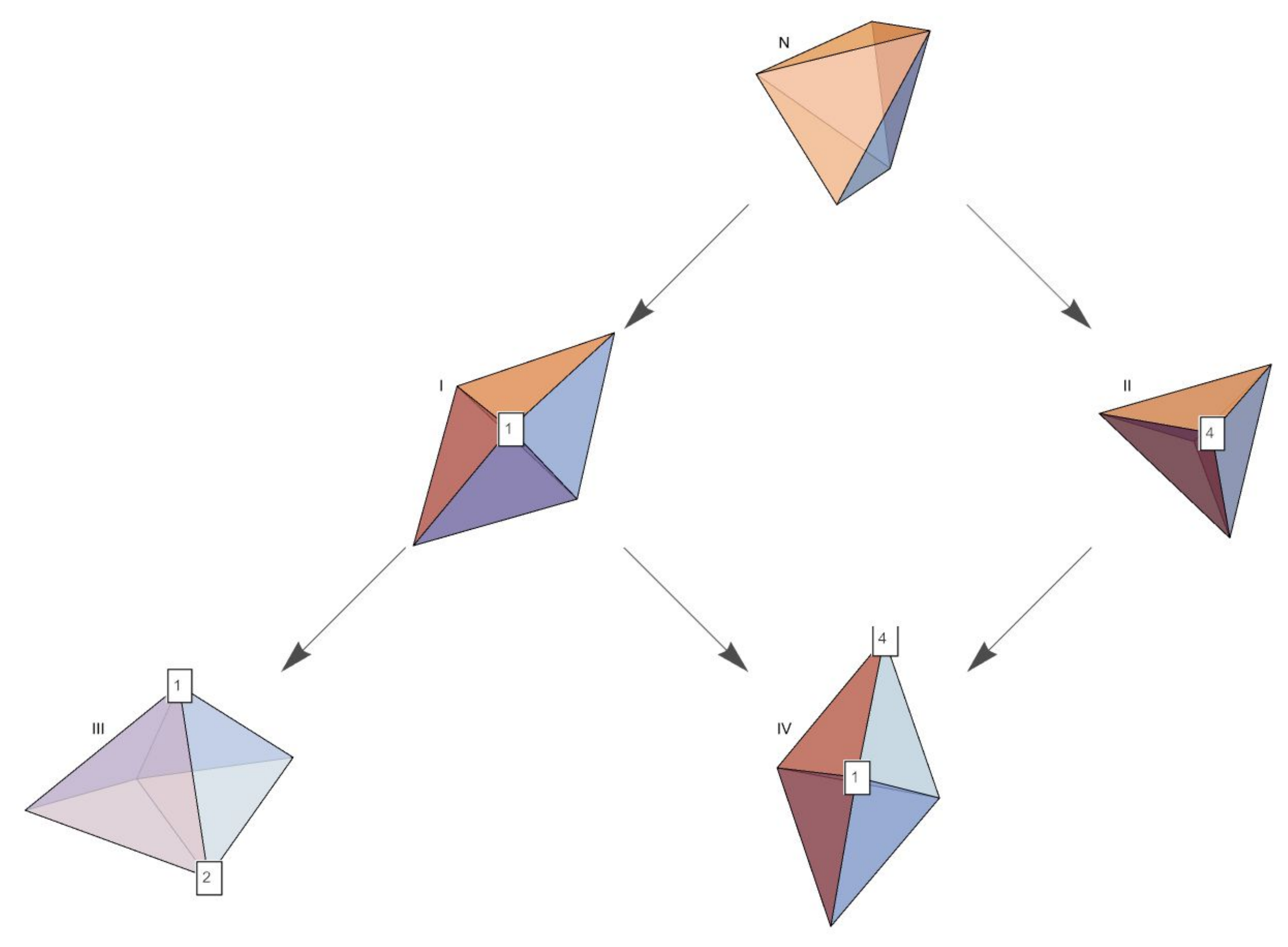}
    \vspace{0em}
    \\
    \begin{tabular}{|c|c|c|c|c|}
    \hline
        $C$ & I & II & III & IV\\
    \hline
        $w_C$ & 3 & 2 & 3 & 6\\
        $n_C$ & 4 & 3 & 6 & 5\\
    \hline
    \end{tabular}
    \caption{Triangular bipyramid genealogy}
    \label{fig:BIPYRAMID_GENEALOGY}
\end{figure}
\end{minipage}
\hfill
\begin{minipage}[b]{0.60\textwidth}
\vspace{-10em}
\begin{figure}[H]
    \centering
    \vspace{-1em}
    \includegraphics[width=0.60\textwidth]{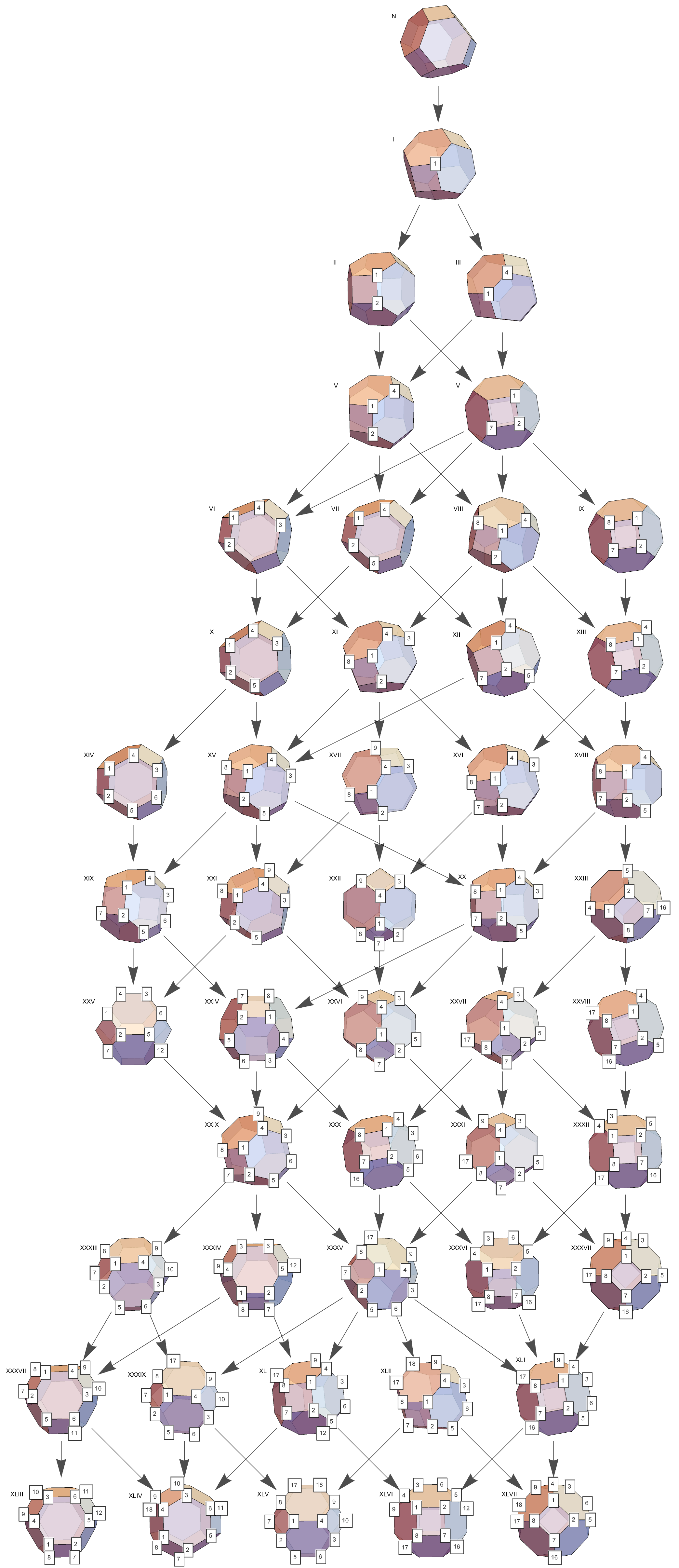}
    \vspace{0.7em}
    \\
    \begin{tabular}{|
      p{0.7em}|
      >{\centering}p{1.1em}|
      >{\centering}p{1.3em}|
      >{\centering}p{1.2em}|
      >{\centering}p{1.3em}|
      >{\centering}p{1.4em}|
      >{\centering}p{1.6em}|
      >{\centering}p{1.3em}|
      >{\centering}p{1.3em}|
      >{\centering}p{1.3em}|
      >{\centering\arraybackslash}p{1.1em}
      |}
    \hline
        \tiny $C$ & \!\!\tiny I\!\! & \!\!\tiny II\!\! & \!\!\tiny III\!\! & \!\!\tiny IV\!\! & \!\!\tiny V\!\! & \!\!\tiny VI\!\! & \!\!\tiny VII\!\! & \!\!\tiny VIII\!\! &  \!\!\tiny IX\!\! & \!\!\tiny X\!\!\\
    \hline
        \!\small $w_C$ & 24 & 24 & 12 & 48 & 24 & 24 & 24 & 24 & 6 & 48\\[-0.5ex]
        \!\small $n_C$ & 3 & 4 & 4 & 5 & 5 & 6 & 6 & 6 & 4 & 7\\
    \hline
        \tiny $C$ & \!\!\tiny XI\!\! & \!\!\tiny XII\!\! & \!\!\tiny XIII\!\! & \!\!\tiny XIV\!\! & \!\!\tiny XV\!\! & \!\!\tiny XVI\!\! & \!\!\tiny XVII\!\! & \!\!\tiny XVIII\!\! & \!\!\tiny XIX\!\! & \!\!\tiny XX\!\!\\
    \hline
        \!\small $w_C$ & 48 & 48 & 24 & 8 & 48 & 48 & 12 & 24 & 48 & 48\\[-0.5ex]
        \!\small $n_C$ & 7 & 7 & 5 & 6 & 8 & 6 & 8 & 6 & 7 & 7 \\
    \hline
        \tiny $C$ & \!\!\tiny XXI\!\! & \!\!\tiny XXII\!\! & \!\!\tiny XXIII\!\! & \!\!\tiny XXIV\!\! & \!\!\tiny XXV\!\! & \!\!\tiny XXVI\!\! & \!\!\tiny XXVII\!\! & \!\!\tiny XXVIII\!\! & \!\!\tiny XXIX\!\! & \!\!\tiny XXX\!\!\\
    \hline
        \!\small $w_C$ & 48 & 24 & 24 & 24 & 24 & 48 & 48 & 6 & 48 & 48\\[-0.5ex]
        \!\small $n_C$ & 9 & 7 & 7 & 6 & 8 & 8 & 8 & 8 & 7 & 7\\
    \hline
        \tiny $C$ &
   \!\!\tiny XXXI\!\! & \!\!\tiny XXXII\!\! & \!\!\!\tiny XXXIII\!\! & \!\!\tiny XXXIV\!\! & \!\!\tiny XXXV\!\! & \!\!\tiny XXXVI\!\! & \!\!\!\tiny XXXVII\!\! & \!\!\!\tiny XXXVIII\!\! & \!\!\tiny XXXIX\!\! & \!\!\tiny XL\!\!\\
    \hline
        \!\small $w_C$ & 24 & 48 & 24 & 24 & 48 & 24 & 24 & 48 & 48 & 48\\[-0.5ex]
        \!\small $n_C$ & 9 & 9 & 6 & 8 & 8 & 8 & 10 & 7 & 7 & 9\\
    \hline
        \tiny $C$ & \!\!\tiny XLI\!\! & \!\!\tiny XLII\!\! & \!\!\tiny XLIII\!\! &
   \!\!\tiny XLIV\!\! & \!\!\tiny XLV\!\! & \!\!\tiny XLVI\!\! & \!\!\tiny XLVII\!\! & & & \\
    \hline
        \!\small $w_C$ & 48 & 24 & 4 & 24 & 6 & 12 & 12 & & & \\[-0.5ex]
        \!\small $n_C$ & 9 & 7 & 6 & 8 & 6 & 10 & 8 & & & \\
    \hline
    \end{tabular}
    \caption{Truncated octahedron genealogy}
    \label{fig:OCTRUN_GENEALOGY}
\end{figure}
\end{minipage}

\newpage
\begin{minipage}[b]{0.55\textwidth}
\vspace{-4em}
\begin{figure}[H]
    \centering     \includegraphics[width=0.55\textwidth]{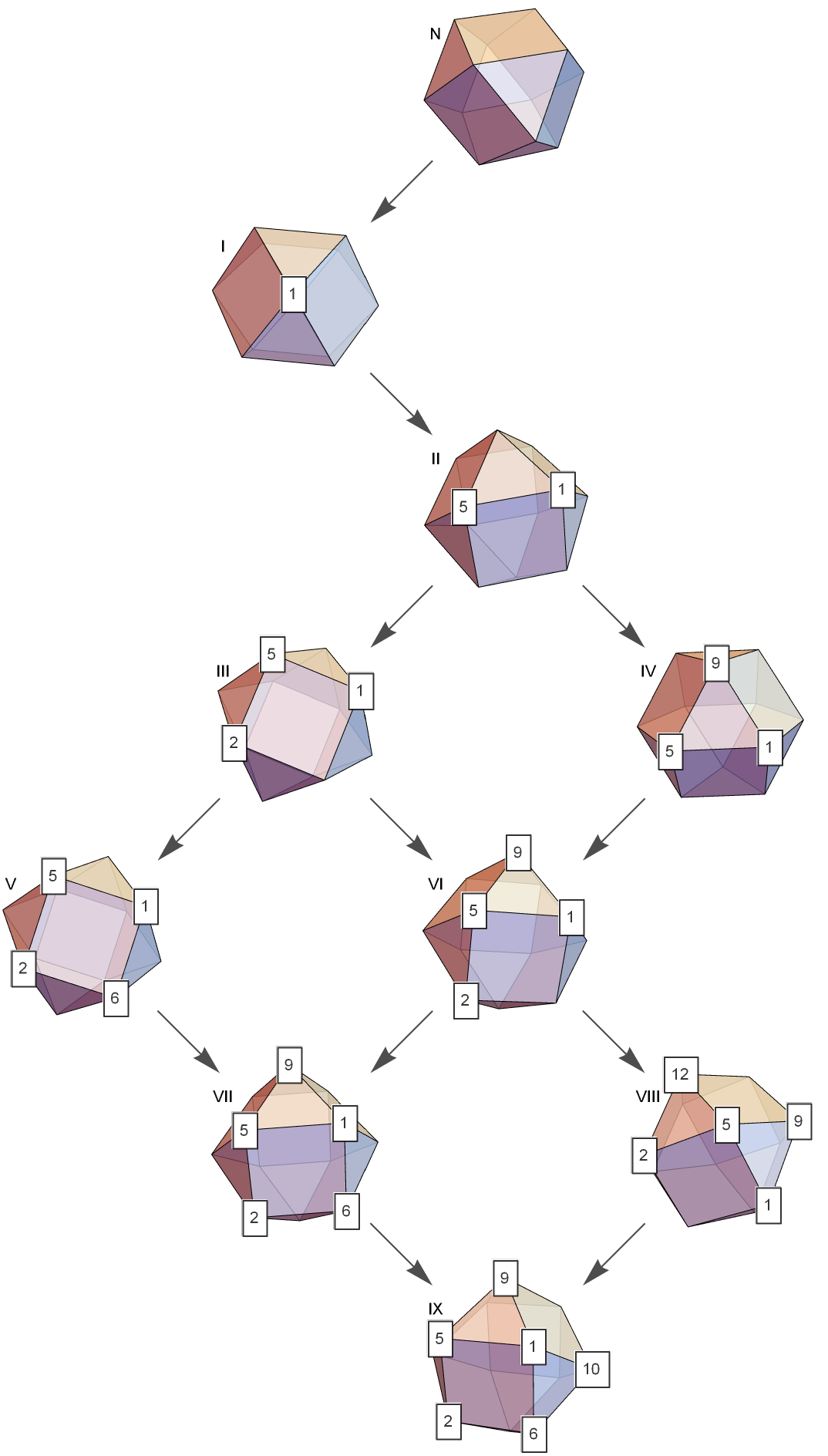}
    \vspace{0.7em}
    \\
    \begin{tabular}{|
      p{0.7em}|
      >{\centering}p{0.7em}|
      >{\centering}p{0.6em}|
      >{\centering}p{0.7em}|
      >{\centering}p{0.6em}|
      >{\centering}p{0.4em}|
      >{\centering}p{0.7em}|
      >{\centering}p{0.7em}|
      >{\centering}p{1.1em}|
      >{\centering\arraybackslash}p{0.7em}
      |}
    \hline
        \small $C$ & \!\small I & \!\small II & \!\small III & \!\small IV & \!\small V & \!\small VI & \!\!\small VII & \!\!\small VIII & \!\small IX\\
    \hline
        $\! w_C$ & \!12 & \!24 & \!24 & 8 & 6 & \!48 & \!24 & 12 & \!12\\
        $\! n_C$ & \!4 & \!6 & \!8 & 6 & 8 & \!8 & \!8 & 8 & \!8\\
    \hline
\end{tabular}
    \caption{Cuboctahedron genealogy}
    \label{fig:CUBOCTAHEDRON_GENEALOGY}
\end{figure}
\vspace{-2em}
\begin{figure}[H]
    \centering     \includegraphics[width=0.75\textwidth]{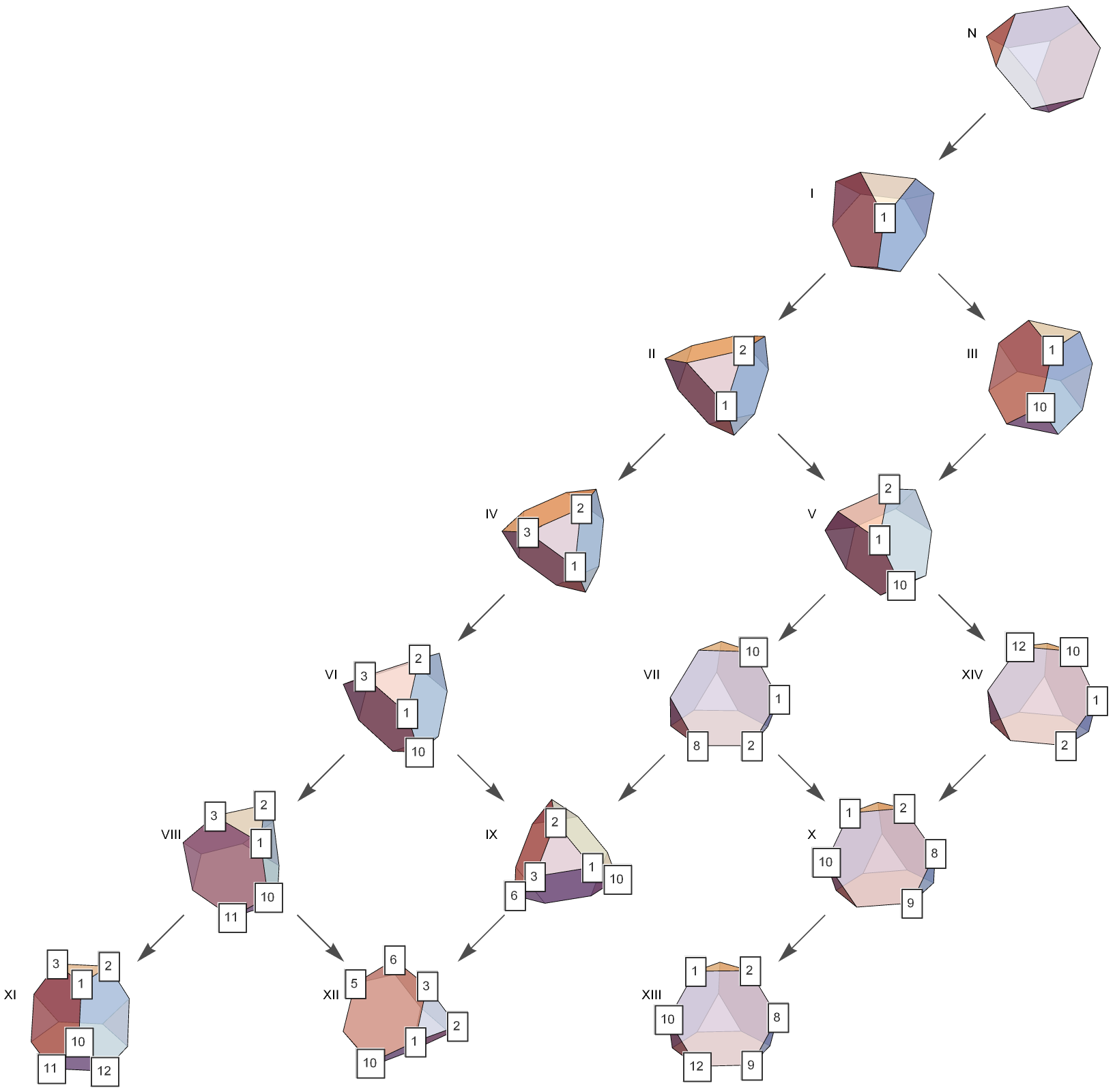}
    \vspace{0.7em}
    \\
    \begin{tabular}{|c|c|c|c|c|c|c|c|}
    \hline
        $C$ & I & II & III & IV & V & VI & VII\\
    \hline
        $ w_C$ & 12 & 12 & 6 & 4 & 24 & 12 & 12\\
        $ n_C$ & 3 & 4 & 4 & 3 & 5 & 4 & 6\\
    \hline
    \hline
        $C$ & I & II & III & IV & V & VI & VII\\
    \hline
        $ w_C$ & 24 & 12 & 24 & 3 & 12 & 4 & 12 \\
        $ n_C$ & 5 & 5 & 7 & 4 & 6 & 6 & 6\\
    \hline\end{tabular}
    \caption{Tuncated tetrahedron genealogy}
    \label{fig:TETRUN_GENEALOGY}
\end{figure}

\end{minipage}
\hspace{3em}
\begin{minipage}[b]{0.35\textwidth}

\begin{figure}[H]
   \hspace{-2.3em}\includegraphics[width=1.25\textwidth]{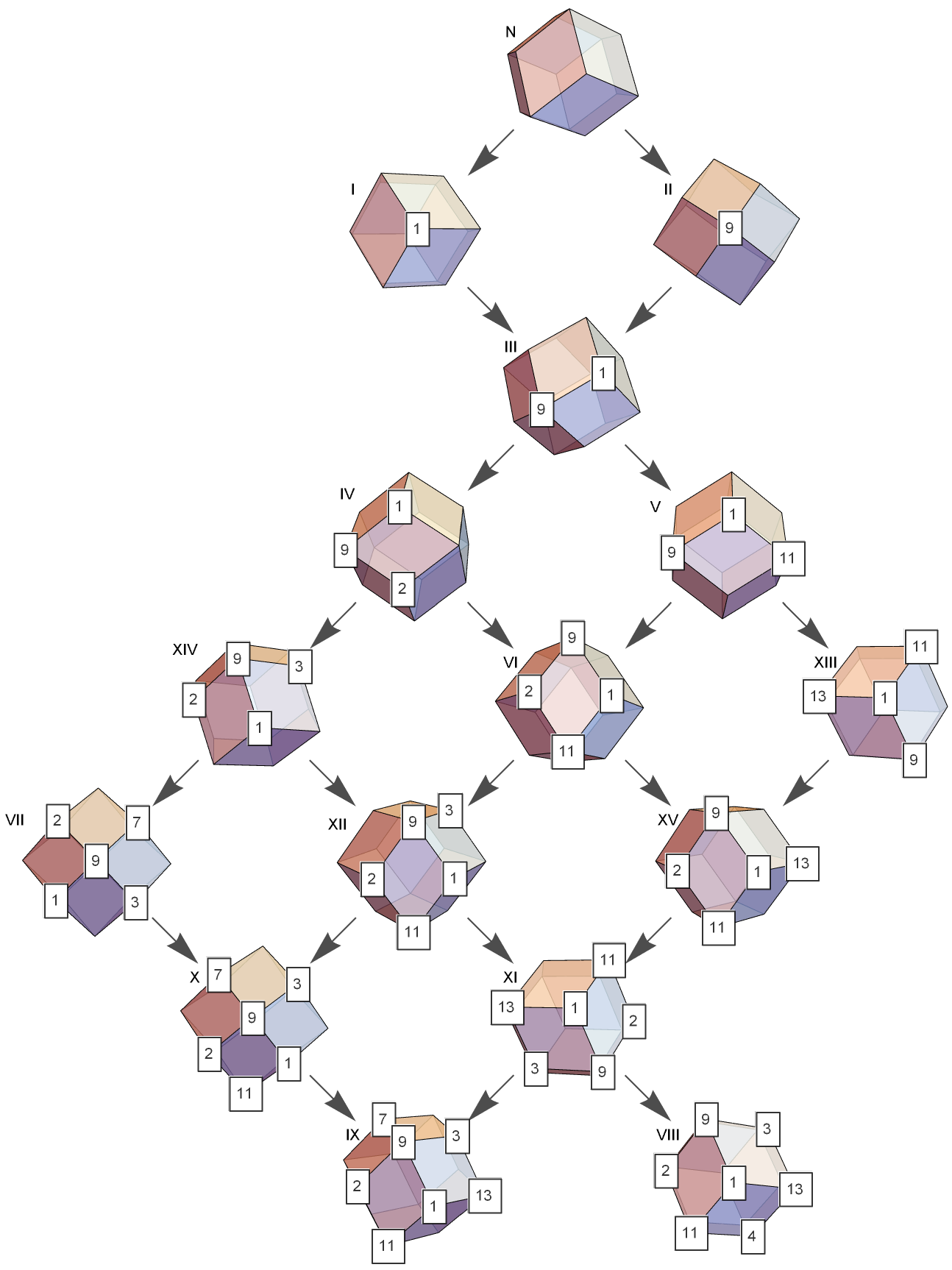}
    \vspace{0em}
    \\
    \begin{tabular}{|c|c|c|c|c|c|}
    \hline
        $C$ & I & II & III & IV & V \\
    \hline
        $w_C$ & 8 & 6 & 24 & 24 & 24 \\
        $n_C$ & 3 & 4 & 5 & 6 & 7 \\
    \hline
    \hline
        $C$ & V & VII & VIII & IX & X \\
    \hline
        $w_C$ & 12 & 6 & 4 & 12 & 24 \\
        $n_C$ & 6 & 8 & 6 & 8 & 8 \\
    \hline
    \hline
        $C$ & XI & XII & XIII & XIV & XV \\
    \hline
        $w_C$ & 24 & 48 & 8 & 24 & 24 \\
        $n_C$ & 7 & 7 & 9 & 7 & 8 \\
    \hline
    \end{tabular}
    \caption{Rhombic dodecahedron genealogy}
    \label{fig:RHOMBIC_GENEALOGY}
\end{figure}

\end{minipage}

\FloatBarrier
\section{Codes in Mathematica}
\setcounter{table}{0}

\subsection{General formulae}
\begin{lstlisting}[language=Mathematica,caption={Code to evaluate $e_d^{(k)}$ for general $d$ and $k$},label=code:efun]
efun[d_,1]:= (d+1)*(a @@ ConstantArray[0,d])^(d+1);
efun[d_,k_] := Simplify[(d+1)!/(d!)^k
  Sum[Times @@ Array[Signature[p[#+1]] &, k-1]*
        (Times @@ (a @@ Table[Count[#,i], {i,1,d}] &
            /@ Table[Flatten[{i-1, Table[p[j][[i]]-1, {j,2,k}]}], {i,1,d+1}])), ##] &
    @@ Table[{p[i], Permutations[Range[d+1]]}, {i,2,k}]];
\end{lstlisting}

\subsection{Tetrahedron area moments}
\begin{lstlisting}[language=Mathematica,caption={Code to evaluate $\iota_3^{(k)}(\bm{\sigma})$ in $\mathbb{T}_3$, configuration I},label=code:iota]
k = 1 (*desirable integer*);
Tcanon = Tetrahedron[{{0,0,0},{1,0,0},{0,1,0},{0,0,1}}];
Tabc = Tetrahedron[{{0,0,0},{1/a,0,0},{0,1/b,0},{0,0,1/c}}];
iotaint = Simplify[Integrate[(Dot[{a,b,c},x] - 1)^k, x \[Element] Tcanon] - (1 - (-1)^k) Integrate[(Dot[{a,b,c},x] - 1)^k, x \[Element] Tabc],Assumptions -> 1 < a && 1 < b && 1 < c]
\end{lstlisting}

\begin{lstlisting}[language=Mathematica,caption={Code to evaluate $\iota_3^{(k)}(\bm{\sigma})$ in $\mathbb{T}_3$, configuration II},label=code:iota2]
k = 1(*desirable integer*);
Tcanon = Tetrahedron[{{0,0,0},{1,0,0},{0,1,0},{0,0,1}}];
Tabc = Tetrahedron[{{0,0,0},{1/a,0,0},{0,1/b,0},{0,0,1/c}}];
Tstar = Tetrahedron[{{0,0,1},{(1 - c)/(a - c),0,(a - 1)/(a - c)},{0,(1 - c)/(b - c),(b - 1)/(b - c)},{0,0,1/c}}];
iotaint = Simplify[(Integrate[(Dot[{a,b,c},x] - 1)^k, x \[Element]
    Tcanon] - (1 - (-1)^k) Integrate[(Dot[{a,b,c},x] - 1)^k, 
    x \[Element] Tabc] + (1 - (-1)^k) Integrate[(Dot[{a,b,c},x]
    - 1)^k, x \[Element] Tstar]),
  Assumptions -> 1 < a && 1 < b && 0 < c < 1]
\end{lstlisting}

\begin{lstlisting}[language=Mathematica,caption={Code to evaluate $v_2^{(k+1)}(\mathbb{U}_2^{\alpha\beta})$, odd $k$},label=code:vII]
k = 1(*desirable odd integer*);
Delta = 1/2! Det[{x1 - x0, x2 - x0}];
trianab = Triangle[{{0, 0}, {\[Alpha], 0}, {0, \[Beta]}}];
trianunit = Triangle[{{0, 0}, {1, 0}, {0, 1}}];
meancut = 
 Simplify[(2/(1 - \[Alpha] \[Beta]))^(
   k + 4) (Integrate[Delta^(k + 1), x0 \[Element] trianunit, 
      x1 \[Element] trianunit, x2 \[Element] trianunit] - 
     3 Integrate[Delta^(k + 1), x0 \[Element] trianunit, 
       x1 \[Element] trianunit, x2 \[Element] trianab] + 
     3 Integrate[Delta^(k + 1), x0 \[Element] trianunit, 
       x1 \[Element] trianab, x2 \[Element] trianab] - 
     Integrate[Delta^(k + 1), x0 \[Element] trianab, 
      x1 \[Element] trianab, x2 \[Element] trianab]), 
  Assumptions -> 0 < \[Alpha] < 1 && 0 < \[Beta] < 1]
\end{lstlisting}

\subsection{Pentachoron 4-volume moments}
\begin{lstlisting}[language=Mathematica,caption={Code to evaluate $\iota_4^{(k)}(\bm{\sigma})$ in $\mathbb{T}_4$, configuration I},label=code:4diota]
k = 1 (*desirable integer*);
Tcanon = 
  Simplex[{{0, 0, 0, 0}, {1, 0, 0, 0}, {0, 1, 0, 0}, {0, 0, 1, 0}, {0, 0, 0, 1}}];
Tabcd = 
  Simplex[{{0, 0, 0, 0}, {1/a, 0, 0, 0}, {0, 1/b, 0, 0}, {0, 0, 1/c, 0}, {0, 0, 0, 1/d}}];
iotaint = 
 Simplify[
  Integrate[(Dot[{a, b, c, d}, x] - 1)^k, x \[Element] 
     Tcanon] - (1 - (-1)^k) Integrate[(Dot[{a, b, c, d}, x] - 1)^
     k, x \[Element] Tabcd], 
  Assumptions -> 1 < a && 1 < b && 1 < c && 1 < d]
\end{lstlisting}

\begin{lstlisting}[language=Mathematica,caption={Code to evaluate $\iota_4^{(k)}(\bm{\sigma})$ in $\mathbb{T}_4$, configuration II},label=code:4diota2]
k = 1(*desirable integer*);
Tcanon = Simplex[{{0, 0, 0, 0}, {1, 0, 0, 0},
    {0, 1, 0, 0}, {0, 0, 1, 0}, {0, 0, 0, 1}}];
Tabcd = Simplex[{{0, 0, 0, 0}, {1/a, 0, 0, 0},
    {0, 1/b, 0, 0}, {0, 0, 1/c, 0}, {0, 0, 0, 1/d}}];
Tstar = Simplex[{{0, 0, 0, 1},
    {(1 - d)/(a - d), 0, 0, (a - 1)/(a - d)},
    {0, (1 - d)/(b - d), 0, (b - 1)/(b - d)},
    {0, 0, (1 - d)/(c - d), (c - 1)/(c - d)},
    {0, 0, 0, 1/d}}];
iotaint = Simplify[(Integrate[(Dot[{a, b, c, d}, x] - 1)^k,
    x \[Element] Tcanon] - 2 Integrate[(Dot[{a,b,c,d},x] - 1)^k,
    x \[Element] Tabcd] + 2 Integrate[(Dot[{a,b,c,d},x] - 1)^k, x
    \[Element] Tstar]),
   Assumptions -> 1 < a && 1 < b && 1 < c && 0 < d < 1]
\end{lstlisting}

\begin{lstlisting}[language=Mathematica,caption={Code to evaluate $v_3^{(k+1)}(\mathbb{U}_3^{\alpha\beta\gamma})$, odd $k$},label=code:4dvII]
k = 1(*desirable odd integer*);
Delta = 1/3! Det[{x1 - x0, x2 - x0, x3 - x0}];
Tabc = Tetrahedron[{{0, 0, 0}, {\[Alpha], 0, 0}, {0, \[Beta], 0}, {0, 0, \[Gamma]}}];
Tcan = Tetrahedron[{{0, 0, 0}, {1, 0, 0}, {0, 1, 0}, {0, 0, 1}}];
meancut = 
 Simplify[(6/(1 - \[Alpha] \[Beta] \[Gamma]))^(
   k + 5) (Integrate[Delta^(k + 1), x0 \[Element] Tcan, 
    x1 \[Element] Tcan, x2 \[Element] Tcan,x3 \[Element] Tcan]
    - 4 Integrate[Delta^(k + 1), x0 \[Element] Tcan, 
    x1 \[Element] Tcan, x2 \[Element] Tcan, x3 \[Element] Tabc]
    + 6 Integrate[Delta^(k + 1), x0 \[Element] Tcan, 
    x1 \[Element] Tcan, x2 \[Element] Tabc, x3 \[Element] Tabc]
    - 4 Integrate[Delta^(k + 1), x0 \[Element] Tcan, 
    x1 \[Element] Tabc, x2 \[Element] Tabc, x3 \[Element] Tabc]
    + Integrate[Delta^(k + 1), x0 \[Element] Tabc, x1 \[Element]
    Tabc, x2 \[Element] Tabc, x3 \[Element] Tabc]), 
  Assumptions -> 
   0 < \[Alpha] < 1 && 0 < \[Beta] < 1 && 0 < \[Gamma] < 1]
\end{lstlisting}

\subsection{GECRA: Genealogy creation algorithm}
The following algorithm generates realisable configurations and their weights for any polytopes by exploiting their symmetries. The code works on iterating over \texttt{nos} (the number of selected vertices) and it has the following steps
\vspace{1em}
\begin{itemize}
    \item Step 0: initialize empty configuration $\mathrm{N}$
    \item CYCLE
    \begin{itemize}
        \item Step I: generate new configurations from old ones
        \item Step II: group them into classes, select first configuration from each (the so called representant)
        \item Step III: for each representant, determine if it is realisable, discard unrealisable
    \end{itemize}
    \item repeat step I until \texttt{nos} reaches half the number of vertices,
\end{itemize}
\vspace{1em}
The algorithm is initialised by inserting vertices of $P_d$ into \texttt{solid} as a list of their coordinates and the symmetry group $\mathcal{G}(P_d)$ into \texttt{symgroup} as list of permutations on indexes of these vertices. In the code, \texttt{dimen} is the dimension $d$. For example, Code \ref{code:GECRAinput} shows the input for $P_d = C_3$ (the three-dimensional unit cube). Note that we only store the generators of $\mathcal{G}(C_3)$ since the whole symmetric group can be obtain by successive composition of the elements with themselves.
\begin{lstlisting}[language=Mathematica,caption={Input for GECRA for $P_d = C_3$},label=code:GECRAinput]
solid = {{0, 0, 0}, {1, 0, 0}, {0, 1, 0}, {0, 0, 1},
    {0, 1, 1}, {1, 0, 1}, {1, 1, 0}, {1, 1, 1}};
generators = {(*reflection*){4, 6, 5, 1, 3, 2, 8, 7},
 (*2fold rotation*){3, 7, 5, 1, 4, 2, 8, 6},
 (*3fold rotation*){7, 8, 3, 2, 1, 6, 5, 4}};
symgroup = FixedPoint[Union[Flatten[Table[PermutationProduct[#, p] & /@ #, {p, #}], 1]] &, generators];
\end{lstlisting}
The output of the GECRA program is the following
\begin{itemize}
    \item \texttt{alltypes}:the list of cofigurations, each configuration is represented by a list of indices of vertices
    \item \texttt{allweights}: list of weights of configurations
    \item \texttt{allgenealogy}: the genealogy as a list of pairs $i \to j$, where $i,j$ are indices of configurations in the list of configurations
    \item \texttt{gengraph}: the genealogy graph (a Hasse diagram)
\end{itemize}

\begin{lstlisting}[language=Mathematica,caption={GECRA: Genealogies from symmetry groups},label=code:GECRA]
Clear[classreps, rawclassreps, oineqsel, isrealisable, orbitmaker, 
  allsuccesors, weightsel, dimen];
dimen = 4;
ofvertices=Length[solid];
etaparams = Table[a[i], {i, dimen}];
orbitmaker[sel_] := orbitmaker[sel] = 
   Union[Table[(sel)[[#[[i]]]], {i, ofvertices}] & /@ symgroup];
weightsel[sel_] := 
  If[Total[sel] < ofvertices/2, Length[orbitmaker[sel]], 
   Length[orbitmaker[sel]]/2];
(*inequalities for etaparams a,b,c,d,... for a given 0,1 selection of \vertices*)
oineqsel[sel_] := oineqsel[sel] = 
   With[{representant = Pick[solid, # == 1 & /@ sel]}, 
    Reduce[Or[
      And @@ Flatten[{Dot[etaparams, #] > 1 & /@ representant, 
         Dot[etaparams, #] < 1 & /@ Complement[solid, representant]}],
      And @@ Flatten[{Dot[etaparams, #] < 1 & /@ representant, 
         Dot[etaparams, #] > 1 & /@ 
          Complement[solid, representant]}]]]];
isrealisable[sel_] := 
  isrealisable[sel] = If[Length[oineqsel[sel]] == 0, 0, 1];
(*step 0*)
classreps[0] = {ConstantArray[0, ofvertices]};
(*step I*)
allsuccesors[sel_] := 
  allsuccesors[sel] = ReplaceList[sel, {a___, 0, b___} :> {a, 1, b}];
(*step II*)
rawclassreps[i_] := 
  rawclassreps[i] = 
   Map[Last, 
    Union[orbitmaker[#] & /@ 
      Union[Flatten[allsuccesors[#] & /@ classreps[i - 1], 
        1]]]];
(*step III*)
classreps[nos_] := 
  classreps[nos] =(*sort by weight of a configuration*)
   SortBy[Select[rawclassreps[nos], isrealisable[#] == 1 &], weightsel];
    
(*OUTPUT*)
allreps = Flatten[Table[classreps[i], {i, 1, Floor[ofvertices/2]}], 1];
(*01 representants as their index*)
repstoindexesrule = 
  Flatten[{{ConstantArray[0, ofvertices] -> 0}, 
    Table[allreps[[i]] -> i, {i, Length[allreps]}]}, 1];
allgenealogy = 
  Flatten[Table[(sel /. repstoindexesrule) -> (suc /. 
       repstoindexesrule), {i, 0, Floor[ofvertices/2] - 1}, {sel, classreps[i]}, {suc, 
     Intersection[Flatten[(orbitmaker[#] & /@ allsuccesors[sel]), 1], classreps[i + 1]]}], 2];
gengraph = 
  GraphPlot[RomanNumeral[#[[1]]] -> RomanNumeral[#[[2]]] & /@ allgenealogy, VertexLabeling -> True, DirectedEdges -> True];
(*representants in coordinates*)allrepscoord =
  Pick[solid, # == 1 & /@ #] & /@ allreps;
alltypes = Pick[Range[ofvertices], # == 1 & /@ #] & /@ allreps;
allweights = weightsel[#] & /@ allreps;
allnoofsel = Total[#] & /@ allreps;
\end{lstlisting}

\end{document}